\documentclass[a4paper,12pt]{amsart}
\usepackage{amssymb,amsmath,amsthm}
\usepackage{kbordermatrix}  
\usepackage{color}
\usepackage{enumitem}
\newcommand\red{\color{red}}
\usepackage{hyperref}
\hypersetup{colorlinks=true,linkcolor=blue}
\newtheorem{property}[equation]{Property}
\newtheorem{conjecture}[equation]{Conjecture}
\newtheorem{corollary}[equation]{Corollary}
\newtheorem{definition}[equation]{Definition}
\newtheorem{definitions}[equation]{Definition}
\newtheorem{lemma}[equation]{Lemma}
\newtheorem{proposition}[equation]{Proposition}

\newtheorem{theorem}[equation]{Theorem}
\newtheorem{theodef}[equation]{Theorem--Definition}
\newtheorem{lemmadef}[equation]{Lemma--Definition}
\theoremstyle{remark}
\newtheorem{examples}[equation]{Examples}
\newtheorem{example}[equation]{Example}
  \newtheorem{remark}[equation]{Remark}
  \newtheorem{remarks}[equation]{Remarks}
  \newtheorem{notation}[equation]{Notation}
\numberwithin{equation}{section}
  \newcommand\BC{{\mathbb C}}
  \newcommand\BF{{\mathbb F}}
  \newcommand\BN{{\mathbb N}}
  \newcommand\BQ{{\mathbb Q}}
  \newcommand\BR{{\mathbb R}}
  \newcommand\BZ{{\mathbb Z}}
  \newcommand\fa{{\mathfrak a}}
  \newcommand\fb{{\mathfrak b}}
  \newcommand\fF{{\mathfrak F}}
  \newcommand\fI{{\mathfrak I}}
  \newcommand\fj{{\mathfrak j}}
  \newcommand\fn{{\mathfrak n}}
  \newcommand\fp{{\mathfrak p}}
  \newcommand\fq{{\mathfrak q}}
  \newcommand\frr{{\mathfrak r}}  
  \newcommand\fR{{\mathfrak R}}
  \newcommand\fS{{\mathfrak S}}
  \newcommand\CF{{\mathcal F}}
  \newcommand\CL{{\mathcal L}}
  \newcommand\CO{{\mathcal O}}
  \newcommand\CP{{\mathcal P}}
  \newcommand\CS{{\mathcal S}}
  \newcommand\al{\alpha}
  \newcommand\be{\beta}
  \newcommand\la{\lambda}
  \newcommand\vp{\varphi}
  \newcommand\si{\sigma}
  \newcommand\Si{\Sigma}
  \newcommand\bmu{{\boldsymbol\mu}}
  \newcommand\ra{\rightarrow}
  \newcommand\iso{{\,\overset{\sim}{\longrightarrow}\,}\,}
  \newcommand\lexp[2]{\kern\scriptspace\vphantom{#2}^{#1}\kern-\scriptspace#2}
  \newcommand\inv{^{-1}}
  \newcommand\tinv{^{-*}}
  \newcommand\scal[2]{{\langle{#1},{#2}\rangle}}
  \newcommand\genby[1]{{\mathopen\langle#1\mathclose\rangle}}
  \newcommand\ab{{\operatorname{ab}}}
  \newcommand\Arr{{\operatorname{Arr}}}
  \newcommand\Aut{{\operatorname{Aut}}}
  \newcommand\Bad{{\operatorname {Bad}}}
  \newcommand\Car{{\operatorname{Car}}}
  \newcommand\Ch{{\operatorname {Ch}}}
  \newcommand\diag{{\operatorname{diag}}}
  \newcommand\GL{{\operatorname{GL}}}
  \newcommand\Hom{{\operatorname{Hom}}}
  \newcommand\Id{{\operatorname{Id}}}
  \newcommand\im{{\operatorname{im}\,}}
  \newcommand\Rf{{\operatorname{Ref}}}
  \newcommand\bul{$\bullet\,\,\,$}
  \newcommand\apriori{{\it a priori\/}}
  \newcommand\ibidem{{\it ibidem\/}}
  \newcommand\ie{{\it i.e.\/},}
\newcommand\zroot{{$\BZ_k$-root}}
\newcommand\habel[1]{{\label{#1}\hfill}}
\newcommand\ideal[1]{#1 \BZ_k}          
\newcommand\frn[2]{\fn(#1,#2)}
\newcommand\pd{\mbox{-}}
\newcommand\unit[1]{{\mathbf{#1}}}
  \newcommand\CHEVIE{{\tt CHEVIE}}
  \newcommand\GAP{{\tt GAP3}}
\title{Cyclotomic root systems and bad primes}
\author{Michel Brou\'e, Ruth Corran, Jean Michel}
\date{\today}
\address{Michel Brou\'e, Universit\'e Paris-Diderot, Paris}
 \email{broue@math.univ-paris-diderot.fr}
 \urladdr{http://webusers.imj-prg.fr/~michel.broue/}
\address{Ruth Corran, The American University of Paris, Paris}
 \email{ruth.corran@aup.fr}
 \urladdr{https://www.aup.edu/profile/rcorran}
\address{Jean Michel, CNRS, Institut de Math\'ematiques de Jussieu, Paris}
 \email{jean.michel@imj-prg.fr}
 \urladdr{http://webusers.imj-prg.fr/~jean.michel//}
\thanks{We thank Gabriele Nebe for her previous seminal work on this topic}
\makeindex
\begin{document}
\begin{abstract}
We generalize the definition and properties of root systems to complex
reflection groups --- roots become rank one projective modules over the ring of
integers of a number field $k$.

In the irreducible case, we provide a classification of root systems
over the field of definition $k$ of the reflection representation.

In the case of spetsial reflection groups, we generalize as well the
definition and properties of bad primes.
\end{abstract}
\maketitle
\tableofcontents

\section{\color{red} Introduction}


        The spirit of the Spetses program (\cite{spe}, \cite{spe2})
         is to consider (at least some of) the complex
        reflections groups as Weyl groups for some mysterious object which
        looks like a ``generic
        finite reductive group'' --- and which is yet unknown.
        
        Some of the data attached to finite reductive groups, 
        such as the parameterization of unipotent characters, 
        their generic degrees, Frobenius eigenvalues,
        and also the families and their Fourier matrices, turn out to
        depend only on the $\BQ$-representation of the Weyl group.
        But supplementary data,
        such as the parameterization of unipotent classes, the values of unipotent characters on 
        unipotent elements, depend on the entire root datum. 
        
        For a complex reflection group,
        not necessarily defined over $\BQ$, but defined over a number field $k$, it 
        thus seems both necessary and
        natural to study ``\zroot\ data'', where $\BZ_k$ denotes the ring of integers of $k$.
        
        Some related work has already been done in that direction, in particular
        by Nebe \cite{nebe}, who classified $\BZ_k$-lattices invariant under the reflection
        group -- and whose work inspired us. 
        More recently, motivated by some work on $p$-compact groups,
        Grodal and others
        \cite{angr}, \cite{gr} defined root data for 
        reflection groups defined over finite fields.
        
        Here we define and classify \zroot\ 
        systems\footnote{By analogy with the well established terminology 
        \emph{``cyclotomic Hecke algebras''} --- a crucial notion in the Spetses program ---, 
        we propose to call
        these generalised root systems \emph{``cyclotomic root systems''}.}, as well as
        root lattices and coroot lattices, for all complex reflection groups. 
        Of course most of the rings $\BZ_k$ are not principal ideal domains (although
        --- quite a remarkable fact --- they are P.I.D. for the 34 exceptional irreducible complex
        reflection groups) and one has naturally to replace \textsl{elements of $\BZ_k$\/} by
        \textsl{ideals of $\BZ_k$\/}. Taking this into account, our definition of root system
        mimics Bourbaki's definition \cite{bou} and of course working with
        ``ideal numbers'' is much more appropriate, on general Dedekind domains, than
        working with ``numbers'', as suggested by what follows.
\begin{itemize}
        \item   
                A  complex reflection  group may  occur as  a parabolic subgroup of another
                reflection  group whose field of definition is  larger.  
                A first problem with considering vectors (as in the usual approach of root systems)
                instead of one-dimensional $\BZ_k$-modules (as we do here)
                is that we would have ``too many'' of them when restricting to a parabolic subgroup.
        \item
                In  the  case  of  the  group of type $B_2$, 
                over a field  where  the  ideal generated  by $2$ has a  square root, 
                such  as  $\BQ(\sqrt{2})$  or  $\BQ(\sqrt{-1})$,
                not only  do we find the usual system of
                type $B_2$ but  we also find a
                system which
                affords the exterior automorphism $\lexp2B_2$. 
                If we were considering
                numbers instead of ideals, this automorphism would exist only if the
                number $2$ (rather than the ideal generated by 2) has a square root.
                
                The exceptional group denoted by
                $G_{29}$, defined  over $\BQ(\sqrt{-1})$, has  a subgroup of type
                $B_2$  and the normaliser of this subgroup induces the
                automorphism $\lexp2B_2$. Our corresponding root system has the same
                automorphism since $1+i$ and $1-i$ generate the same ideal and
                $(1+i)(1-i)=2$, thus 2 has an ``ideal square root'' 
                in the ring $\BZ[\sqrt{-1}]$.
\end{itemize}
        
        Perhaps the most interesting and intriguing fact which comes out of the classification
        concerns the generalisation of the notions of \emph{connection index\/} and 
        \emph{bad primes\/}:     
        in the case of spetsial reflection groups, the order of the group is
        divisible by the factorial of the rank times the connection index, and the
        bad primes for the corresponding Spets make up the remainder, just as in the case 
        of finite reductive groups and Weyl groups.\footnote{Notice though that, as shown by
        Nebe \cite{nebe}, the bad primes for spetsial groups (see Section 8)
        do \emph{not\/} occur as divisors
        of the orders of the quotient of the root lattice by the root lattice of maximal reflection
        subgroups.}

\section{\red{Complex reflection groups}}

        We denote $\lambda \mapsto \lambda^*$ the complex conjugation and we 
        denote by
        $k$ a subfield of $\BC$ stable by complex conjugation. 
        
\subsection{Preliminary material about reflections}\habel{preliminary}

        Let $(V,W)$ be a pair of finite dimensional $k$-vector spaces with a given
        Hermitian pairing $V\times W\to k: (v,w)\mapsto \scal vw$; that is:
\begin{itemize}
        \item 
                $\scal\pd\pd$ is linear in $V$ and semi-linear in $W$:
                for $\lambda, \mu\in k$, $v\in V$ and $w\in W$ we have 
                $\scal{\lambda v}{\mu w} = \lambda\mu^*\scal vw.$
         \item 
                $\scal vw=0$ for all $w \in W$ implies $v=0$.
         \item 
                $\scal vw=0$ for all $v \in V$ implies $w=0$.
\end{itemize}
        Similarly $(w,v)\mapsto{\scal vw}^*$ defines a Hermitian pairing
        $W\times V\to k$ that we will also denote $(w,v)\mapsto\scal wv$ 
        when its meaning is clear from the context.
                
        Any vector space can be naturally endowed with a Hermitian pairing with its twisted dual: 
        
\begin{definition}\label{twistedddual}
        The \emph{twisted dual} \index{Twisted dual} of a 
        $k$-vector space $V$, denoted $\lexp *V$, is the $k$-vector space which is 
        the conjugate under $*$ of the dual $V^*$ of $V$. In other words,
\begin{itemize}
        \item 
                as an abelian group, $\lexp *V =V^*$,
         \item 
                an element $\la \in k$ acts on $\lexp *V$ as $\la^*$ acts on $V^*$. 
\end{itemize}   
        The pairing $V\times \lexp *V \to k: (v,\phi)\mapsto \phi(v)$ is a 
        Hermitian pairing, called the \emph{canonical pairing associated with $V$}.
        \index{Canonical pairing associated with $V$}
\end{definition}
                
        When $V$ is a real vector space, the twisted dual is the usual dual.

        Let $G_{(V,W)}$ be the subgroup of $\GL(V)\times\GL(W)$ which preserves the pairing.
        The first (resp. second) projection gives an isomorphism $G_{(V,W)}\iso\GL(V)$ (resp.
        $G_{(V,W)}\iso\GL(W)$). Composing the second isomorphism with the inverse of the
        first we get an isomorphism $g\mapsto g^\vee:\GL(V)\iso\GL(W)$. The inverse
        morphism has the same definition reversing the roles of $V$ and $W$, and
        we will still denote it $g\mapsto g^\vee$ so that $(g^\vee)^\vee=g$.
        
        In the case of the canonical pairing associated with $V$,
        the isomorphism $g\mapsto g^\vee$ is just the contragredient $\lexp tg\inv$.

\begin{definition}\label{reflectingspaces}
        A \emph{reflection} \index{Reflection} is an element $s\in\GL(V)$ of finite order such that
        $\ker(s-1)$ is an hyperplane. 
        Define the
\begin{itemize}
        \item 
                 {\em reflecting hyperplane} \index{Reflecting hyperplane} of $s$ as $H_s:=\ker(s-1)$,
        \item 
                 {\em reflecting line} \index{Reflecting line} of $s$ as $L_s:=\im(s-1)$,
        \item 
                 {\em dual reflecting line} \index{Dual reflecting line} $M_s$ of $s$ as the orthogonal (in $W$) of $H_s$,
        \item   
                 {\em dual reflecting hyperplane} \index{Dual reflecting hyperplane} $K_s$ of $s$ as the orthogonal (in $W$) of $L_s$.
\end{itemize}
Denote by $\zeta_s$ the determinant of $s$, which is a root of unity.

\end{definition}

        It is clear that a reflection $s$ is determined by $H_s$, $L_s$ and
        $\zeta_s$. In turn, $H_s$ is determined by the dual reflecting line
        $M_s$. Note that $M_s$ is not orthogonal
        to $L_s$ since $H_s$ does not contain $L_s$.
        Thus giving a reflection is equivalent to giving the following data:
        
\begin{definition}
        A {\em reflection triple} \index{Reflection triple} is a triple $(L,M,\zeta)$ where
\begin{itemize}
        \item
                $L$ is a line in $V$ and $M$ a line in $W$ which are not orthogonal.
        \item
                $\zeta\in k^\times$ is a root of unity.
\end{itemize}
\end{definition}

Formulae for the reflections defined by a reflection triple are symmetric in $V$ and $W$:

\begin{proposition}\label{computereflection}
        A reflection triple $(L,M,\zeta)$ defines a pair of reflections 
        $(s,s^\vee) \in\GL(V)\times\GL(W)$ (which preserve the pairing)
        by the formulae:
        $$
        \begin{aligned}
                s(v)&=v-\frac{\scal vy}{\scal xy}(1-\zeta)x,\\
                s^\vee(w)&=w-\frac{\scal wx}{\scal yx}(1-\zeta)y,\\
        \end{aligned}
        $$
        for any non-zero $x\in L$ and $y\in M$.
\end{proposition}

\begin{proof}
        An easy computation shows that the pair $(s,s^\vee)$ of reflections preserves the pairing 
        $\scal\pd\pd$.
        Furthermore, the reflection $s$ defined by this formula determines the triple $(L,M,\zeta)$,
        and the reflection $s^\vee$ (in $W$) determines the triple $(M,L,\zeta)$.
\end{proof}

        As long as $\zeta^m\ne 1$, the pair $(s^m,s^{\vee m})$ is precisely the pair of
        reflections defined by the triple $(L,M,\zeta^m)$: the order of $s$ is the
        order of the element $\zeta\in k^\times$.

        To summarize, we have:

\begin{proposition}\label{reflectiontriples}
        Reflection triples are in bijection with
\begin{itemize}
        \item 
                reflections in $\GL(V)$,
        \item 
                reflections in $\GL(W)$,
        \item 
                pairs of reflections $(s,s^\vee)$ in $\GL(V)\times\GL(W)$
                which preserve the pairing $\scal\pd\pd$.
\end{itemize}
\end{proposition}

        An element $g\in\GL(V)$ acts naturally on reflection triples $(L,M,\zeta)$
        through the action of $(g,g^\vee)$ on pairs $(L,M)$. It follows from the previous
        proposition that $g$ commutes with a reflection $s$ if and only
        if $(g,g^\vee)$ stabilizes $(L_s,M_s)$.

\begin{notation}
        Let $t=(L,M,\zeta)$ be a reflection triple;   denote by
        $s_t$ the corresponding reflection, and write $s^\vee_t$ for 
        the reflection corresponding to $(M,L,\zeta)$. 
        We will also write $L_t, H_t, M_t, K_t$
        for $L_{s_t}, H_{s_t}, M_{s_t}, K_{s_t}$.
\end{notation}

\subsubsection*{Stable subspaces}\hfill
\smallskip

 A reflection is diagonalisable, hence so is its restriction to a stable 
 subspace. The next lemma follows directly.
 
\begin{lemma}\habel{stablebyref}
 Let $V_1$ be a subspace of $V$ stable by a reflection $s$. Then
 \begin{itemize}
  \item
   either $V_1$ is fixed by $s$ (\ie\ $V_1\subseteq H_s$),
  \item
   or $V_1$ contains $L_s$, and then $V_1 = L_s\oplus (H_s\cap V_1)$, in
   which case the restriction of $s$ to $V_1$ is a reflection.
 \end{itemize}
 In particular, the restriction of a reflection to a stable subspace is either trivial or a reflection.
\end{lemma}

\subsubsection*{Commuting reflections}

\begin{lemma}\label{commutingreflections}
        Let $t_1=(L_1,M_1,\zeta_1)$ and $t_2=(L_2,M_2,\zeta_2)$ be two reflection
         triples.
        We have the following three sets of
        equivalent assertions.
\smallskip

\begin{enumerate}
        \item[(I)]
         \begin{enumerate}
                \item[(i)]
                         $(s_{t_1},s^\vee_{t_1})$ acts trivially on $(L_2,M_2)$.
                \item[(ii)]
                        $(s_{t_2},s^\vee_{t_2})$ acts trivially on $(L_1,M_1)$.
                \item[(iii)]
                        $ L_{1} \subseteq H_{t_2}$ and $L_{2} \subseteq H_{t_1}$,
                        in which case we say that $t_1$ and $t_2$ are 
                        \emph{orthogonal}. \index{Orthogonal reflecting triples}
        \end{enumerate}
\smallskip 

        \item[(II)]
        \begin{enumerate}
                \item[(i)]
                        $(s_{t_1},s^\vee_{t_1})$ acts by $\zeta_1$ on $(L_2,M_2)$.
                \item[(ii)]
                        $(s_{t_2},s^\vee_{t_2})$ acts by $\zeta_2$ on $(L_1,M_1)$.
                \item[(iii)]
                        $ L_{1} = L_{2}$ and $H_{t_1} = H_{t_2}$,
                        in which case we say that $t_1$ and $t_2$ are 
                        \emph{parallel.} \index{Parallel reflecting triples}
        \end{enumerate}
\smallskip
 
        \item[(III)]
        \begin{enumerate}
                \item[(i)]
                        $s_{t_1}s_{t_2} = s_{t_2}s_{t_1}\,,$
                \item[(ii)]
                        $s_{t_1}$ stabilizes $t_2$.
                 \item[(iii)]
                        $s_{t_2}$ stabilizes $t_1$.
                \item[(iv)]
                        $t_1$ and $t_2$ are either orthogonal or parallel.
        \end{enumerate} 
\end{enumerate}
\end{lemma}

\begin{proof}
        The proof of both (I) and (II) proceeds by showing the equivalence of (i) and (iii), from which the     equivalence between (ii) and (iii) follows by symmetry. (III) then follows from (I), (II), and the definitions.
\end{proof}

        From now on and until the end of this section, we shall work ``on the $V$'' side.
        Nevertheless, we shall go on using notions previously introduced in connection
        with an Hermitian pairing $V\times W \ra k$.
\smallskip

\subsection{Reflection groups}\hfill
\smallskip

\begin{definition}\label{deffrg}
        A \emph{reflection group} on $k$ is a pair $(V,G)$, where
        $V$ is a finite dimensional $k$-vector space and $G$ is
        a  subgroup of $\GL(V)$ generated by reflections.
        A reflection group is said to be \emph{finite} if $G$ is finite, and \emph{complex} if $k\subseteq \BC$.
\end{definition}

Throughout this subsection, $(V,G)$ denotes a reflection group.

        Whenever a set of reflections $S$  generates $G$, then 
 $ \bigcap_{s\in S}H_s = V^G$, 
  the set of elements fixed by $G$.

\begin{definition}\label{def:essential}
        A reflection group $(V,G)$ is  \index{Essential}\emph{essential} if $V^G =\{ 0\}$.
\end{definition}

\begin{definition}
A set of reflections  is \emph{saturated} \index{Saturated set of reflections} if it is closed under conjugation by the group it generates.
\end{definition}

        When $S$ is saturated,
        $G$ is a (normal) subgroup of the subgroup of $\GL(V)$
        which stabilizes $S$, and the subspace $V_S$, defined by
        $$
                V_S := \sum_{s\in S} L_s
        \,,
        $$
         is stable by the action of $G$.
    
\subsubsection*{Orthogonal decomposition}
\habel{sec:ortho decomp}

\begin{definition}\habel{sim}
\begin{enumerate}
        \item
                Define  an  equivalence  relation  $\sim$  on  $S$  as  the
                transitive  closure of:  $s_t \sim s_{t'}$ whenever $t$ is
                not orthogonal to $t'$.
        \item\label{def:irredRS}
                The set of reflections $S$ is said to be \emph{irreducible}
                \index{Irreducible  set of reflections} if it consists of a
                unique $\sim$-equivalence class.
\end{enumerate}
\end{definition}

\begin{lemma}\label{equivstable}
        Let  $S$  be  a  saturated  set  of reflections which generates the
        reflection  group $(V,G)$.  Then the  $\sim$-equivalence classes of
        $S$ are stable under $G$-conjugacy.
\end{lemma}

\begin{proof}
        This results from the stability under $G$-conjugacy of $S$ and of the relation 
        ``{\sl being orthogonal\/}''.
\end{proof}

When  a group $G$ acts on a set $V$ (which could be $G$ itself on which $G$
acts by conjugation), for $X\subset V$ we denote by $N_G(X)$ the normalizer
(stabilizer) of $X$ in $G$, \ie\ the set of $g\in G$ such that $g(X) = X$.
\label{StabiliserFixatorNotation} 
We  denote by $C_G(X)$  the centralizer (fixator)  of $X$, \ie\  the set of
$g\in G$ such that, for all $x \in X$, $g(x) = x$.
        Notice that $C_G(X) \triangleleft N_G(X)$.

\begin{lemma}\label{Rcentralizer} Let $S$ be a set of reflections on $V$.

\begin{enumerate} 
        \item The action of $N_{\GL(V)}(S)$ on $S$ induces an injection:
                $$N_{\GL(V)}(S)/C_{\GL(V)}(S)\hookrightarrow \fS(S),$$
                into the symmetric group on $S$.
        \item 
                If  $S$ is  saturated and  irreducible, then $C_{\GL(V)}(S)$
                acts by scalars on $V_S$.
\end{enumerate}
\end{lemma}

\begin{proof} 
        Item (1) is trivial. Let us prove (2).

        Take any $g \in C_{\GL(V)}(S)$. The pair $(g,g^\vee)$ stabilizes the
        reflection   triples   of   all   reflections   in  $S$.  Take  two
        non-orthogonal  reflections  $s,  s'$  in  $S$  with  corresponding
        reflection   triples  $(L,M,\zeta)$,  $(L',M',\zeta')$  and  choose
        non-zero  elements $x\in L,y\in M, v\in  L', w\in M'$. Then $g(x) =
        \lambda  x$ and $g(v) = \mu  v$ for some non-zero scalars $\lambda$
        and $\mu$.

        Since   $S$  is   saturated,  it   also  contains   the  reflection
        corresponding  to the  triple $s\cdot  (L',M',\zeta')$. But $s(v) =
        v-\frac{\scal  vy}{\scal  xy}(1-\zeta)  x$,  so  $g(s(v))  =  \mu v
        -\frac{\scal  vy}{\scal  xy}(1-\zeta)\lambda  x  = \alpha s(v)$ for
        some  non-zero scalar $\alpha$. In  particular, this implies $\mu =
        \lambda$, and $g$ acts by scalar multiplication on $L_s+L_{s'}$.

        Since $S$ is irreducible, this shows that $g$ acts by scalar
        multiplication on $L_s+L_t$ for any pair $s,t \in S$, hence acts
        by scalar multiplication on $V_S$.
\end{proof}

\begin{lemma}\label{whatwhat}
        The number of $\sim$-equivalence classes of reflections in $S$
        is bounded by the dimension of $V$. In particular, it is finite.
\end{lemma}

\begin{proof} 
        Assume that 
         $t_1=(L_1,M_1,\zeta_1)\,,
         \dots$, $t_m=(L_m,M_m,\zeta_m)$ 
         correspond to reflections which
         belong to distinct $\sim$-equivalence classes. 
         So in particular all the $t_i$ are mutually orthogonal.
         This implies that the 
         $L_i$ are linearly independent:  for let
         $x_i\in L_i,y_i\in M_i$ be non-zero elements and assume
        $\lambda_1x_1+\lambda_2x_2+\dots+\lambda_mx_m = 0$. 
        The Hermitian product
        with $y_i$ yields
         $\lambda_i\scal{x_i}{y_i} = 0$, hence $\lambda_i = 0$.
\end{proof}

        The following lemma is straightforward to verify. 
  
\begin{lemma}\label{orthodecom} 
        Let $S$ be a set of reflections which generate the reflection group $(V,G)$. 
        Let $S = S_1\sqcup S_2\sqcup\dots\sqcup S_m$ be the
        decomposition of $S$ into $\sim$-equivalence classes.
         Denote by $G_i$ the subgroup of $G$ generated by the reflections $s$
         for $s\in S_{i}$
        and by $V_i$ the subspace of $V$ generated by the lines $L_s$
         for $s\in S_i$. 

\begin{enumerate}
  \item
   The group $G_i$ acts trivially on $\sum_{j\neq i} V_j$.
  \item
   For $1\leq i\neq j\leq m$, the groups $G_i$ and $G_j$ commute.
  \item
   $G = G_1G_2\dots G_m\,.$
\end{enumerate}
\end{lemma}

        The above result can be refined in the case where $G$ acts 
        completely reducibly on $V$ -- then the decomposition is direct (see  \ref{ifss2} below).

        Lemma~\ref{orthodecom} can be applied to obtain a criterion for finiteness of $G$.

\begin{proposition}\label{Gfinite}
        Let $G$ be the group generated by a saturated set of reflections $S$. If $S$ is finite, then $G$ is finite.
\end{proposition} 

\begin{proof}
        Consider the case where $S$ is irreducible. By the second part of
        Lemma~\ref{Rcentralizer}, we know that the centralizer of $S$ is contained in
        $k^\times$, so in particular, $G \cap C_{\GL(V)}(S)$ is contained in
        $k^\times$. Moreover, each reflection has finite order, so in fact 
        $G \cap C_{\GL(V)}(S) \subset \bmu(k)$, the set of roots of unity of $k$. 
        Now the determinants of the elements of $G$
        belong to the finite subgroup of $\bmu(k)$ generated by the determinants of the
        elements of $S$. Thus $G\cap C_{\GL(V)}(S)$ is finite.

        As $G$ is contained in $N_{\GL(V)}(S)$, by the first part of Lemma~\ref{Rcentralizer}, 
        it is an extension of $G \cap C_{\GL(V)}(S)$ by a subgroup of $\fS(S)$. 
        By finiteness of $S$, this is a finite extension. Hence $G$ is finite.

        Finally, consider the general case $G = G_1  \cdots G_m$, where the groups $G_i$ 
        are generated by        reflections in distinct equivalence classes, as in Lemma~\ref{orthodecom}. 
        By the preceding comments, each $G_i$ is finite;  so 
        $G$ must also be finite.
\end{proof}
\smallskip

\subsubsection*{The case when $G$ is completely reducible}\hfill
\smallskip

\begin{proposition}\label{ifss}
        Let $S$ be a set of reflections generating the reflection group $(V,G)$, 
        and suppose that the action of $G$ on $V$ is completely reducible. Then
 \begin{enumerate}
  \item
   $
    V = V_S \oplus V^G
    \,,
   $ and
  \item
   the restriction from $V$ to $V_S$ induces an isomorphism from
   $G$ onto its image in $\GL(V_S)$, an essential reflection group on $V_S$.
 \end{enumerate}
\end{proposition}

\begin{proof}

        The subspace $V_S$ is $G$-stable, hence since $G$ is completely
        reducible there is a complementary
        subspace $V'$ which is $G$-stable. We have $V_S\supset L_s$ for all
        $s\in S$, thus (by Lemma \ref{stablebyref}) $V'$
        is contained in $H_s$; it follows that 
        $V' \subseteq \bigcap_{s\in S} H_s \subseteq V^G$.
        So it suffices to prove that $V_S \cap V^G = 0$.
   
        Since $V^G$ is stable by $G$, there exists a complementary subspace $V''$
        which is stable by $G$. Whenever $s\in S$, we have $L_s \subseteq V''$
        (otherwise, by Lemma \ref{stablebyref}, we have $V'' \subseteq H_r$, which
         implies that $s$ is trivial since $V = V^G\oplus V''$, a contradiction). This shows
         that $V_S \subseteq V''$, and in particular that $V_S\cap V^G = 0$.
\end{proof}

\begin{proposition}\label{ifss2} 
         Let $S$ be a set of reflections generating the reflection group $(V,G)$, 
         and suppose that the action of $G$ on $V$ is completely reducible. 
        Denote by $\{V_i\}_i$ the subspaces associated to an orthogonal decomposition
        as in Lemma~\ref{orthodecom}.
\begin{enumerate}
  \item
                For $1\leq i\leq m$, the action of $G_i$ on $V_i$ is irreducible.
  \item
                $
         V_S = \bigoplus_{i=1}^{i=m} V_i
          \,.
         $
  \item 
        $G = G_1\times G_2\times\cdots\times G_m$
\end{enumerate}
\end{proposition}

\begin{proof}

        (1)
        The subspace $V_i$ is stable under $G$, and the action of $G$ on a stable
        subspace is completely reducible. But the image of $G$ in $\GL(V_i)$ is the
        same as the image of $G_i$. So the action of $G_i$ on $V_i$ is completely
        reducible.
 
        So we write $V_i = V_i' \oplus V_i''$ where $V_i'$ and $V_i''$ are stable by $G_i$. 
        Define:
        $$
                S_i' := \{s\in S_i\,|\,L_s \subseteq V_i'\}
                \quad\text{and}\quad 
                S_i'':= \{s\in S_i\,|\,L_s \subseteq V_i''\}
                 \,.
        $$
        By Lemma \ref{stablebyref},  if $s\in S_i'$, then
        $V_i'' \subseteq H_s$, and if $s\in S_i''$, then $V_i' \subseteq H_s$.
        Hence any two elements of $S_i'$ and $S_i''$ are mutually
        orthogonal. Thus one of them has to be all of $S_i$.
        
        (2)
        By Lemma~\ref{orthodecom}, we have $\sum_{j\neq i} V_j \subset V^{G_i}$.
        By (1), and by Proposition~\ref{ifss}, we then get 
        $
                V_i \cap \sum_{j\neq i} V_j = 0
                 \,.
        $
 
        (3)
        An element $g\in G_i$ which also belongs to $\prod_{j\neq i} G_j$ acts
        trivially on $V_i$. Since (by (1) and by Proposition~\ref{ifss}) the representation of
        $G_i$ on $V_i$ is faithful, we see that $g=1$.
\end{proof}

\subsubsection*{Reflecting pairs}\hfill
\smallskip
 
        For $H$ a reflecting hyperplane, notice that
        $$
          C_G(H) = \{1\} \cup \{g\in G\,\mid\, \ker\,(g-1) = H\}
          \,.
        $$
        For $L$ a reflecting line, we have:
        $$
                 C_G(V/L) = \{1\} \cup \{g\in G\,\mid\, \im(g-1) = L\}
                 \,.
        $$ 
         \noindent
         So $C_G(V/L)$  is the group of all elements of $G$
        which stabilize $L$ and
        which act trivially on $V/L$; a normal subgroup of $N_G(L)$.

        Similarly, if $M$ is a dual reflecting line in $W$, $C_G(W/M) \triangleleft N_G(M)$.

\begin{remark} 
        Recall that orthogonality between $V$ and $W$ induces a bijection between
        reflecting hyperplanes and dual reflecting lines, as well as between
        reflecting lines and dual reflecting hyperplanes.
 
        Then if $M$ is the orthogonal (in $W$) of the reflecting hyperplane $H$, 
        we have
        $$
                C_G(H) = C_G(W/M)
                \,.
        $$
\end{remark} 
 
 We will be considering the following property.  

\begin{property}\label{CR} 
         The reflection group $(V,G)$ is such that the representations of $G$ and all its proper 
         subgroups on $V$ are completely reducible.
\end{property}
        
        Note that this is the case in particular when $G$ is finite.

\begin{proposition}\label{reflectingpairs}
Let $(V,G)$ be a reflection group with Property~\ref{CR}.

\begin{enumerate}
  \item
                Let $H$ be a reflecting hyperplane for $G$. There exists a unique reflecting line $L$ such that
                $C_G(V/L) = C_G(H)$.
   
                In other words:
   
\noindent  
                Let $M$ be a dual reflecting line for $G$. There exists a unique reflecting line $L$ such that
                $C_G(V/L) = C_G(W/M)$.
  \item
                Let $L$ be a reflecting line for $G$. There exists a unique reflecting hyperplane $H$ such that
                $C_G(H) = C_G(V/L)$.
   
                In other words:
\noindent 
                Let $K$ be a dual reflecting hyperplane for $G$. 
                There exists a unique reflecting hyperplane $H$ such that
        $C_G(K) = C_G(H)$.
  \item
                If $(L,H)$ (or $(L,M)$, or $(H,K)$) is a pair as above, then
                \begin{enumerate}
                        \item
                $C_G(H)$ consists of the identity and of reflections
                $s$ where $H_s = H$ and $L_s = L$,
                \item
                $C_G(H)$ is isomorphic to a subgroup of $k^\times$, 
                and so is cyclic if $G$ is finite, and
                \item
                $N_G(H) = N_G(L) = N_G(M) = N_G(K)$.
                \end{enumerate}
\end{enumerate}
\end{proposition}

\begin{proof}[Proof of \ref{reflectingpairs}]\hfill

\bul
        Assume $C_G(H) \neq \{1\}$. Since the action of $C_G(H)$ on $V$ is completely
        reducible, there is a line $L$ which is stable by $C_G(H)$ and such that
        $H\oplus L = V$. Such a line is obviously the eigenspace (corresponding
        to an eigenvalue different from 1) for any non-trivial element of $C_G(H)$.
        This shows that $L$ is uniquely determined, and that $C_G(H)$ consists
        of 1 and of reflections with hyperplane $H$ and line $L$. It follows also
        that $C_G(H) \subseteq G(V/L)$.
        Notice that $H$ and $L$ are the isotypic components of $V$ under the
        action of $C_G(H)$.

\bul
        Assume $C_G(V/L)\neq \{1\}$. Since the action of $C_G(V/L)$ on $V$ is completely
        reducible, there is a hyperplane $H$ which is stable by $C_G(V/L)$ and such that
        $L\oplus H = V$. Such a hyperplane is clearly the kernel of any nontrivial
        element of $C_G(V/L)$. This shows that $H$ is uniquely determined, and that
        $C_G(V/L) \subseteq C_G(H)$.
 
        We let the reader conclude the proof. 
\end{proof}

        Notice the following improvement to Lemma \ref{commutingreflections} due to
        complete reducibility.

\begin{proposition}\label{commuting2} 
        Let $t$, $t'$ be two reflection triples such that $s_t$ and $s_{t'}$ belong to a group satisfying Property~\ref{CR}. Then
\begin{enumerate}
        \item 
                $t$ and $t'$ are orthogonal if $L_t\subseteq H_{t'}$ 
                or $L_{t'}\subseteq H_t$.
        \item 
                $t$ and $t'$ are parallel if $ L_t = L_{t'}$ or $H_t = H_{t'}$,
\end{enumerate}
\end{proposition}

\subsection{The Shephard--Todd classification}\hfill

\begin{definition}\label{isomRefGp}
                Given $(V,G)$ and $(V',G')$ finite reflection groups on $k$,
                an isomorphism from $(V,G)$ to $(V',G')$ is a $k$-linear
                isomorphism $f : V \iso V'$ which conjugates the group $G$
                onto the group $G'$.
\end{definition}

        From now on in this subsection we assume that $k \subseteq \BC$.
\smallskip

\subsubsection*{The family of finite complex reflection groups denoted
$G(de,e,r)$ \index{Gde@$G(de,e,r)$}}\hfill
\smallskip
 
        Let $d$, $e$ and $r$ be three positive integers.

  Let $D_r(de)$ be the set of diagonal complex matrices with
  diagonal entries in the group $\bmu_{de}$ of all $de$--th
  roots of unity. 
  The $d$--th power of the determinant defines a surjective
  morphism
  $$
   {\det}^d : D_r(de) \twoheadrightarrow \bmu_e
   \,.
  $$
  Let $A(de,e,r)$ be the kernel of the above morphism. In
  particular we have
  $|A(de,e,r)| = (de)^r/e\,.$  
  Identifying the symmetric group $\fS_r$ with the usual
  $r\times r$ permutation matrices,
  we define
  $$
   G(de,e,r) := A(de,e,r) \rtimes \fS_r
   \,.
  $$

        We have
        $
                |G(de,e,r)| = (de)^r r!/e
                \,,
        $
        and $G(de,e,r)$ is the group of all monomial $r\times r$
        matrices, with entries in $\bmu_{de}$, and 
        product of all non-zero entries in $\bmu_d$.

\begin{examples}\hfill

\begin{itemize}
  \item
        $G(e,e,2)$ is the dihedral group of order $2e$.
        \item
         $G(d,1,r)$ is isomorphic to the wreath product
        $\bmu_d\wr\fS_r$. For $d= 2$, it is isomorphic to the
        Weyl group of type $B_r$ (or $C_r$).
  \item
        $G(2,2,r)$ is isomorphic to the Weyl group of type $D_r$.
\end{itemize}
\end{examples}

        The following theorem, stated in terms of abstract groups, is
        the main result of \cite{shto}.
        It is explicitly proved in \cite[2.4, 3.4 and 5.12]{co}.
        
\begin{theorem}[Shephard--Todd)]\label{classi} 
        Let $(V,G)$ be a finite irreducible complex reflection group. Then one of 
        the following assertions is true:
\begin{itemize}
  \item
         $(V,G) \simeq (\BC^r,G(de,e,r))$ for some integers $d,e,r$, with $de\geq 2$, $r\geq 1$
  \item
                $(V,G) \simeq (\BC^{r-1},\fS_r)$ for some an integer $r\geq 1$
                
  \item
        $(V,G)$ is isomorphic to one of 34 exceptional reflection groups.
\end{itemize}
\end{theorem}
   The exceptional groups are traditionally denoted $G_4,\ldots,G_{37}$.
\begin{remark}
        Conversely, any group $G(de,e,r)$ is irreducible on $\BC^r$
        except for $d=e=1$ and $d=e=r=2$.
\end{remark}

\begin{remark}\label{rem:galois&shephardtodd}
        Theorem \ref{classi} has the following consequence.
        
        Assume that $(V,G)$ is a complex finite reflection group where $V$ is
        $r$-dimensional.
        Choose a
        basis of $V$ so that $G$ is identified with a subgroup of $\GL_r(\BC)$.
        Now, given an automorphism $\si$ of the field $\BC$, applying $\si$ to all
        entries of the matrices of $G$ defines another group $\lexp\si G$ and
        so another complex finite reflection group $(V,\lexp\si G)$.
        
        Then it follows from Theorem \ref{classi} that there exists $\phi \in \GL(V)$
        and $a \in \Aut(G)$ such that, for all $g\in G$,
        $$
                \si(g) = \phi a(g) \phi\inv
                \,.
        $$
\end{remark}

\begin{definition}\label{def:wellgenerated}
        A finite reflection group $(V,G)$ is said to be \index{Well-generated} 
        \emph{well-generated} if $G$ may be generated by $r$ reflections, where $r=\dim(V)$.
\end{definition}
    The well-generated irreducible groups are $G(d,1,r), G(e,e,r)$ and all
    the exceptional groups excepted
    $G_7,G_{11},G_{12},G_{13},G_{15},G_{19},G_{22},G_{31}$.
\smallskip

\subsubsection*{Field of definition}\hfill
\smallskip

        The following theorem has been proved (using a case
        by case analysis) by Benard \cite{ben} (see also
        \cite{field}), and generalizes a well known result
        on Weyl groups.
 
\begin{theodef}\label{bessis}
        Let $(V,G)$ be a finite complex reflection group. Let
        $\BQ_G$ \index{QG@$\BQ_G$} be the field generated by the traces
        on $V$ of all elements of $G$. Then all
        irreducible 
        $\BQ_G G$--representations are absolutely irreducible.
 
        The field $\BQ_G$ is called the \emph{field of definition}
        of the reflection group $(V,G)$.
\end{theodef}

\bul
        If $\BQ_G \subseteq \BR$, then $(V,G)$ is a (finite)
        {\sl Coxeter group\/}. 

\bul
        If $\BQ_G = \BQ$, then $(V,G)$ is a Weyl group.

\subsection{Parabolic subgroups}\hfill

        Throughout this subsection we assume only that 
        $V$ is a $k$-vector space of finite dimension, and 
        that $G$ is a finite subgroup of $\GL(V)$.
 
        We denote by $\Rf(G)$\index{Ref(G)@$\Rf(G)$} 
        the set of all reflections of $G$, and by $\Arr(G)$ 
        \index{Arr(G)@$\Arr(G)$}
        the set of reflecting hyperplanes of elements of $\Rf(G)$.
 
        Notice that, since $G$ is finite and $k$ of characteristic zero,
        the $kG$-module $V$ is completely reducible.
\smallskip
 
\begin{definition}
        We denote by $\Arr_X(G)$\index{ArrXG@$\Arr_X(G)$}
        the set of reflecting hyperplanes containing $X$, 
        and by
        $F_X$\index{FX@$F_X$}
        the \emph{flat of $X$ in $\Arr(G)$}:\index{Flat}
        $$
                F_X := \bigcap_{H\in\Arr_X(G)} H
                \,.
        $$
\end{definition}
  
        The assertion (1) of the following theorem has first been proved by Steinberg \cite{st}. 
        A short proof may now be found in \cite{lehrer}.

\begin{theorem}\label{para}
        Let $X$ be a subset of $V$.
\begin{enumerate}
        \item
                The fixator $C_G(X)$ of $X$ is generated by those reflections
    whose reflecting hyperplane contains $X$.
        \item
    The flat $F_X$ is the set of fixed points of $C_G(X)$ and there exists a unique
    $C_G(X)$--stable subspace $V_X$ of $V$ such that
    $
     V = F_X \oplus V_X
     \,.
    $
        \item 
                $C_G(X)=C_G(F_X)$ and $N_G(X)/C_G(X)$ is 
                naturally isomorphic to a subgroup of $\GL(F_X)$.
\end{enumerate}
\end{theorem}

\begin{proof}[Proof of (2)]
        Since $C_G(X)$ is generated by reflections whose reflecting hyperplanes contain $F_X$,
        we see that the flat $F_X$ is fixed by $C_G(X)$. Conversely, if $x\in V$ is fixed
        under $C_G(X)$, it
        is fixed by all the reflections of $C_G(X)$, hence belongs to $F_X$.
 
        If $F_X = 0$, the assertion (2) is obvious. Assume $F_X \neq 0$. Then $C_G(X) \neq 1$.
        Since $F_X$ is the trivial isotypic component of $C_G(X)$, the space $V_X$ is the sum of
        all nontrivial isotypic components. \end{proof}

\begin{definition}
        The fixators of subsets of $G$ in $V$ are called 
        \index{Parabolic subgroup} \emph{parabolic subgroups}
        of $G$.
\end{definition}

        By Theorem \ref{para} above, a parabolic subgroup $C_G(X)$ acts faithfully as
        an essential reflection group on the uniquely defined subspace $V_X$.
  
\begin{corollary}\label{orderreversing}
        The map
        $F \mapsto C_G(F)$
        is an order reversing bijection
        from the set of all flats of $\Arr(G)$ onto
        the set of parabolic subgroups of $G$
        (where both sets are ordered by inclusion).
\end{corollary}

\subsection{Linear characters of a finite reflection group}\hfill
\smallskip

        Let $(V,G)$ be a finite reflection group.

        The following description of the linear characters of a reflection group,
        inspired by the results of \cite{co},
        may be found, for example, in \cite[Theorem 3.9]{berkeley}.
 
        Denote by $G^\ab$ the quotient of $G$ by its derived group, so that
        $
                \Hom(G,\BC^\times) = \Hom(G^\ab,\BC^\times) 
                \,.
        $
        Recall that $\Arr(G)$ denote the collection of reflecting hyperplanes
        of the reflections $s$ for $s \in G$.

        In what follows, the notation $H\in\Arr(G)/G$ means that $H$ runs over a complete
        set of representatives of the orbits of $G$ on the set $\Arr(G)$ 
        of its reflecting hyperplanes.

\begin{theorem}\habel{linear}

 \begin{enumerate}
  \item
   The restrictions from $G$ to $C_G(H)$ define an isomorphism
   $$
    \Hom(G,\BC^\times) \iso \prod_{H\in\Arr(G)/G} \Hom(C_G(H),\BC^\times)
    \,.
   $$
  \item
   The composition $i_H:C_G(H)\to G\to G^\ab$ is injective, and
   $$
     \prod_{H\in\Arr(G)/G}i_H:
     \prod_{H\in\Arr(G)/G}C_G(H)\to G^\ab
   $$
   is an isomorphism.
 \end{enumerate}
\end{theorem}

\begin{corollary}\label{conj gens C_G(H)}
        Let $S$ be a generating set of reflections for 
        $G$ and let $\CO$ be the set of $G$-conjugates of the elements of $S$. 
        Then for any $H\in\Arr(G)$ the set $\CO\cap C_G(H)$ generates $C_G(H)$.
\end{corollary}

\begin{proof}
        The set $\CO\cap C_G(H)$ has the same image in $G^\ab$ as the
        set $S_H$ of elements of $S$ which are conjugate to an element of $C_G(H)$.
        If we denote $x\mapsto x^\ab$ the quotient map $G\to G^\ab$, we
        have $S^\ab=\coprod_{H\in\Arr(G)/G}S_H^\ab$, where $S_H^\ab$ lies
        in the component $C_G(H)$ of $G^\ab$.
        Since $S$ generates $G$, $S^\ab$ generates $G^\ab$, thus
        $S_H^\ab$ generates $C_G(H)$.
\end{proof}

\begin{definition}\label{def:distinguished}
        Let $G$ be a finite subgroup of $\GL(V)$ generated by reflections.
        A reflection $s\in G$ is said to be \index{Distinguished reflection} 
        \emph{distinguished} with respect to $G$ if 
        $\det(s) = \exp\left(\frac{2 \pi i}{d} \right)$ 
        where $d=|C_G(H_s)|$. 
\end{definition}

        In particular, if $H$ is a reflecting hyperplane for a reflection of $G$,
        every $C_G(H)$ is generated by a single distinguished reflection.
        
        The next property has been noticed by Nebe (\cite{nebe}, \S 5), as a consequence
        of \cite[(1.8) \& (1.9)]{co}.

\begin{corollary}\label{distinguishedgenerators}
        Let $S$ be a generating set of distinguished reflections for
        $G$. 
        Then any distinguished reflection of $G$ is conjugate to an element of $S$.
\end{corollary}

\begin{proof}
        It follows from \ref{conj gens C_G(H)} and from the fact that the conjugate
        of a distinguished reflection is still distinguished.
\end{proof}

\section{\red{Root Systems}}

\subsubsection*{Notation and conventions}\hfill
\smallskip
 
        From now on, the following notation will be in force.

        The field $k$ is a number field, stable by the complex conjugation
        denoted $\lambda \mapsto \lambda^*$. Its ring of integers is $\BZ_k$, a
        Dedekind domain. A 
        \index{Fractional ideal}\emph{fractional ideal}
        is a finitely generated $\BZ_k$-submodule of $k$. 
        Denote by $\ideal\lambda$ the (principal) fractional ideal generated by $\lambda \in k$.
 \smallskip
  
        For $\fa$ a fractional ideal, we set\index{a@$\fa \inv$}
        $$
                \fa\inv := \{ b\in k\mid b\fa \subset \BZ_k\}, \text{ and } 
                \fa\tinv := {(\fa\inv)}^*
                \,.
        $$
        Since $\BZ_k$ is Dedekind, 
        $\fa\fa\inv = \ideal 1$ and $(\ideal\lambda)\inv = \ideal{\lambda\inv}$
        for $\lambda\in k$.
\smallskip

        Throughout, $(V,W)$ is a pair of finite dimensional $k$-vector spaces 
        with a given Hermitian pairing (see Subsection~\ref{preliminary})
        $$
                V\times W\to k: (v,w)\mapsto \scal vw
                \,.
        $$

        For $I$ a finitely generated $\BZ_k$-submodule of $V$ and $J$ a finitely
        generated $\BZ_k$-submodule of $W$, 
        we denote by $\scal IJ$\index{I,J@$\scal{I}{J}$}
        the fractional ideal generated by all $\scal \al\be$ for $\al\in I$
        and $\be \in J$.
\medskip

        Let $I$ be a rank one finitely generated $\BZ_k$-submodule of $V$,
        generating the line $kI$ in $V$. Then whenever $v$ is a nonzero
        element of $kI$, there is a fractional ideal $\fa$ of $\BZ_k$ such that 
        $I = \fa v$.
        If, similarly, $J = \fb w$ for some fractional ideal $\fb$ and
        some $w \in kJ$, then
        $$
                \scal IJ = \fa\fb^*\scal vw
                \,.
        $$

\subsection{\zroot s}\hfill

\begin{definitions}\hfill

\begin{enumerate}
        \item
                A \index{Zk@\zroot} \emph{\zroot} 
                (for $(V,W)$) is a triple
                $\frr = (I,J,\zeta)$ where
                \begin{itemize}
                        \item 
                        $I$ is a rank one finitely generated $\BZ_k$-submodule of $V$, 
                         \item 
                        $J$ is a rank one finitely generated $\BZ_k$-submodule of $W$, 
                \item 
                        $\zeta$ is a nontrivial root of unity in $k$,
                \end{itemize}
                such that $\scal IJ=\ideal{(1-\zeta)}$, the principal ideal generated by $1-\zeta$.

                A \zroot\ $\frr = (I,J,\zeta)$ is called a
                \index{zeta@$(\zeta,\BZ_k)$-root} \emph{$(\zeta,\BZ_k)$-root}.
        \item
                If $\frr = (I,J,\zeta)$ is a \zroot, and $\fa$ is a fractional ideal, we set
                $$
                        \fa\cdot\frr := (\fa I,\fa\tinv J,\zeta)
                        \,.
                $$
                Two \zroot s $\frr_1$ and $\frr_2$
                are said to be of the same \emph{genus} \index{Genus of \zroot s} 
                if there exists a fractional ideal $\fa$ such that
                $$
                        \frr_2 = \fa\cdot\frr_1
                        \,.
                $$
\end{enumerate}
\end{definitions}
 
        The group $\GL(V)$ acts on left on the set of \zroot s, as follows:
        for $g\in\GL(V)$ and $\frr = (I,J,\zeta)$ a \zroot, set
        $$
                g\cdot\frr := (g(I),g ^\vee (J),\zeta) .
        $$
        In particular $\lambda\in k^\times\subset Z\GL(V)$ acts by
        $\lambda\cdot\frr=(\lambda I,\lambda\tinv J,\zeta)$.
        The action of $k^\times\Id_V = Z\GL(V)$ preserves genera.

\begin{remark}
        The pair $(I,J)$ does not determine $\zeta$.
 
        Indeed one may have an equality of ideals
        $\ideal{(1-\zeta)}=\ideal{(1-\xi)}$ without $\zeta$ and $\xi$ having even
        the same order. For example, as soon as $\zeta$ has a composite order,
        $1-\zeta$ is invertible
        and so $\ideal{(1-\zeta)}= \ideal{}$
        (see Lemma~\ref{lem:mcompositeornot} in Appendix~\ref{arithmetic}).
\end{remark}

        Given a \zroot\ $\frr = (I,J,\zeta)$, choose
        $v\in kI$ and $w\in kJ$ such that
        $\scal v{w} = 1-\zeta$.
        Then the formula
        $$
                x\mapsto x-\scal xw v
        $$
        defines a reflection independent of the choice of $ v$,
        since it is also the reflection attached to the reflection triple
        $(kI,kJ,\zeta)$.
        We will denote by $s_\frr$ this reflection.

\begin{definitions}\habel{reflection and dual}

\begin{enumerate}
        \item
                If $s$ is a reflection,
                an \index{s,Zk-root@$(s,\BZ_k)$-root} \emph{$(s,\BZ_k)$-root} is a
                \zroot\ $(I,J,\zeta)$
                where $(kI,kJ,\zeta) = (L_s,M_s,\zeta_s)$.
        \item
                If $\frr = (I,J,\zeta)$ is a \zroot\ for $(V,W)$, 
                we call $\frr^\vee=(J,I,\zeta)$ --- a \zroot\ for $(W,V)$ --- the {\em dual root}.
\end{enumerate}
\end{definitions}

                Notice that the dual of an $(s,\BZ_k)$-root is an
                $(s^\vee,\BZ_k)$-root. Thus
                $s_{\frr^\vee}=s_\frr^\vee$.
                
\begin{lemma}\label{lemma:reflection=genus}\hfill
\begin{enumerate}
        \item
                Given a \zroot\ $\frr = (I,J,\zeta)$, given $v \in kI$
                and $w\in kJ$ such that $\scal vw = 1-\zeta$, there exists
                a fractional ideal $\fa$ such that 
                $I = \fa v$ and $J = \fa\tinv w$.
        \item
                For any \zroot\ $\frr = (I,J,\zeta)$, there exists a unique
                reflection $s$ in $\GL(V)$ such that $\frr$ is an $(s,\BZ_k)$-root.
        \item
                For any reflection $s$ in $\GL(V)$, the set of $(s,\BZ_k)$-roots
                form a single genus of roots.
\end{enumerate}
\end{lemma}

\begin{proof} 
        (1) and (2) are clear. Let us prove (3). Let $s$ be a reflection.
        
        Choose $ v \in L_s$ and $w \in M_s$ such that 
        $ \scal vw = 1-\zeta_s$. For $\fa$ any fractional ideal, define
        $I := \fa v, J := \fa\tinv w\,$. Then $\frr=(I,J,\zeta_s)$ is an
        $(s,\BZ_k)$-root. 
        
        Let now $\frr'$ be a root giving rise to the same reflection triple.
        Then $\frr' = (\fb v,\fb\tinv w,\zeta)$
        for some fractional ideal. We have $\frr'=\fb\fa\inv\cdot\frr$
        thus $\frr$ and $\frr'$ are in the same genus.
\end{proof}
  
\begin{remark}\label{IzDetermineJ}
        Given a reflection $s$ and an $(s,\BZ_k)$-root $\frr = (I,J,\zeta)$, 
        Lemma \ref{lemma:reflection=genus}, (1) ensures that 
        $J$ is determined by $I$ (and similarly $I$ is determined by $J$). 
\end{remark}

\subsubsection*{Pairing between \zroot s}\hfill
\smallskip

 Let $\frr_1 =(I_1,J_1,\zeta_1)$ and $\frr_2 = (I_2,J_2,\zeta_2)$ be two
 \zroot s.
 There is a pairing on the set of \zroot s, defined to be the fractional 
 ideal:
  \index{n(r1,r2)@$\frn{\frr_1}{\frr_2}$}
 $$
   \frn{\frr_1}{\frr_2} := \scal{I_1}{J_2}
  \,.
 $$
 If $\frr=(I,J,\zeta)$, then by definition we have $\frn\frr\frr=\ideal{(1-\zeta)}$.
\smallskip

\subsubsection*{Principal \zroot s}\hfill
\smallskip

        Let $I$ be a rank one $\BZ_k$-submodule of $V$.
        The reader will easily check that
        the following assertions are equivalent:
        \begin{itemize}
                \item[(i)\,]
                        $I$ is a free $\BZ_k$-module (hence of rank 1),
                \item[(ii)\,]
                        whenever $v \in kI$ and $\fa$ is a fractional ideal of $k$ such
                        that $I = \fa v$, then $\fa$ is a principal ideal.
        \end{itemize}
	This implies the following result:

\begin{lemmadef}\label{lem:principalroot}\index{Principal root}
        Let $\frr = (I,J,\zeta)$ be a \zroot. 
        The following assertions are equivalent:
        \begin{itemize}
                \item[(i)\,]
                        $I$ is a free $\BZ_k$-module (hence of rank 1),
                \item[(ii)\,]
                        $J$ is a free $\BZ_k$-module (hence of rank 1).
        \end{itemize}
        If the preceding properties are true, we say that the root $\frr$
        is a \emph{principal \zroot\/}.
\end{lemmadef}

\begin{remark}\label{rem:principal}
        If $\frr = (I,J,\zeta)$ is a principal \zroot, we may choose
        $\al\in kI$ and $\be \in kJ$ such that $I = \BZ_k \al$, $J = \BZ_k \be$ and
        $\scal{\al}{\be} = 1-\zeta$.
        The vector $\al$ is then unique up to multiplication by a unit
        of $\BZ_k$, and it determines $\be$ (and conversely).
\end{remark}
 
\subsection{\zroot\ systems}\hfill
\smallskip

\subsubsection*{Definition and first properties}\hfill
\smallskip

        The following definition 
        is modeled on that of Bourbaki \cite[chap. VI, \S 1, \emph{D\'efinition}]{bou}.

\begin{definition}\label{Zkrootsystems}
        Let 
        $\fR = \{ \frr = (I_\frr,J_\frr,\zeta_\frr) \}$ be a set of \zroot s.
        We say that $\fR$ is a \emph{\zroot\ system} \index{Zk@\zroot\ system}
        if it satisfies the following conditions:
        \begin{description}
                \item[(RS$_I$)]
                        $\fR$ is finite, and the family $(I_\frr)_{\frr\in\fR}$ generates $V$,
        \item[(RS$_{II}$)]
                        Whenever $\frr \in \fR$, we have $s_\frr\cdot \fR = \fR$,
        \item[(RS$_{III}$)]
                        Whenever $\frr_1,\frr_2 \in \fR$, we have
                        $\frn{\frr_1}{\frr_2} \subseteq \BZ_k$.
        \end{description}
\end{definition}
\medskip

        In particular, in the case when $\BZ_k = \BZ$, the root datum above is
        equivalent to that required for a root system as defined in 
        \emph{loc.cit.} (see Remark \ref{rem:Zk=Z}).
        
         If $G$ is any of the 34 exceptional reflection groups of the classification of finite irreducible
        complex reflection groups, and $k=\BQ_G$ is the field of definition of $G$, then $\BZ_k$ is 
        known to be a principal ideal domain~\cite{nebe}.
\smallskip
  
\subsubsection*{Principal \zroot\ systems} \label{pid case}
\hfill \smallskip

        If $\BZ_k$ is a principal ideal domain, all \zroot s are principal.

\begin{definition}\label{def:princrootsys}
        A \zroot\ system is  \emph{principal\/} if all its roots are principal.
\end{definition}

        Remark~\ref{rem:principal} implies that  
        a principal \zroot\ may be viewed as a triple $(A,B,\zeta)$ where 
\begin{itemize}
        \item
                $\zeta$ is a root of unity, 
        \item
                $A = \BZ_k^\times \alpha$ and $B = \BZ_k^\times \beta$,
                where $\alpha$ and $\beta$
                are nonzero elements of $V$ and $W$ respectively, and
        \item
                $\scal\alpha\beta = 1-\zeta\,.$
\end{itemize}
        Such a triple $r$ defines the unique reflection $s_r$ 
        with reflecting line $kA$ 
        and reflecting hyperplane the orthogonal 
        of $kB$.

        Thus a principal \zroot\ system may be viewed as a set $R$
        of triples $r = (A_r,B_r,\zeta_r)_{r\in R}$ such that
\begin{itemize}
        \item[(RS$_I$)] 
                $R$ is finite and the family $(A_r)_{r\in R}$ generates $V$,
        \item[(RS$_{II}$)]
                Whenever $r \in R$, we have $s_r\cdot R = R\,,$
        \item[(RS$_{III}$)]
                Whenever $r_1 = (A_1,B_1,\zeta_1) \in R$ and $r_2 =
                (A_2,B_2,\zeta_2) \in R$, 
                for $\alpha_1\in A_1$ and $\beta_2 \in B_2$,
                we have $\scal{\alpha_1}{\beta_2} \in \BZ_k\,.$
\end{itemize}

\begin{remark}\label{rem:Zk=Z}
        If $\BZ_k = \BZ$ (which implies that $G$ is a Weyl group), the previous
        definition coincides with the usual definition of root system attached to $G$:
        let $\fR_0$ be a root system in the Bourbaki sense, then 
        $$
                 \fR := \{(\BZ\alpha,\BZ\alpha^\vee,-1)\,\mid\, \alpha\in
                 \fR_0\}
        $$
        is a $\BZ$-root system in our sense. Notice that the cardinality of
        $\fR_0$ is twice that of $\fR$ as Bourbaki has distinct roots $\pm
        \alpha$, which give rise to a single $\BZ$-root.
\end{remark}

\begin{remark}\label{rem:defnebe}
        Nebe's definition of a \emph{reduced $k$-root system for $G$} 
        (see \cite[Def.19]{nebe}) coincides with our definition of
        \emph{distinguished principal \zroot\ system for $G$} (see 
        Definition \ref{variousdefs} below).
\end{remark}

\subsubsection*{Reflections and integrality results}\hfill
\smallskip

        We return to the general case, where $\BZ_k$ need not be a P.I.D.

\begin{lemma}\label{reflectionideal} 
        Given \zroot s $\frr_1 = (I_1,J_1,\zeta_1)$ and $\frr_2 =
        (I_2,J_2,\zeta_2)$,
\begin{enumerate}
        \item 
                $
                (s_{\frr_1}-\Id_V)(I_{2}) \subset \frn{\frr_2}{\frr_1}I_{1}
                        \,.
                $
        \item
                If $\frn{\frr_2}{\frr_1} \subset \BZ_k$, then 
                $
                (s_{\frr_1}-\Id_V)(I_{2}) \subset I_{1}
                \,.
                $
        \item
                Reciprocally, if $(s_{\frr_1}-\Id_V)(I_{2}) \subset I_{1}$, then
                $\frn{\frr_2}{\frr_1} \subset \BZ_k$.
\end{enumerate}
\end{lemma}
 
  
 \begin{proof}
        Choose $(v_1,w_1) \in kI_{\frr_1} \times kJ_{\frr_1}$ such that
         $\scal{v_1}{w_1} = 1-\zeta_1$, and 
        denote by $\fa_1$ the fractional
        ideal such that $I_{\frr_1} = \fa_1v_1$ (and so $J_{\frr_1} = \fa_1\tinv w_1$).
  
        Similarly, choose
        $(v_2,w_2) \in kI_{\frr_2} \times kJ_{\frr_2}$ such that
         $\scal{v_2}{w_2} = 1-\zeta_2\,,$
        and denote by $\fa_2$ the fractional
        ideal such that $I_{\frr_2} = \fa_2v_2$ (and so $J_{\frr_2} = \fa_2\tinv w_2$).
 
        Then, for all $a_2\in\fa_2$, 
         \begin{equation}
         \tag{$\star$}
                s_{\frr_1}(a_2v_2) =a_2v_2 - \scal{a_2v_2}{w_1}v_1
                 \,.
        \end{equation}
  
        In order to prove (1),
        write $1 = \sum_i y_i x_i$ for $x_i \in \fa_1$ and $y_i \in \fa_1\inv$.
         Then the above equality ($\star$) may be rewritten
        $$
         \begin{aligned}
                 s_{\frr_1}(a_2v_2) 
                  &= a_2v_2 - \scal{a_2v_2}{w_1}(\sum_i y_i x_i)v_1 \\
                  &= a_2v_2 - \sum_i \big( \scal{a_2v_2}{y_i^*w_1}x_i v_1 \big)
                 \,,
         \end{aligned}
        $$
         and that last equality shows (1).
 \smallskip

        Part (2) follows from (1) and from the inclusion $\BZ_k I_1 \subset I_1$.
 \smallskip
 
        Now assume that $(s_{\frr_1}-\Id_V)(I_{\frr_2}) \subset I_{\frr_1}$. Equality ($\star$)
        shows that, for all $a_2\in\fa_2$, $\scal{a_2v_2}{w_1}v_1 \in I_1$, \ie\
        $\scal{a_2v_2}{w_1} = a_2\scal{v_2}{w_1} \in \fa_1$. This shows that
         $\fa_2 \scal{v_2}{w_1} \subset \fa_1$, hence 
        $\scal{v_2}{w_1} \in \fa_2\inv\fa_1$ and
        $\scal{\fa_2v_2}{\fa_1\tinv w_1} \subset \BZ_k$,
         which is (3).
\end{proof}
  
  \begin{corollary}\label{RSIII equivalent}
        Condition \textbf{(RS$_{III}$)} is equivalent to: 
        $$
                \mbox{\emph{\textbf{(RS$_{III}$)$\,'$}} Whenever $\frr_1,\frr_2 \in \fR$,
                we have
                $(s_{\frr_2}-\Id)I_{\frr_1} \subset I_{\frr_2}$.}
        $$
\end{corollary}
 
\subsubsection*{The group $G(\fR)$} \hfill 
\smallskip

\begin{definition}\label{def:G(R)}
        Given a set $\fR$ of \zroot s, we denote by $G(\fR)$ \index{GR@$G(\fR)$ 
        group defined by roots $\fR$}
        the subgroup of $\GL(V)$ generated by the family of reflections 
        $(s_\frr)_{\frr\in \fR}$.
\end{definition}

If $G(\fR)=G$, we say that $\fR$ is a \zroot\ system with group $G$, or a
\zroot\ system for $G$.
\begin{remark}\label{rem:G(R)}
        By Definition~\ref{Zkrootsystems}, (RS$_{II}$), $g(\frr) \in \fR$ 
        whenever $g \in G(\fR)$ and $\frr \in \fR$.
\end{remark}

 \begin{theorem}\label{groupsandroots}\hfill
\begin{enumerate}
        \item
                Let $\fR$ be a \zroot\ system in $V$. Then  $\left(V,G(\fR)\right)$ is an essential finite
                reflection group  -- that is, $V^{G(\fR)} = 0$.
        \item
                Let $(V,G)$ be an essential finite reflection group.
                Then there exists
                a \zroot\ system in $V$ with group $G$.
\end{enumerate}
\end{theorem}

\begin{proof} 
         (1) The set of reflections $S_\fR=\{s_\frr \mid \frr \in \fR \}$ is finite, by (RS$_{I}$); 
        is saturated, by (RS$_{II}$); and generates $G(\fR)$, by definition.
        Thus Proposition \ref{Gfinite} applies, and $G(\fR)$ is finite. 
 
        (2) If $X$ is a finite 
        spanning set
        of $V$, the (finite) set 
        $$
                 \{ g(x) \,\mid\, (x\in X)(g\in G)\}
        $$
        generates a finitely generated
        $\BZ_k$-submodule $E$ of $V$, which generates $V$ as a $k$-vector space, 
        and which is $G$-stable. Since $E$ is torsion free 
        (and since $\BZ_k$ is Dedekind), $E$ is a lattice in $V$, namely
        there exists a family $E_1,\dots,E_r$ of rank  one projective 
        $\BZ_k$-modules such that 
        $
                E = E_1\oplus\dots\oplus E_r
                \,.
        $
        Notice that if $L$ is any line in $V$, then $L\cap E \neq 0$. Indeed, since
        $kE = V$, we have $L \subset kE$, thus for each $x \in L, x\neq 0$,
        there is $m\in E$
        and $\lambda,\mu\in \BZ_k, \lambda\mu\neq 0$, 
        such that $x = \frac\lambda{\mu} m$, so
        $\lambda m$ is a nonzero element of $L\cap E$.
 
        Let $W$ be a vector space with a Hermitian pairing with $V$ 
        (for example, the twisted dual $^*V)$.
        For each $s \in \Rf(G)$, with reflecting line $L_s$, dual reflecting line $M_s$, 
        and determinant $\zeta_s$, 
       
\begin{itemize}
        \item
                set $I_s := L_s\cap E$ 
                (so $I_s$ is a rank one $\BZ_k$-submodule of $E$), and
        \item
                denote by $J_s$ the rank one $\BZ_k$-submodule of $W$ in $M_s$
                such that 
                $$
                                \scal{I_s}{J_s} = \ideal{(1-\zeta_s)}
                        \,.
                $$
\end{itemize}
        Then ($I_s,J_s,\zeta_s$) is an ($s,\BZ_k$)-root.
 
        Denote by $\fR_E$ the set of all roots ($I_s,J_s,\zeta_s$) for $s \in \Rf(G)$. 
        It is clear that $G(\fR_E)=G$. 
        It remains to show that $\fR_E$ is a \zroot\ system. 
\smallskip

        \textbf{(RS$_I$)} :
        It is clear that $\fR_E$ is finite. Besides, since $V^G = 0$, we have
        $V = \sum_{s\in\Rf(G)} L_s$, hence the family $(I_s)_{s\in\Rf(G)}$
        generates $V$.
\smallskip
  
        \textbf{(RS$_{II}$)} :
        Let $g\in G$. Whenever $s\in\Rf(G)$, we have $g(L_s) = L_{\lexp gs}$, hence
        $g(L_s\cap E) = L_{\lexp gs}\cap E$, which shows that 
        $g(I_s) =I_{\lexp gs}$. It follows immediately that $\fR_E$ is stable under
        the action of $G$.
\smallskip
 
 \textbf{ (RS$_{III}$)} :
        Let $s_1,s_2 \in \Rf(G)$, and set $\frr_i := (I_{i},J_{i},\zeta_{i})$
        for $i=1,2$.
 
        Since the image of $s_2-\Id_v$ is the line $L_{s_2}$, and
        since $s_2-\Id_V$ sends $E$ to $E$, we see that
        $$
                (s_2-\Id_V)(I_1) \subset E\cap L_{s_2} = I_2
                \,.
        $$
        This condition is equivalent to \textbf{ (RS$_{III}$)} 
        by Corollary \ref{RSIII equivalent}.
\end{proof}

\begin{remark}\label{n=0}
        If $\frr_1$ and $\frr_2$ belong to the same root system $\fR$ and
        $\frn{\frr_1}{\frr_2}=0$ then 
        $s_{\frr_1}s_{\frr_2} = s_{\frr_2}s_{\frr_1}$ and
        $\frn{\frr_2}{\frr_1}=0$. 
        Indeed, by item (1) of Proposition~\ref{commuting2}, which is applicable since $G(\fR)$ is finite, 
        the equality $\frn{\frr_1}{\frr_2}=0$ implies that the
        reflection triples defined by $\frr_1$ and $\frr_2$ are orthogonal.
\end{remark}

        Let $(V,G)$ be a finite reflection group on $k$, assumed to be
        essential (Definition~\ref{def:essential}).
        Let $S$ be a set of reflections of $G$, and for each $s\in S$ let 
        $\frr_s = (I_s,J_s,\zeta_s)$ be a $(s,\BZ_k)$-root (see Definitions~\ref{reflection and dual}).
        We set $\CS := \{\frr_s\mid s\in S\}$.

		A  root  $\frr$  is  said  to be \index{Distinguished root}
		\emph{distinguished} with respect to a \zroot\ system $\fR$ if
		$s_\frr$ is distinguished with respect to $G(\fR)$.
                
\begin{proposition}\label{prop:jeanlemma}       
        Assume that
        \begin{enumerate}
                \item[(a)]
                        the set $S$ generates $G$, 
                \item[(b)]
                        for each $s,t \in S$, the ideal $\scal{I_s}{J_t}$ is integral.
        \end{enumerate}
        Then 
        \begin{enumerate}
                \item
                        the orbit $\fR$ of $\CS$ under $G$ is a
                        \zroot\ system, and $G = G(\fR)$,
                \item
                        if all elements of $S$ are distinguished, then so are the elements of $\fR$,
                        and the map $\fR \ra \Rf(G)$, $\frr \mapsto s_\frr$ is a bijection onto
                        the set of distinguished reflections of $G$, and
                \item
                        if each element of $\CS$ is principal, then $\fR$ is principal as well.
        \end{enumerate}
\end{proposition}

\begin{proof}
        (1) The axiom (RS$_I$) follows from the fact that $(V,G)$ is essential, 
        and the axiom (RS$_{II}$)
        is trivial. Let us prove (RS$_{III}$). 

\begin{lemma}\label{lem:jeanlemma}
        Let $s,t,u$ be reflections on $V$, with associated \zroot s respectively
        $\frr_s = (I_s,J_s,\zeta_s)$, $\frr_t$, $\frr_u$. 
        Then
        $$
                \scal{I_{sts\inv}}{J_u} \subset \scal{I_s}{J_u} + \scal{I_t}{J_s}\scal{I_s}{J_u}
                \,.
        $$
\end{lemma}     

\begin{proof}[Proof of \ref{lem:jeanlemma}]
        By item (1) of Lemma \ref{reflectionideal},
        $
                s(I_t) \subset I_t + \scal{I_t}{J_s}I_s
        $
        hence
        $
                \scal{I_{sts\inv}}{J_u} \subset \scal{I_s}{J_u} + \scal{I_t}{J_s}\scal{I_s}{J_u}
                \,.
        $
\end{proof}

We return to the proof of  Proposition~\ref{prop:jeanlemma}.
        Say that a set $T$ of reflections of $G$ is \emph{integral} if $s,t \in T$
        implies $\scal{I_s}{J_t}\in \BZ_k$. The preceding lemma shows that, given any
        integral set $T$ of reflections of $G$, then $T \cup \{sts\inv\mid s,t \in S \}$ is again
        an integral set of reflections of $G$.
        Thus the set of all conjugates of the elements of $S$ is integral,
        which is axiom (RS$_{III}$).

        (2)     
        Since the elements of $\CS$ are all distinguished, the assertion results from
        Corollary~\ref{conj gens C_G(H)}.

        Part (3) is obvious.
\end{proof}

\subsubsection*{Case where $V=W$ and $\scal{\pd}{\pd}$ is positive}\hfill

\begin{remark} \label{L=M}
        If $V=W$ and $\scal{\pd}{\pd}$ is positive, a reflection $s\in G$
        is determined by its root line and its determinant. 
        Indeed, since $s$ preserves 
        $\scal{\pd}{\pd}$, we have $s=s^\vee$, so if $s$ is associated to the 
        reflection triple $(L,M,\zeta)$, we have $L=M$.
\end{remark} 

\begin{lemma}\label{Hermitianformandn}
        Assume $V=W$ and $\scal{\pd}{\pd}$ is positive.
\begin{enumerate}
        \item
                Let $\frr=(I,J,\zeta)$ be a \zroot. 
                Then $J= (1-\zeta^*){\scal II}\inv I$.
        \item
                Let $\frr,\frr'$ be two roots from a \zroot\ system $\fR$.
                Then
                $$
                        \frn{\frr'}\frr^*=
                        \frn\frr\frr^* \frn{\frr'}{\frr'}\inv
                        \frn{\frr'}{\frr^{\prime\vee}}\frn\frr{\frr^\vee}\inv
                        \frn\frr{\frr'}
                        .
                $$
\end{enumerate}
\end{lemma}

\begin{proof}\hfill
 
        (1) 
        Choose $v\in kI,w\in kJ$ such that $\scal vw = 1-\zeta$.
        There is a fractional ideal $\fa$ such that $I=\fa v$
        and $J=\fa\tinv w$. By Remark~\ref{L=M} there exists $\lambda\in k$ such
        that $w =\lambda v$ thus $1-\zeta=\lambda^*\scal vv$
        and $w=\frac{(1-\zeta^*)}{\scal vv}  v$.
        Thus 
        $$
                J=\fa\tinv w=(1-\zeta^*)\dfrac v{\fa^*\scal vv}
                =(1-\zeta^*)\dfrac{\fa v}{\fa\fa^*\scal vv}
                =(1-\zeta^*) \dfrac I{\scal II} 
                .
        $$

        Part (2) 
        follows from (1) observing that for a root $\frr=(I,J,\zeta)$
        we have $\frn\frr\frr=\ideal{(1-\zeta)}$ and $\frn\frr{\frr^\vee}=\scal II$.
\end{proof}

\begin{remark}
        In the case where $k\subset\BR$, and 
        $\frr=(I,J,-1)$, $\frr'=(I',J',-1)$,
        Lemma \ref{Hermitianformandn} (2) reduces to 
        $\frn{\frr'}\frr= \scal{I'}{I'} {\scal II}\inv \frn\frr{\frr'}$
        which generalizes the case of finite Coxeter groups
        \cite[Chap 6, \S 1, no. 1.1. formula (9)]{bou}.
\end{remark}

\subsubsection*{Some properties of a root system}\hfill
\smallskip

        We return to the general case, where $W$ need not be the same as $V$.

        Let $\fR$ be a \zroot\ system. Recall that 
        $G(\fR)$ (or simply $G$) denotes the group generated by the reflections 
        defined by the elements of $\fR$. 
 
\begin{proposition}\label{ArrR}
        Let $\fR$ be a \zroot\ system. 
        Then, for any reflecting hyperplane $H$ of $G(\fR)$, 
        the fixator of $H$, $C_{G(\fR)}(H)$, is generated by 
        the set of reflections $s_\frr$ (where $\frr\in\fR$) with 
        reflecting hyperplane $H$.
\end{proposition}

\begin{proof}
        It is a consequence of Corollary \ref{conj gens C_G(H)}.
        Indeed, it suffices to notice (see Remark \ref{rem:G(R)})
        that, for $\frr \in \fR$ and $g\in G(\fR)$,
        then $gs_\frr g\inv = s_{g(\frr)}$ and $g(\frr) \in \fR$.
\end{proof}

\begin{definition}
        Let $\fR$ be a \zroot\ system.
\begin{enumerate}\label{variousdefs}
        \item\label{def:reduced}
                We say that $\fR$ is 
                 \index{Reduced root system}\emph{reduced} if the map
                $ \frr \mapsto s_\frr $ is injective.
        \item\label{def:complete}
                We say that $\fR$ is
                \index{Complete root system}\emph{complete}
                if the map      $ \frr \mapsto s_\frr $ is surjective onto $\Rf(G(\fR))$.
        \item\label{def:distinguishedsystem}
                We say that $\fR$ is
                \index{Distinguished root system}\emph{distinguished} if
        \begin{enumerate}
                \item 
                                it consists of distinguished roots, and
                \item 
                                it is reduced.
        \end{enumerate}
\end{enumerate}
\end{definition}

\begin{remarks}\hfill
\begin{enumerate} 
        \item If all $s_\frr$ have order 2 (for example, the real reflection groups and 
        the infinite family $G(e,e,r)$) then:
                \begin{itemize}
                        \item
                                every distinguished root system is complete (and reduced), and
                        \item 
                                every complete and reduced root system is distinguished.
                \end{itemize}
        \item In a reduced root system, distinct roots have different genus.
\end{enumerate}
\end{remarks}

\begin{proposition}\label{prop:distinguishedbijection}
        Let $\fR$ be a distinguished \zroot\ system. Then the map
        $\frr \mapsto s_\frr$ is a bijection from $\fR$ onto the set of distinguished
        reflections of $G(\fR)$.
\end{proposition}

\begin{proof}
        It suffices to prove that, whenever $H$ is a reflecting hyperplane of $G(\fR)$,
        there exists $\frr \in \fR$ such that $s_\frr$ is the distinguished reflection of
        $C_{G(\fR)}(H)$. This results from \ref{ArrR}.
\end{proof}

       The following lemma follows the lines of \cite[Remark 20]{nebe}.

\begin{proposition}\label{nebegen}
        Let $(V,G)$ be a $k$-reflection group. 
        Let $\fR$ be a reduced and complete \zroot\ system (resp.~a
        distinguished \zroot\ system)
        with respect to $G$.
        Assume that $\fR_1,\dots,\fR_m$ are the orbits of $G$ on $\fR$. 
        
        Then any other reduced and complete \zroot\ system 
        (resp.~distinguished \zroot\ system) with group $G$ is of the form
        $\fa_1\fR_1\cup\cdots\cup\fa_m\fR_m$ where $\fa_1,\dots,\fa_m$
        are some fractional ideals.
\end{proposition}

\begin{proof}
        We give the proof for reduced and complete systems (the proof for the distinguished ones
        is similar).
        
        Let $\fR'$ be another reduced and complete root system with respect to $G$.
        If $s$ is a reflection of $G$, and if
        $\frr_s \in \fR$ and $\frr'_s \in \fR'$ are the roots associated with $s$, we know 
        by Lemma~\ref{lemma:reflection=genus} that there exists a fractional ideal $\fa_s \neq 0$ such that
        $\frr'_s = \fa_s\cdot\frr_s$.
        
        Choose $i$ such that $1\leq i\leq m$ and choose a reflection $s$ such that $\frr_s \in \fR_i$.
        For $g\in G$, $g(\frr_s)$ is a \zroot\ attached to $g(s)$, hence by hypothesis
        we have  $g(\frr_s) = \frr_{gsg\inv}$. Similarly, $g(\frr'_s) = \frr'_{gsg\inv}$. Hence we see
        that $\fa_s = \fa_{g(s)}$, thus $\fa_s$ depends only on $i$. We set $\fa_i := \fa_s$. 
        This shows that if $\fR_1,\dots,\fR_m$ are the orbits of $G$ on $\fR$, then
        for each $i=1,\dots,m$, there exists a fractional ideal $\fa_i$ such that
        $\fa_1\fR_1,\dots,\fa_m\fR_m$ are the orbits of $G$ on $\fR'$.
\end{proof}
\smallskip

\subsubsection*{Distinguishing and completing root systems} \hfill 
\smallskip

\begin{proposition}\label{ReducingR}
        Any \zroot\ system $\fR$ contains a reduced subsystem with group
        $G(\fR)$.
\end{proposition}

\begin{proof}
	Let $[\Arr(G(\fR))/G(\fR)]$ be a complete set of representatives of orbits of $G(\fR)$ on the
	set $\Arr(G(\fR))$ of reflecting hyperplanes of $G(\fR)$. For each $H \in [\Arr(G(\fR))/G(\fR)]$,
	let us set 
	$$
		\Rf_\fR(H) := \{ s_\frr \mid (\frr \in \fR)\, (s_\frr \in C_{G(\fR)}(H))\}
		\,.
	$$
        By Corollary~\ref{conj gens C_G(H)}, 
        we know that $\Rf_\fR(H)$ generates $C_{G(\fR)}(H)$.
        For each $H \in [\Arr(G(\fR))/G(\fR)]$, choose a subset $\Rf^0_\fR(H)$ of $\Rf_\fR(H)$
        which is minimal subject to being a generating subset of $C_{G(\fR)}(H)$. 
        For each element of
        $\Rf^0_\fR(H)$, we choose a corresponding $\frr \in \fR$, so we get a set
        $\{\frr_1,\dots,\frr_n\}$ of elements of $\fR$.
        
        Define $\fR'$ to be the union of the $G(\fR)$-orbits of $\{\frr_1,\dots,\frr_n\}$.
        We claim that $\fR'$ is a reduced root system with group $G(\fR)$.

        It is clear that  $\fR'$ is a root system with
        group $G(\fR)$. Let us prove that $\fR'$ is reduced.
        Suppose that $\frr$ and $g\cdot\frr=\fa\cdot\frr$ are in $\fR'$ for
        $g\in G(\fR)$ for some fractional ideal $\fa$, then $g^n\cdot\frr=\fa^n\cdot\frr=\frr$
        for some $n$ which shows that $\fa^n=\BZ_k$ which implies 
        (since $\BZ_k$ is a Dedekind domain, hence fractional ideals have a unique
        decomposition in prime ideals) that $\fa=\BZ_k$.
\end{proof}

\begin{proposition}\label{RestrictingR}
        Let $\fR$ be a complete \zroot\ system. Then $\fR$ contains
        a distinguished subsystem with group $G(\fR)$.
\end{proposition}

\begin{proof}
        Denote by $\fR_0$ the set of distinguished roots of $\fR$.    
        Since $\fR$ is complete, any distinguished
        reflection of $G(\fR)$ is of the form $s_\frr$ for $\frr\in\fR_0$.
        
        Thus the set of reflecting lines of $\fR_0$ is the same as the
        set of reflecting lines of $\fR$, and this proves that condition (RS$_I$)
        (see Definition \ref{Zkrootsystems}) is satisfied for $\fR_0$. 
        
        Condition (RS$_{III}$) on $\fR_0$ is inherited from $\fR$.
        
        It remains to check that $\fR_0$ is stable under $G(\fR_0)$.
        Notice that $G(\fR_0) = G(\fR)$. Thus (RS$_{II}$) follows from the fact that 
        the image of a distinguished
        root by an element of $G(\fR_0)$ is still distinguished.

        Now apply Proposition \ref{ReducingR} to get a reduced subsystem of
        $\fR_0$, and we get a distinguished root system.
\end{proof}

         Let $\frr =(I_\frr, J_\frr, \zeta)$ be a \zroot\ 
        with $\zeta=\exp(\frac{2 \pi i}{d})$ with $d>2$.
        For $1 \leq i < d$, we denote by $\frr^i$ the root: 
        $$
                \frr^i:= \left( \frac{1-\zeta^i}{1-\zeta}  I_\frr , \, J_\frr, \, \zeta^i \right)
                \,,
                $$
                \index{ri@$\frr^i$}
		which has the property that $s_\frr^i=s_{\frr^i}$.

        An incomplete root system can be augmented 
        by adjoining all of the $\frr^i$ to obtain a complete root system.
        
\begin{proposition}\label{completingdistinguishedR}
        Let $\fR$ be a distinguished \zroot\ system for $V$. Denote by $\widehat{\fR}$ 
        the set of roots obtained by adjoining 
        all roots of the  form $\frr^i$ to $\fR$. 
        Then $\widehat{\fR}$ is
        a complete reduced \zroot\ system for $V$, whose set of distinguished roots
        is $\fR$.
\end{proposition}

\begin{proof} 
        Since $\fR$ is finite and only a finite number of roots are added to form $\widehat{\fR}$, 
        the condition \textbf{(RS$_{I}$)} is immediately satisfied for $\widehat{\fR}$.
        
        Let us check \textbf{(RS$_{II}$)}. Since all the reflections $s_\frr'$ for 
        $\frr'\in\widehat{\fR}$ belong to $G(\fR)$, it suffices to check that $\widehat{\fR}$
        is stable under all $s_\frr$ for $\frr \in \fR$. This results from the fact that
        if $s(\frr_1) = \frr_2$ then $s(\frr_1^i) = \frr_2^i$.

        Finally, consider two roots $\frr_1^i$ and $\frr_2^j$ in $\widehat{\fR}$, 
        where $\frr_1$ and $\frr_2$ are in $\fR$. Then 
        $$
                \frn{\frr_1^i}{\frr_2^j}
                = \left \langle \left(\frac{1-\zeta_1^i}{1-\zeta_1} \right) I_{\frr_1}, J_{\frr_2} \right \rangle 
                = \left(\frac{1-\zeta_1^i}{1-\zeta_1} \right) \left \langle I_{\frr_1}, J_{\frr_2} \right \rangle 
                 \,.
        $$
        Since $\frn{\frr_1}{\frr_2}\in \BZ_k$ by \textbf{(RS$_{III}$)} for $\fR$, 
        and since $\frac{1-\zeta_1^i}{1-\zeta_1}$ is always integral, 
        $\frn{\frr_1^i}{\frr_2^j} \in \BZ_k$ as well, 
        verifying the last condition \textbf{(RS$_{III}$)} for $\widehat{\fR}$.
        
        The system $\widehat{\fR}$ is reduced since 
        $s_{\frr^i} = s_\frr^i$.
\end{proof}

\subsubsection*{Dual root system, irreducible root systems}\hfill
\smallskip

        Recall   (see Definitions \ref{reflection and dual}(2)) that  
        if $\frr = (I,J,\zeta)$ is a \zroot\ in $V$, its 
        \emph{dual root} is the \zroot\ in $W$ defined by
        $ \frr^\vee := (J,I,\zeta)$.
     
\begin{lemmadef}\label{dualrootsystem}
        If $\fR$ is a \zroot\ system in $V$, the set
        $$
                \fR^\vee := \{ \frr^\vee\mid\frr\in \fR\}
        $$
        is a \zroot\ system in $W$, called 
        the \emph{dual root system} of $\fR$. \index{Dual root system}\index{Rvee@$\fR^\vee$}
\end{lemmadef}

\begin{proof} 
        The fact that $\fR^\vee$ is a \zroot\ system
        follows directly from the definition and the
        equality:
        $$
                \frn{\frr_1^\vee}{\frr_2^\vee} = \langle J_1,I_2 \rangle
                = \langle I_2,J_1\rangle^*=\frn{\frr_2}{\frr_1}^*
                \,.
        $$
\end{proof}

        Recall (Definition \ref{sim} (\ref{def:irredRS})) that a set of reflections is 
        \emph{irreducible} if it consists of a single equivalence class with respect to 
        the  closure of the ``is not orthogonal" relation $\sim$.
        
\begin{definition}\label{irreduciblerootsystem}
        Let $\fR$ be a set of \zroot s, and $S_\fR$ be the corresponding set of reflections. 
        Then $\fR$ is said to be \emph{irreducible} if $S_\fR$ is irreducible.
\end{definition}

        So a set of roots $\fR$ is irreducible if
        for every pair of roots $\frr$ and $\frr'$, there is a sequence 
        $\frr=\frr_{i_0}, \frr_{i_1},\ldots,\frr_{i_p}=\frr'$ 
        such that each adjacent pair of roots in the
        sequence is not orthogonal -- that is, 
        $\frn{\frr_{i_j}}{\frr_{i_j+1}} \neq 0$.

\subsection{Root systems and parabolic subgroups}\habel{rootandpara}
\smallskip

        In this subsection,  $\fR$ denotes a \zroot\ system, and  $G := G(\fR)$.
        
        Let $F$ be a flat of $G$ in $V$, that is, an intersection of reflecting hyperplanes of
        $G$ in $V$. Recall that we denote by $\Arr_F(G)$ 
        the family of all reflecting hyperplanes of
        $G$ containing $F$, so that
        $
                F = \bigcap_{H\in\Arr_F(G)} H
                \,.
        $
        
        For $H\in\Arr_F(G)$, we denote by $L_H$ the reflecting line in $V$ attached to $H$
        (see Proposition \ref{reflectingpairs}), by $M_H$ the orthogonal of $H$ in $W$ 
        (a dual reflecting line for $G$ in $W$),
        and by $K_H$ the corresponding dual reflecting hyperplane in $W$.
        
        We set
        $$
                V_F := \sum_{H\in \Arr_F(G)} L_H
                \,\text{ and }\,
                W_F := \sum_{H\in\Arr_F(G)} M_H
                \,.
        $$      
        The Hermitian pairing between $V$ and $W$ restricts to a Hermitian pairing
        between $V_F$ and $W_F$. 
        
        Let $C_G(F)$ be
        the corresponding parabolic subgroup of $G$, the fixator of $F$.
        We recall (see Theorem \ref{para} above) that 
\begin{itemize}
        \item
                $C_G(F)$ is generated by those reflections whose reflecting hyperplanes
                belong to $\Arr_F(G)$, 
        \item
                $F$ is the set of fixed points of $C_G(F)$ in $V$,
\end{itemize}
        and $C_G(F)$ is naturally identified with a subgroup of $\GL(V_F)$ generated
        by reflections.
  
        Let $\frr = (I,J,\zeta) \in \fR$ such that $s_\frr \in C_G(F)$. Then $s_\frr$ is
        a reflection in its action on $V_F$, and
        $\frr$ may be viewed as a \zroot\ for $(V_F,W_F)$, since
        $I \subset V_F$ and $J \subset W_F$.

\begin{proposition}\label{pararootsystem}
        Let $\fR$ be a \zroot\ system in $V$, and let $F$ be a flat of $G(\fR)$ in $V$.
\begin{enumerate}
        \item   
                The set
                $$
                        \fR_F := \{ \frr \,\mid\, s_\frr \in C_G(F) \}
                $$
                is a \zroot\ system for the parabolic subgroup $C_G(F)$ viewed as a reflection group
                acting on $V_F$.
        \item  
                If $\fR$ is complete then $\fR_F$ is a complete root system for $C_G(F)$.
        \item  
                If $\fR$ is distinguished then $\fR_F$ is a distinguished root system for $C_G(F)$.
\end{enumerate}
\end{proposition}

\begin{proof}
        (1) It suffices to check that $\fR_F$ is stable under the action of $C_G(F)$.
        It is enough to check that for $\frr \in \fR_F$ and $t\in C_G(F)$, we have
        $t\cdot\frr \in \fR_F$. But $s_{t\cdot\frr} = ts_\frr t\inv$, which fixes $F$.
        
        Completeness and distinguishedness (items (2) and (3)) are inherited directly from $\fR$.
\end{proof}

\subsection{Root lattices, root bases}\hfill

\begin{definition}\label{def:RootLattice}
        Let $\fR = \{\, \frr = (I_\frr,J_\frr,\zeta_\frr)\,\}$ be a \zroot\ system.
        \index{Root lattice}\index{Coroot lattice}
        \index{QR@$Q_\fR$}
        \index{QRcheck@$Q_\fR^\vee$}
\begin{enumerate}
        \item
                The \emph{root lattice} $Q_\fR$ and the \emph{coroot lattice} $Q_\fR^\vee$ 
                are defined by
                $$
                        Q_\fR := \sum_{\frr\in \fR} I_\frr
                         \quad\text{and}\quad
                        Q_\fR^\vee := \sum_{\frr\in \fR} J_\frr
                        \,.
                $$
        \item\index{Weight lattice}\index{Coweight lattice}\index{PR@$P_\fR$}
                The \emph{weight lattice} $P_\fR$ and the \emph{coweight lattice}
                $P_\fR^\vee$ are the dual
                of $Q_\fR^\vee$ and $Q_\fR$ respectively, \ie\
                $$
                \begin{aligned}
                         &P_\fR := \{\, x\in V\mid \scal{x}{Q_\fR^\vee} \subseteq \BZ_k\} \\
                        &P_\fR^\vee := \{\, y\in W\mid \scal{y}{Q_\fR} \subseteq \BZ_k\} \\
                \end{aligned}
                $$
\end{enumerate}
\end{definition}

        The following properties are straightforward.
\begin{itemize}
        \item
                $Q_\fR \subseteq P_\fR$ and $Q_\fR^\vee \subseteq P_\fR^\vee$.
        \item
                $Q_{\fR^\vee} = Q^\vee_\fR\,$, $P_{\fR^\vee} = P^\vee_\fR$.
\end{itemize}

\begin{definition}\label{DefAutRootSys}
         The group of automorphisms of a \zroot\ system $\fR$
        denoted $\Aut(\fR)$\index{AR@$\Aut(\fR)$}, is the group of
        all $g\in \GL(V)$ such that $g(\fR) = \fR$.
        
        In other words, 
        $$
        \Aut(\fR) = \{g \in \GL(V) \mid (I,J,\zeta) \in \fR 
        \Rightarrow (g(I),g^\vee(J),\zeta) \in \fR 
         \}
                \,.
        $$
\end{definition}
 
        If $g \in \Aut(\fR)$, $g$ conjugates the reflection $s_\frr$ defined by a root 
        $\frr \in \fR$
        to the reflection $s_{g(\frr)}$ defined by $g(\frr)$, hence 
        $$
                G(\fR) \vartriangleleft \Aut(\fR)
                \,.
        $$

\begin{proposition}\habel{ActionOfARonGR}\hfill
\begin{enumerate}
        \item
                The lattices $Q_\fR$, $P_\fR$, $Q_\fR^\vee$, $P_\fR^\vee$
                are all $\Aut(\fR)$-stable
                finitely generated projective $\BZ_k$-submodules
                of $V$ and $W$ respectively.
        \item
                The group $G(\fR)$ acts trivially on $P_\fR/Q_\fR$ and on 
                $P_\fR^\vee/Q_\fR^\vee$, hence the group $\Aut(\fR)/G(\fR)$ acts
                on these quotients.
        \item
                The Hermitian pairing $V\times W \ra k$ induces 
                a non-degenerate pairing of $\BZ_k(\Aut(\fR)/G(\fR))$--modules
                $$
                        P_\fR/Q_\fR \times P_\fR^\vee/Q_\fR^\vee
                        \longrightarrow k/\BZ_k
                        \,.\
                $$
\end{enumerate}
\end{proposition}

\begin{proof}
        Assertion (1) is clear.
        Assertion (2) results from the following lemma, which gives an alternative
        description of the reflection associated with a \zroot.
        
\begin{lemma}\label{descriptionofsr}
        Let $\frr = (I,J,\zeta)$ be a \zroot.
        Assume that $(\al_i)_{i\in E}$ and $(\be_i)_{i\in E}$ are finite families of elements
        of $I$ and $J$ respectively such that 
        $$
                \sum_{i\in E} \scal{\al_i}{\be_i}= 1-\zeta
                \,.
        $$
        Then, for all $v\in V$,
        $$
                s_\frr(v) = v - \sum_{i\in E} \scal v{\be_i} \al_i
                \,.
        $$
\end{lemma}

\begin{proof}
        For all $i\in E$, we have (see Proposition~\ref{computereflection})
        \begin{align*}
                &s_\frr(v) = v - \dfrac{\scal v{\be_i}}{\scal {\al_i}{\be_i}}(1-\zeta) \al_i
                \,,
                \text{ hence } \\
                &\scal {\al_i}{\be_i} s_\frr(v) = \scal {\al_i}{\be_i} v - \scal v{\be_i}(1-\zeta) \al_i
                \,.
        \end{align*}
        Summing the last equality over $E$, then simplifying by $(1-\zeta)$, gives the
        expected formula.
\end{proof}

        Now (with the same notation as in the above lemma),
        for all $v \in P_\fR$, $\scal v{\be_i} \in \BZ_k$, hence 
        $\scal v{\be_i} \al_i \in Q_\fR$, 
        which shows that $s_\frr$ acts trivially
        on $P_\fR/Q_\fR$.
        
        Assertion (3) is immediate.
\end{proof}
\smallskip
 
        Notice also that for $\fa$ a fractional ideal, we have
        $$
        \begin{array}{ll}
                &Q_{\fa\cdot \fR} = \fa Q_\fR \,\,,\,\, Q_{\fa\cdot \fR}^\vee =
                        \fa\tinv Q_\fR^\vee \\
                &P_{\fa\cdot \fR} = \fa P_\fR \,\,,\,\, P_{\fa\cdot \fR}^\vee =
                        \fa\tinv P_\fR^\vee
                \,.
         \end{array}
        $$

 \begin{definition} Let $\fR_1$ and $\fR_2$ be two \zroot\ systems.
        \begin{enumerate}
        \item
                Say that $\fR_1$ and $\fR_2$ are of the same  
                \index{Genus of root systems}\emph{genus} if there
                exists a fractional ideal $\fa$ such that
                $$
                          \fR_2 = \fa\cdot \fR_1 := \{\,\fa\cdot\frr_1\,\mid\,(\frr_1\in \fR_1)\,\} .
                $$
        \item
                Say that $\fR$ and $\fR'$ are \index{Lattice equivalence of root systems}
                \emph{lattice equivalent}  if there exists a fractional ideal $\fa$
                such that $Q_{\fR'} = \fa Q_\fR$ and $Q_{\fR'}^\vee = \fa\tinv Q_{\fR}^\vee$.
        \end{enumerate}
\end{definition}

        Notice that if two root systems have the same genus, then they are
        lattice equivalent.
        The converse is false, as seen in the following example.

\begin{example}
        Consider $k := \BQ(\zeta_3)$, hence $\BZ_k = \BZ[\zeta_3]$ (a principal ideal domain). 
        Set $W=V=k$ and $G := \bmu_3$ acting on $V$ by multiplication. 
        Now $\ideal{(1-\zeta_3^2)}=\ideal{(1-\zeta_3)}$, 
        since $1-\zeta_3^2 = -\zeta(1-\zeta_3)$.
 
        It is easily checked (see Section~\ref{cyclic} below for the 
        general case of the cyclic groups)
        that there are exactly \emph{three} genera of complete root systems for $G$ as 
        follows: let $\fp = \ideal{(1-\zeta_3)} = \ideal{(1-\zeta_3^2)}$, then 
        $$
        \begin{array}{lll}
                &\fR_{1,1} := & \left\{ \Big( \ideal{},\fp,\zeta_3 \Big) \,,\, 
                   \Big( \ideal{},\fp,\zeta^2_3 \Big)   \,\right\}  \\
                                \\
                &\fR_{1,\fp} := & \left\{ \Big( \ideal{},\fp,\zeta_3 \Big)\,,\,
                  \Big( \fp,\ideal{},\zeta^2_3 \Big) \,\right\}  \\ 
                                \\
                &\fR_{\fp,1} := & \left\{ \Big( \fp,\ideal{},\zeta_3 \Big)\,,\,
                  \Big( \ideal{},\fp,\zeta^2_3 \Big)     \,\right\}  \\ 
        \end{array}
        $$
        Now $Q_{\fR_{1,\fp}} = Q_{\fR_{\fp,1}} = \ideal{}$, and
        $P_{\fR_{1,\fp}} = P_{\fR_{\fp,1}} = \ideal{}$, so these two distinct 
        (with respect to genus) root systems are in fact lattice equivalent.  
 \end{example}

 \begin{definition}
 \label{Def_rootBasis_etc}
        A subset $\Pi$ of elements of $\fR$ is said to be 
        \begin{itemize}
        \item a  \emph{set of root generators} \index{Root generators} if 
        $  Q_\fR = \sum_{\frr\in \Pi} I_\frr$
        \item a \index{Root lattice basis} \emph{root lattice basis} if
        $ Q_\fR = \bigoplus_{\frr\in \Pi} I_\frr$
 \item a \index{Root basis} \emph{root basis} if
\begin{enumerate}
        \item
                $
                                Q_\fR = \bigoplus_{\frr\in \Pi} I_\frr
                                \ and 
                $
        \item
                the family $(s_\frr)_{\frr\in \Pi}$ generates $G(\fR)$.
\end{enumerate}
\end{itemize}
\emph{Coroot generators, coroot lattice bases, coroot bases} 
\index{Coroot bases} are defined analogously.
\end{definition}

\begin{example}
        As above, let us choose $k := \BQ(\zeta_3)$, hence $\BZ_k = \BZ[\zeta_3]$, 
        $V=W = k$ and
        $G := \bmu_3$ acting on $V$ by multiplication.
        Then each of the root systems $\fR_{1,1}$, $\fR_{1,\fp}$, $\fR_{\fp,1}$ 
        contains a root basis -- for example, $\Pi= \left\{ \Big( \ideal{},
          \ideal{(1-\zeta_3)},\zeta_3 \Big) \right\} $. 
         Indeed, $\Pi$ is both a root basis and a coroot basis of $\fR_{1,1}$.
        However $\fR_{1,\fp}$ and $\fR_{\fp,1}$ do not contain a subset which is
        \emph{simultaneously} a root basis and a coroot basis.
\end{example}

On the other hand, a distinguished root system for a well generated group 
always contains a subset which is simultaneously a root basis and a coroot basis:

\begin{proposition}\label{prop:RootBasis}
        Let $\fR$ be a distinguished \zroot\ system.
        Let $\Pi$ be a subset of $\fR$ such that $\{s_\frr \mid \frr\in \Pi\}$
        generates $G(\fR)$. Then
\begin{enumerate}
        \item
                whenever $\frr \in \fR$, there exist
                $\frr_0,\frr_1,\dots,\frr_m \in \Pi$ such that
                $
                        \frr = (s_{\frr_m}\cdots s_{\frr_1})\cdot{\frr_0}
                        \,,
                $
        \item $\Pi$ is a set of root generators and a set of coroot generators,
        \item
                if $\Pi$ consists of principal \zroot s, then
                $\fR$ is principal,
        \item
                if $|\Pi| = \dim V$, then $\Pi$ is a root basis and a coroot basis
\end{enumerate}
\end{proposition}

\begin{proof}
        Write $G$ for $G(\fR)$. 
        Notice that it is enough to prove the results concerning roots: the ones concerning
        coroots follow by considering the contragredient operation 
        $g\mapsto g^\vee$ of $G$ on $W$.
 
        (1)
        Let $\frr \in \fR$. By Corollary \ref{distinguishedgenerators}, $s_\frr$ is conjugate
        to some $s_{\frr_0}$ for $\frr_0 \in \Pi$. Thus there exist
        $\frr_1,\dots,\frr_m \in \Pi$ with
        $
                s_\frr = s_{\frr_m}\cdots s_{\frr_1}s_{\frr_0}s_{\frr_1}\inv\cdots s_{\frr_m}\inv
                \,.
        $
        Since $\fR$ is stable under $G$, and since 
        $ws_{\frr_0}w\inv = s_{w.s_{\frr_0}}$, the above 
        equality implies that
        $
                \frr = (s_{\frr_m}\cdots s_{\frr_1})\cdot{\frr_0}
                \,.
        $
        
        (2)
        Suppose that for $\frr, \frr_0,\dots,\frr_m \in \Pi$ are as in (1).
        By Lemma~\ref{reflectionideal},
        $$
                s_{\frr_1}(I_{\frr_0}) \subset I_{\frr_0} + I_{\frr_1}
                \,,
        $$
        so
        $$
                s_{\frr_2}s_{\frr_1}(I_{\frr_0}) \subset s_{\frr_2}(I_{\frr_0}) +  s_{\frr_2}(I_{\frr_1})
                \subset
                I_{\frr_0} + I_{\frr_1} + I_{\frr_2}
                \,,
        $$
        and an iteration shows that
        $$
                 I_\frr = (s_{\frr_m}\cdots s_{\frr_1})(I_{\frr_0})
                  \subset I_{\frr_0} + I_{\frr_1} +\dots + I_{\frr_m}
                 \,.
        $$
        
        Part (3)
        results from assertion (1) and from the remark that, for any $g \in \GL(V)$
        and $\frr$ a principal \zroot, $g(\frr)$ is still principal.
 
        Item (4) is clear.
 \end{proof}

 \begin{proposition}\label{root and coroot basis}
 Let $\fR$ be a distinguished \zroot\ system. If $\Pi$ is a subset of $\fR$
 such that $|\Pi|=\dim V$ and the family $(s_\frr)_{\frr\in\Pi}$ generates
 $G(\fR)$, then $\Pi$ is a root basis and coroot basis of $\fR$.
 \end{proposition}
 \begin{proof}
 This results from Proposition \ref{prop:RootBasis} (4).
 \end{proof}
        
        Note that root bases only exist when $G(\fR)$ is well generated.

\begin{remark}
        When $\BZ_k$ is not a P.I.D., a root basis does not necessarily provide a basis 
        of $Q_\fR$ as a $\BZ_k$-module.
        Nevertheless, we shall see later 
        (Theorem~\ref{thm:freelattices}, see also \cite[Corollary 13]{nebe}) that 
        every reflection group 
        has at least one principal \zroot\ system, 
        and the root lattice of a principal \zroot\ system 
        is always a \emph{free $\BZ_k$-module.}
\end{remark}

\subsection{Example: the Weyl group of type $B_2$}\habel{subsecB2}
\smallskip

        Let $k$ be a number field. Set $V = k^2$  
        with canonical basis $\{e_1, e_2\}$
        and $W = k^2$ 
        with canonical dual basis $\{f_1,f_2\}$.
        The Weyl group of type $B_2$, denoted $G$, may be considered to be 
        the subgroup of $\GL(V)$
        generated by $S=\{s,t\}$ where $s$ and $t$ are the automorphisms of $V$
        corresponding respectively to the following
        matrices on the basis $\{e_1, e_2\}$:
        $$
                 s := \begin{pmatrix} -1&0\\0&1\end{pmatrix}
                \quad\text{and}\quad
                t := \begin{pmatrix} 0&1\\ 1&0 \end{pmatrix}
                \,.
        $$
        The corresponding reflecting lines are
        $$
        \left\{
        \begin{array}{lll}
                L_s = k v_s \text{ with } v_s = e_1 
                        \quad&\text{and}\quad
                &L_t =k v_t \text{ with } v_t = e_2-e_1 \,,             \\
                        M_s = k v_s ^\vee \text{ with } v_s^\vee = 2f_1
                        \quad&\text{and}\quad
                &M_t = k v_t^\vee \text{ with } v_t^\vee =f_2-f_1 \,.
        \end{array}
        \right.
        $$
        
        The orbits under $G$ of the following root bases (corresponding to 
        generators $s$ and $t$ in that order) are \zroot\ systems 
        corresponding to the types $B_2$ and $C_2$ respectively:
        $$
                \Pi(B_2) :=
                \left\{\,\Big(\ideal{}v_s,\ideal{}v_s^\vee,-1\Big)
                \,,\,
                \Big(\ideal{}v_t,\ideal{}v_t^\vee, -1\Big)\,\right\}
                \,,
        $$
        $$
                \Pi(C_2) :=
                \left\{\,\Big(\ideal2v_s,\ideal{\frac12} v_s^\vee,-1\Big)
                \,,\,
                \Big(\ideal{}v_t,\ideal{}v_t^\vee, -1\Big)\,\right\}
                \,.
        $$
        Swapping $V$ and $W$, and $s$ and $t$, defines an isomorphism between 
        the coroot system of type $B_2$ and the root system of type $C_2$, 
        and vice versa. We say that they are mutually dual root systems. 
        
        It is  immediate to check that the element $\phi \in \GL(V)$ defined by
        $$
                \phi : 
                \left\{
                \begin{aligned}
                        &e_1 \mapsto -e_1+e_2 \,,\\
                        &e_2 \mapsto e_1+e_2\,,
                \end{aligned}
                \right.
        $$
        that is, the automorphism of $V$ with matrix 
        $\begin{pmatrix}-1&1\\1&1 \end{pmatrix}$
        on the basis $(e_1,e_2)$, has the following properties:
        \begin{enumerate}
                \item
                        $\phi^2 = 2\Id_V\,,$
                \item
                        it swaps $s$ and $t$ (by conjugation), 
                \item
                        it sends ${\Pi}(B_2)$ onto $\Pi(C_2)$ and $\Pi(C_2)$ onto $2\Pi(B_2)$,
                        hence swaps $\fR(B_2)$ and $\fR(C_2)$, up to genus.
        \end{enumerate}
        That is, the automorphism denoted $\lexp2B_2$ of $G$ swaps, up to genus,
         $\fR(B_2)$ and $\fR(C_2)$, which are thus isomorphic (up to genus).

\begin{lemma}\hfill
\begin{enumerate}
        \item
                The following assertions are equivalent.
                \begin{itemize}
                        \item[(i)]
                                There exists a \zroot\ system with group $G$ which is stable
                                by the automorphism $\phi$ (up to genus),
                        \item[(ii)]
                                there exists a principal ideal $\fa$ of $\BZ_k$ such that
                                $\fa^2 = \ideal 2$.
                \end{itemize}
        \item
                If $\fa = \ideal (a)$ is such that $a^2 = 2u$ with $u\in\BZ_k^\times$,
                we set:
                $$
                        \Pi_\fa := 
                        \left\{\,\Big(\ideal{\fa}v_s,\ideal{\fa\tinv}v_s^\vee,-1\Big)
                        \,,\,
                        \Big(\ideal{}v_t,\ideal{}v_t^\vee, -1\Big)\,\right\}
                        \,,
                $$
                and denote by $\fR_\fa$ the orbit of $\Pi_\fa$ under the group generated by the 
                reflections $s_\frr$ for $\frr \in \Pi_\fa$. Then
                \begin{enumerate}
                        \item
                                $\fR_\fa$ is a \zroot\ system with group $G$,
                        \item
                                the flips between $V$ and $W$ and between $s$ and $t$ define 
                                an isomorphism between $\fR_\fa$ and its coroot system
                                (thus, $\fR_\fa$ is ``self-dual''),
                        \item
                                $\phi(\fR_\fa) = \fa \fR_\fa$ (thus $\fR_\fa$ is stable
                                by $\phi$ up to genus).
                \end{enumerate}
\end{enumerate}
\end{lemma}
That is, the root system $\fR_\fa$ affords an automorphism corresponding to
the automorphism $\lexp2B_2$ of $G$.
\begin{proof}
        The proof of (2) is easy and left to the reader. Moreover, (2) implies the implication
        (ii)$\Rightarrow$(i) of (1).

        Let us prove (1), (i)$\Rightarrow$(ii). 
%
%
        We may assume that $\fR$, the \zroot\ system for $G$ 
        which is stable (up to genus) under $\phi$, has root basis:
        $$
                \Pi :=
                \left\{\,\Big(\fa_s v_s,\fa_s\tinv v_s^\vee,-1\Big)
                        \,,\,
                        \Big(\fa_t v_t,\fa_t\tinv v_t^\vee, -1\Big)\,
                \right\}
        $$
        for some fractional ideals $\fa_s, \fa_t$.
        Now $\phi(\fR) = a\fR$ for some $a\in k^\times$, so 
        $$
                \fa_s v_t =  a \fa_t v_t
                \quad\text{and}\quad
                2 \fa_t v_s =  a \fa_s v_s
                \,,
        $$
        from which we deduce that 
        $
                a^2 \fa_t = a \fa_s = 2 \fa_t . 
        $
 		Multiplication by $\fa_t\inv$ gives
		$
				a^2 \BZ_k = 2 \BZ_k . 
		$
\end{proof}

        For example, set $k = \BQ(i)\text{ or }\BQ(\sqrt{2})$.
        The ring $\BZ_k$ is a principal ideal domain. Setting
        $\fa := \BZ[i](1+i)\text{ or }\BZ[\sqrt{2}]\sqrt{2}$, then
        $\fa=\fa^*$ and $ \ideal2 = \fa ^2 $ is the decomposition of
        $\ideal2$ in $\BZ_k$.

        It is immediate to check that, if $\ideal 2 = \fa^2$,
        there are at least three genera of
        reduced \zroot\ systems
        for $G$ described as the orbits under $G$ of the following three pairs of
        roots:
        $$
        \begin{array}{ll}
        & \{\,(\ideal{}v_s,\ideal{}v_s^\vee,-1),(\ideal{}v_t,\ideal{}v_t^\vee,-1)\}\\
        & \{\,(\ideal2v_s,\ideal{\frac12}v_s^\vee,-1),(\ideal{}v_t,\ideal{}v_t^\vee,-1)\}\\
        & \{\,(\fa v_s,\fa\tinv  v_s^\vee,-1),(\ideal{}v_t,\ideal{}v_t^\vee,-1)\}.\\
        \end{array}
        $$

\subsection{Connection index}\hfill
\smallskip

        Let $\fR = \{ \frr = (I_\frr,J_\frr,\zeta_\frr) \}$ be a \zroot\ system.
        The \emph{characteristic ideal\/} (see for example \cite[2.3.4.2]{beijing})
        of the torsion $\BZ_k$-module $P_\fR/Q_\fR$ is defined by
        $$
                \bigwedge^r Q_\fR = \Ch(P_\fR/Q_\fR) \bigwedge^r P_\fR
                \,,
        $$
        where $r := \dim V$ (see \cite[\S4, n$^0$6]{boualg}).
        
\begin{remark}
        The characteristic ideal is the image in the group of fractional ideals of $\BZ_k$ of the divisor
        called ``contenu'' in \cite[\S4, n$^0$5, Definition 4]{boualg}.
\end{remark}

        The next definition is inspired by the definition given in \cite[chap. 6, no 1.9]{bou}.

\begin{definition}\label{def:connectionindex}
        The characteristic ideal of the torsion $\BZ_k$-module $P_\fR/Q_\fR$
        is called the \emph{connection index of the root system $\fR$}.
        \index{Connection index of $\fR$}
\end{definition}

\begin{theodef}\label{th:connectionindex}
        Let $(V,G)$ be an irreducible well generated reflection group. 
        The connection index
        of a distinguished \zroot\ system $\fR$ for
        $G$ does not depend on the choice of $\fR$,
        and is called the \emph{connection index} of $(V,G)$.
        \index{Connection index of $(V,G)$}
\end{theodef}

\begin{proof}[Proof of Theorem~\ref{th:connectionindex}]

        Let $r :=\dim V$. Let $\fR$ be a distinguished \zroot\ system for $G$.
        By item (4) of Proposition \ref{prop:RootBasis}, since $(V,G)$ is well generated,
        by Proposition \ref{root and coroot basis} there exists a set $\Pi$ of 
        $r$ roots such that 
        $
                Q_\fR = \bigoplus_{\frr\in\Pi} I_{\frr}
        $
        and 
        $
                Q_\fR^\vee = \bigoplus_{\frr\in\Pi} J_{\frr}
                \,.
        $
        
         For all $\frr\in\Pi$ write $I_\frr=\fa_\frr v_\frr$ and
         $J_\frr=\fa_\frr\tinv  w_\frr$ for some vectors $v_\frr$ and $w_\frr$
         with $\scal{v_\frr}{w_\frr}=1-\zeta_\frr$ and some fractional ideal
         $\fa_\frr$.
         Let $w'_\frr$ be the dual basis of $w_\frr$ and set $J'_\frr=\fa_\frr
         w'_\frr$. Then $\CP_\fR=\bigoplus_{\frr\in\Pi}J'_\frr$.

        Assume given another distinguished \zroot\ system $\fR'$
        associated with the same set of reflections. 
        For each $\frr \in \fR$, associated with the reflection $s_\frr$, let
        us denote by $\frr'$ the element of $\fR'$ associated with the same
        reflection $s_\frr$. Then, if 
        $
                \frr'= (I_{\frr'},J_{\frr'},\zeta_\frr)
                \,,
        $
        we have $I_{\frr'} = \fb_\frr I_\frr$ and $J_{\frr'} = \fb_\frr\tinv  J_\frr$
        for a fractional ideal $\fb_\frr$.
        
        Then
        $$
                Q_{\fR'} = \bigoplus_{\frr\in\Pi} \fb_\frr I_\frr
                \quad\text{ and }\quad
                P_{\fR'} = \bigoplus_{\frr\in\Pi} \fb_\frr J'_\frr
                \,,
        $$
        and
        $$
        \begin{aligned}
                &\bigwedge^r Q_{\fR'} 
                = \big( \prod_{\frr\in\Pi} \fb_{\frr}\big) \bigwedge^r Q_{\fR} \,,\\
                &\bigwedge^r P_{\fR'} 
                = \big( \prod_{\frr\in\Pi} \fb_{\frr}\big) \bigwedge^r P_{\fR} 
                \,.
        \end{aligned}
        $$
        This shows that 
        $$
                \Ch(P_{\fR'}/Q_{\fR'}) = \Ch(P_\fR/Q_\fR)
                \,,
        $$
        and ends the proof.
\end{proof}

\begin{remark}
        Let $V$ and $W$ be as above, such that $\dim V = r$.
        Let $(V,G)$ be a reflection group and
        let $s_1,\ldots,s_r$ be a set of reflections such that
        $V = \bigoplus_{i=1}^r L_{s_i}$ and $W = \bigoplus_{i=1}^r M_{s_i}$.
        For each $i = 1,\ldots r$, pick $v_i\in L_{s_i}$ and $v_i^\vee\in M_{s_i}$
        such that $\scal{v_i}{v_i^\vee} = 1 -\zeta_{s_i}$. Then the \emph{Cartan matrix}
        $(\scal{v_i}{v_i^\vee})_{i,j}$ depends only (up to conjugation by a diagonal matrix)
        on the choice of the set $s_1,\ldots,s_r$, hence its determinant depends only
        on such a choice.
        
        We shall see later (Proposition \ref{prop:connectionwellgen}) that 
        if moreover $s_1,\ldots,s_r$ generate $G$, that determinant generates
        (as an ideal) the connection index, hence in particular it does not
        depend (up to a unit) on the choice of the
        generators $s_1,\ldots,s_r$.
\end{remark}
        
\section{\red{The cyclic groups}}\label{cyclic}

\subsection{Generalities}\hfill
 
        As an introduction, let us consider the case of the cyclic group. 
        Let $k$ be a finite extension of $\BQ$ which contains $\bmu_d$, 
        also viewed as a one dimensional vector space $V$
        over itself; it is paired with $W=k$ by $\scal{a}{b} := ab^*$.
  
        Let $\zeta := \exp(2\pi i/d)$, and 
        let $G := \bmu_d = \{1,\zeta,\dots,\zeta^{d-1}\}$ be the cyclic subgroup 
        of $\BC^\times$ of order $d$.
        We let $G$ act on $k$ by multiplication.

\begin{proposition}\label{rootcyclic}\hfill
\begin{enumerate}
        \item
                Whenever $\fF := (\fa_1,\fa_2,\dots,\fa_{d-1})$ is a family of ideals of 
                $\BZ_k$ such that,
                for each $j$ ($1\leq j\leq d-1$), $\fa_j$ divides
                $\ideal{(1-\zeta^j)}$,
                then the set
                $$
                        \fR(\fF) := \{ (\fa_j,(1-\zeta^j)^*\fa_j\tinv ,\zeta^j)\,\mid\,(1\leq j\leq d-1) \}
                $$
                is a complete reduced \zroot\ system for $G = \bmu_d$.
        \item
                The family $(\fR(\fF))$ where $\fF = (\fa_1,\fa_2,\dots,\fa_{d-1})$
                runs over the families as above such that $\fa_1,\fa_2,\dots,\fa_{d-1}$
                are relatively prime,
                is a complete set of representatives for the genera of
                complete reduced \zroot\ systems for $G$. 
\end{enumerate}
\end{proposition}

\begin{proof}
        The assertion (1) is trivial. Let us prove (2).
 
        According to Lemma \ref{lemma:reflection=genus}(1) a \zroot\
        for a reflection of $G$ is a triple $(\fa v,\fa\tinv w,\zeta^j)$ where
        $v \in V=\BC$ and $w \in W=\BC$ and $\scal vw=1-\zeta^j$.
        We will write this $(\fa_j,\fb_j,\zeta^j)$ where we have set
        $\fa_j=\fa v$ and $\fb_j=\fa\tinv w$ and the condition becomes
        $\scal{\fa_j}{\fb_j}=\fa_j\fb_j^*=\ideal{(1-\zeta^j)}$. This implies that
        $\fb_j = ((1-\zeta^j)\fa_j\inv)^* =(1-\zeta^j)^*\fa_j\tinv$.

         According to definitions \ref{Zkrootsystems}, \ref{variousdefs}.\ref{def:reduced} and
        \ref{variousdefs}.\ref{def:complete} a complete and reduced \zroot\ 
        system for $G$ is a set of roots as above
        $$
                \fR = \{ (\fa_j,\fb_j,\zeta^j)\,\mid\, (0<j<d)\}
                \,,
        $$
         subject to the integrality condition that for $1\leq j,k\leq d-1$, we have
        $$
                \fa_j \fb_k^* \subseteq\BZ_k
                \,.
        $$

        In other words, 
        $$
                 \fR = \{ (\fa_j,(1-\zeta^{-j})\fa_j\tinv ,\zeta^j)\,\mid\, (1\leq j\leq d-1)\}
                \,,
        $$
         where $ (1-\zeta^k) \fa_j \fa_k\inv \subseteq\BZ_k$; or equivalently:
        \begin{equation}
        \tag{$\fR_{j,k}$}
                (1-\zeta^k) \fa_j \subseteq \fa_k\BZ_k \subseteq \fa_k
                \,.
         \end{equation}
 
        Since we want to know only the genus of $\fR$, we may multiply
        the family $(\fa_j)_{1\leq j\leq d-1}$ by an integral ideal so that the result is a family
        (still denoted by $(\fa_j)_{1\leq j\leq d-1}$) such that 
        \emph{all $\fa_j$ are integral, and they are relatively prime} (\ie\
        their sum is $\BZ_k$).
 
        Letting $j$ vary in the equality $\fR_{j,k}$ shows then that, for all $k$, $\fa_k$
        divides $\ideal{(1-\zeta^k)}$.
\end{proof}

\begin{corollary}\habel{cycliclattice}

\begin{enumerate}
        \item
                Each genus of complete root systems for the cyclic group
                contains a root system whose root lattice is $\BZ_k$.
        \item
                All root lattices for complete root systems
                for the cyclic group are lattice equivalent. 
\end{enumerate}
\end{corollary}

\begin{proof}
        Reasoning as in the proof of item (2) of Proposition \ref{rootcyclic},
        each genus contains
        $\fR(\fF) := \{ (\fa_j,\fb_j ,\zeta^{i_j})\}_j$
        where $\sum_i \fa_i= \BZ_k$, which shows item (1).

        Item (2) is an immediate consequence of item (1).
\end{proof}

        Every cyclic group has at least one complete root system which is principal 
        and has a root basis.
        For example, for the family of ideals $\fF$ for which 
        $\fa_1 = \dots = \fa_{d-1} = \BZ_k$, then
        the singleton
        $\{(\BZ_k,(1-\zeta)^*\BZ_k,\zeta)\}$ is a root basis.
        
        The next easy remark will be useful later.

\begin{proposition}\label{prop:bases&indexcyclic}
        Let $G = \bmu_d$ be the cyclic group of order $d$ and let $\zeta := \exp(2\pi i/d)$.
\begin{enumerate}
        \item
                The singleton $\{(\BZ_k,(1-\zeta)^*\BZ_k,\zeta)\}$ is a complete principal
                \zroot\ system, and provides a root basis.
        \item
                The connection index is the principal ideal
                $\ideal{(1-\zeta)}$. 
\end{enumerate} 
\end{proposition}

\subsection{The case of $G=\bmu_2$}\habel{bmu2}\hfill
\smallskip

        In that case, whatever the decomposition of $\ideal{2}$ into prime ideals of $\BZ_k$ is,
        the following description results from \ref{rootcyclic}.
         
\begin{lemma}\label{partcasemu2}
        Whatever the field $k$ is, 
        there is only one genus of reduced root system for 
        $\bmu_2$, represented by the singleton
        $$
                \{(\ideal{},\ideal{2},-1)\}
                \,,
        $$
        or by any singleton $\{(\fa,\fb^*,-1)\}$ where $\fa$ and $\fb$ are two integral
        ideals such that $\fa\fb = \ideal2$.
\end{lemma}  

\begin{remark}
        Motivated by questions related to exceptional Spetses (see \cite{spe}),
        we consider the following cases for the field $k$:
        $$
                \BQ\,,\,\BQ(i)\,,\,\BQ(\zeta_3)\,,\, \BQ(\zeta_{12})
                \,,\, 
                \BQ(\sqrt{-7})\,,\, \BQ(\sqrt{5},\zeta_3)\,,\,
                \BQ(\sqrt{-2},\zeta_3)
                \,.
         $$
\end{remark}

        The prime decomposition of the ideal $\ideal2$ in these cases is as follows:
        
\begin{lemma} \habel{factof2}\hfill
\begin{enumerate}
        \item
                For $k = \BQ$ or $k = \BQ(\zeta_3)$, the ideal $\ideal2$ is prime in $\BZ_k$.
        \item
                For $k = \BQ(i)$, $\BQ(\zeta_{12})$, or $\BQ(\sqrt{-2},\zeta_3)$, 
                $\ideal{2}=\fp^2$ for some integral ideal $\fp$.
        \item
                For $k = \BQ(\sqrt{-7})$ and $k = \BQ(\sqrt{5},\zeta_3)$, we have
                $\ideal2 = \fp\fp^*$
                where $\fp$ is a prime ideal in $\BZ_k$ and $\fp \neq \fp^*$.
 \end{enumerate}
\end{lemma}

\begin{proof}
        Note first that, for example by \cite[Theorem 11.1]{was}, the rings
        $\BZ[\zeta_3]$,
        $\BZ[i]$, $\BZ[\zeta_{12}]$
        considered below are principal ideal domains.
        
\begin{enumerate}
        \item
                For $k = \BQ(\zeta_3)$, we have $\BZ_k = \BZ[\zeta_3] = \BZ[X]/(X^2+X+1)$, and since
                $(X^2+X+1)$ is irreducible over $\BF_2$,  $\ideal2$ is prime in $\BZ_k$.
        \item
                For $k = \BQ(i)$, we have $\BZ_k = \BZ[i]$ and $\ideal2 = (\ideal{(1+i)})^2$.
                
\smallskip
\noindent
                For $k= \BQ(\zeta_{12})$, we have $\BZ_k = \BZ[\zeta_{12}]$, and
                $\ideal2 = (\ideal{(1+i)})^2$.

\smallskip
\noindent
                For $k = \BQ(\sqrt{-2},\zeta_3)$, we have $\BZ_k=
                \BZ[\sqrt{-2},\zeta_3]$ 
                (see Exercise 4.5.13 in \cite{mues})
                and
                $\ideal 2 = (\ideal{\sqrt{-2}})^2$. 
        \item
                For $k = \BQ(\sqrt{-7})$, $\BZ_k= \BZ\left[\frac{1+\sqrt{-7}}{2}\right]$, and
                $
                        2 = \frac{1+\sqrt{-7}}{2} \cdot \frac{1-\sqrt{-7}}{2}
                        \,,
                $
                hence $\ideal2$ is a product of two (conjugate) prime ideals: $\ideal2 = \fp\fp^*$,
                and $\fp \neq \fp^*$ since there is no solution in integers $a$ and $b$
                to the equation
                $$
                        \left(a+b\frac{1+\sqrt{-7}}2\right)\frac{1-\sqrt{-7}}2=\frac{1+\sqrt{-7}}2
                        \,.
                $$
\smallskip

\noindent
                Finally, consider the case 
                $k = \BQ(\sqrt{5},\zeta_3)$, with $\BZ_k = \BZ[\phi,\zeta_3]$, 
                where $\phi = \frac{1+\sqrt{5}}{2}$.
                Since the polynomial $X^2-X-1$ is irreducible over $\BF_2$, the prime 2 remains prime
                in $\BZ[\phi]$, and we know it remains prime over $\BZ[\zeta_3]$.
                Now 
                $2 = (1+\zeta_3\phi)(1+\zeta_3^2\phi)$,
                which shows that in $\BZ_k$ we have again $ \ideal2 = \fp\fp^*$,
                and $\fp \neq \fp^*$, since the only solution in rational numbers to
                $$
                        (a+b\phi+c\zeta_3+d\zeta_3\phi)(1+\zeta_3\phi)=1+\zeta_3^2\phi
                $$
                is $a=c=-d=\frac 12, b=-1$.
\end{enumerate}
\end{proof}
 
\subsection{A few particular cases for $G=\bmu_3$}\habel{bmu3}\hfill
\smallskip

        We first remark that $\frac{1-\zeta_3}{1-\zeta_3^2}=-\zeta_3$ is a unit
        thus the ideals $\ideal{(1-\zeta_3)}$ and $\ideal{(1-\zeta_3^2)}$ are the 
        same ideal $\fj=\fj^*$ and $\ideal3=\fj^2$.

        We want to consider the following cases for the field $k$: 
        $$
                \BQ(\zeta_3)\,,\, \BQ(\zeta_{12})\,,\, \BQ(\sqrt{-2},\zeta_3)
                \,.
        $$
 
 \subsubsection*{The cases where $k = \BQ(\zeta_3)$ and $k = \BQ(\zeta_{12})$}\hfill
\smallskip

        In these cases $\fj$ is prime
        (in the case $k = \BQ(\zeta_{12})$, this follows
        from \ref{zmn} and the fact that $3$ is a generator of $(\BZ/4\BZ)^\times$).
        It follows that there are two genera of reduced complete root systems, represented by
        the two systems
        $$
        \begin{aligned}
                &\fR_1 := \{\,(\ideal{},\fj,\zeta_3)\,,\,(\ideal{},\fj,\zeta^2_3) \,\} \\
                &\fR_2 := \{\,(\fj,\ideal{},\zeta_3)\,,\,(\ideal{},\fj,\zeta^2_3) \,\}
        \end{aligned}
        $$

\subsubsection*{The case $k = \BQ(\sqrt{-2},\zeta_3)$.}\hfill
\smallskip

        We first compute the factorisation of $\fj$ into prime ideals in $\BZ_k$.

{\small
\begin{itemize}
        \item
                We have
                        $
                                3 = (1+\sqrt{-2})(1-\sqrt{-2})
                                \,,
                        $
                 and since $\bmu( \BQ(\sqrt{-2})) = \{\pm 1\}$ (see for example \cite{sam}, 4.5)
                 it follows that the ideal $\ideal3$ decomposes in a product of two different 
                 prime ideals
                 in $\BZ[\sqrt{-2}]$:
                $$
                         \ideal3 = \ideal{(1+\sqrt{-2})} \cdot \ideal{(1-\sqrt{-2})}
                         \,.
                $$
                These ideals are different since
                the equation $(a+b\sqrt{-2})(1+\sqrt{-2})=1-\sqrt{-2}$ has no solution
                in integers $a,b$.
        \item
                We know that $\ideal3=\fj^2$ in $\BZ[\zeta_3]$ (see remarks above).
        \item
                 We also have:
                $$
                \left\{
                \begin{aligned}
                        &1+\sqrt{-2} = (1+\zeta_3\sqrt{-2})(-1-\zeta_3^2\sqrt{-2}) \\
                        &1-\sqrt{-2} = (1-\zeta_3\sqrt{-2})(-1+\zeta_3^2\sqrt{-2})
                         \,.
                \end{aligned}
                \right.
                $$
                Moreover we have the products of ideals
                $$
                \ideal{(1+\zeta_3\sqrt{-2})} \cdot \ideal{(1-\zeta_3\sqrt{-2})}=
                 \ideal{(-1-\zeta_3^2\sqrt{-2})} \cdot \ideal{(-1+\zeta_3^2\sqrt{-2})}= \fj
                $$
        \item
                 Notice that
                $$
                 \left\{
                \begin{aligned}
                        &1-\zeta_3^2\sqrt{-2} = (1-\zeta_3\sqrt{-2}
                        )(\zeta_3-\zeta_3^2-\sqrt{-2})\\
                        &1-\zeta_3\sqrt{-2} = (1-\zeta_3^2\sqrt{-2}
                        )(\zeta^2_3-\zeta_3-\sqrt{-2})
                \end{aligned}
                \right.
                $$
                which shows the equality of ideals
                $$
                        \fp:=\ideal{(1-\zeta_3\sqrt{-2})} =
                        \ideal{(1-\zeta_3^2\sqrt{-2})}
                        \,.
                $$
                Similarly we have the equality of ideals
                $$
                        \fq:=\fp^*=\ideal{(1+\zeta_3\sqrt{-2})} =
                        \ideal{(1+\zeta_3^2\sqrt{-2})}
                        \,.
                $$
                hence one has the following decomposition of $\ideal3$ into products
                of prime ideals in 
                $\BZ[\zeta_3,\sqrt{-2}]$:
                $$
                        \ideal3 = 
                         \underbrace{\fp\fq}_{= \,\fj}
                        \,\,
                        \underbrace{\fp\fq}_{=\, \fj}
                         \,.
                $$
                and $\fp\ne\fq$ since $\fq^2=\ideal{(1+\sqrt{-2})}$ and 
                $\fp^2=\ideal{(1-\sqrt{-2})}$ are different ideals.
 \end{itemize}
}

\noindent
        Now that we have the factorization of $(3)$ we can determine the reduced 
        complete root systems. Up to genus, they are of the form 
        $$
                \fR_{\fa_1,\fa_2} := \{\, (\fa_1, \fj\fa_1\tinv,\zeta_3)\,,\,
                (\fa_2, \fj\fa_2\tinv,\zeta_3^2) \,\} 
                \,,
        $$
        where the ideals $\fa_1$ and $\fa_2$ are integral, coprime, and divide
        $\fj=\fp\fq$.
        It is an elementary arithmetic exercise to deduce now the following classification.

\begin{proposition}\label{9genera}
        There are
        9 genera of reduced complete $\BZ[\zeta_3,\sqrt{-2}]$-root
        systems for $\bmu_3$, represented by the following list:
{\small
        $$
        \begin{array}{lll}
                &\fR_{1,1} = \{\,(1,\fp\fq,\zeta_3 )\,,\,( 1,\fp\fq,\zeta_3^2) \,\}\,,               \\
                &\fR_{\fp,\fq} = \{\,(\fp,\fp,\zeta_3 )\,,\,( \fq,\fq,\zeta_3^2) \,\}\,,
                &\fR_{\fq,\fp} = \{\,(\fq,\fq,\zeta_3 )\,,\,( \fp,\fp,\zeta_3^2) \,\}\,,             \\
                &\fR_{1,\fq} = \{\,(1,\fp\fq,\zeta_3 )\,,\,( \fq,\fq,\zeta_3^2) \,\}\,,
        &\fR_{\fq,1} = \{\,(\fq,\fq,\zeta_3 )\,,\,( 1,\fp\fq,\zeta_3^2) \,\}\,,                      \\
                &\fR_{1,\fp} = \{\,(1,\fp\fq,\zeta_3 )\,,\,( \fp,\fp,\zeta_3^2) \,\}\,,
        &\fR_{\fp,1} = \{\,(\fp,\fp,\zeta_3 )\,,\,( 1,\fp\fq,\zeta_3^2) \,\}\,,                      \\
                &\fR_{1,\fp\fq} = \{\,(1,\fp\fq,\zeta_3 )\,,\,( \fp\fq,1,\zeta_3^2) \,\}\,,
        &\fR_{\fp\fq,1} = \{\,(\fp\fq,1,\zeta_3 )\,,\,( 1,\fp\fq,\zeta_3^2) \,\}\,.  \\
        \end{array} 
        $$
}
\end{proposition}

\section{\red{The group $G(de,e,r)$}}\label{ImprimRS}

        Let $d$ and $e$ be two natural integers such that $de>1$.
         Let $k := \BQ(\zeta_{de})$, and let $V$ be a $k$--vector space with
        basis $(e_1,e_2,\dots,e_r)$.  
         We denote by $(e'_1,e'_2,\dots,e'_r)$ the dual basis of $W$ 
         with respect to the Hermitian pairing between $V$ and $W$.
 
        We denote by $G(de,1,r)$ the group consisting of all monomial matrices with coefficients
        in $\bmu_{de}$, isomorphic to $(\bmu_{de})^r\rtimes\fS_r$. 
 
         We denote by $G(de,e,r)$ the normal subgroup of
        index $e$ in $G(de,1,r)$ consisting in those matrices in $G(de,1,r)$ whose product of nonzero
        entries lies in $\bmu_d$. When $r=1$ then $d\ne 1$ is not allowed.
 
\begin{remark}
         Unless $r=2$ and $d=1$, 
         the  field of definition of $G(de,e,r)$ is $\BQ(\zeta_{de})$. 
         We will classify root systems over that field.
        The field of definition of $G(e,e,2)$, the dihedral group of order $2e$, 
         is the maximal real subfield
         of $\BQ(\zeta_e)$, namely the field $\BQ(\zeta_e+\zeta_e\inv)$. 
        The root systems of $G(e,e,2)$ over that field will be treated in
        Subsection \ref{dihedral}.
\end{remark}

\subsection{The  reflections of $G(de,e,r)$.}
\label{RefsOfDEER}\hfill

        It is well known (see for example \cite{nebe}, proof of Lemma 5)
        that the reflections $\Rf_1(d,r)$ and $\Rf_2(de,r)$ defined below exhaust the 
        collection of reflections of $G(de,e,r)$. 
        The set of reflections of $G(e,e,r)$ is precisely $\Rf_2(e,r)$.

\begin{description}
        \item [$\Rf_1(d,r)$] \index{Ref1@$\Rf_1(d,r)$} 
                elements $s_k^i$ ($1\leq i\leq d-1\,,\,1\leq k \leq r$) defined by
                $$
                        s_k^i : 
                        \left\{ 
                \begin{aligned} 
                        &e_k \mapsto \zeta_d^i e_k \\ 
                                &e_l \mapsto e_l\,\,\text{ if }l\neq k
                \end{aligned}
                \right.
                $$
        Note that if $d=1$ the set $\Rf_1(d,r)$ is empty, while if $d>1$ it consists of reflections.
\smallskip

        \item [$\Rf_2(de,r)$] \index{Ref2@$\Rf_2(e,r)$} 
                the involutive reflections $s_{k,l}^{(j)}$ (with $0\leq j\leq de-1\,,\,1\leq k<l\leq r$) 
                defined by
                 $$
                         s_{k,l}^{(j)}: 
                \left\{ 
                \begin{aligned} 
                        &e_k \mapsto \zeta_{de}^j e_l \\ 
                        &e_l \mapsto \zeta_{de}^{-j} e_k \\ 
                        &e_m \mapsto e_m\,\,\text{ if }m\neq k,l
                \end{aligned}
                \right.
                $$
\end{description}

      It is also well known (see for instance \cite[\S 3]{bmr}) that the
      following set of $r+1$ reflections generates $G(de,e,r)$:
        \begin{equation}\label{PiS}
                \left\{s_{12}^{(0)}, s_{23}^{(0)}, 
                \ldots, s_{(r-1),r}^{(0)}, s_{(r-1),r}^{(1)}, s_r^1\right\}.
        \end{equation}
\subsection{The complete reduced \zroot\ system $\fR(de,e,r)$.}\hfill
\smallskip
 
        For each reflection above we define a \zroot\ -- here again,
         $i,j,k,l,m$ denote natural integers such that
        $1\leq k<l\leq r\,,\,1\leq i\leq d-1\,,\,0\leq j\leq de-1$
        (if $d=1$ some of our objects do not exist):
        $$
                \begin{array}{ll}
                 \frr_k^i := \bigl( \ideal{}e_k,
                 \ideal{(1-\zeta_d^{-i})}e'_k,\zeta_d^i \bigr) ,
                &\mbox{an $s_k^i$-root (if $d>1$),} \\
                \frr_{k,l}^{(j)} := \bigl(
                \ideal{}(e_k-\zeta_{de}^je_l),\ideal{}(e'_k-\zeta_{de}^je'_l),-1 \bigr)
                 \,,
                &\mbox{an $s_{k,l}^{(j)}$-root.}
                \end{array}
        $$

        The actions of the reflections on these roots and the 
        Cartan pairings of roots are described below. 
        The results follow from direct calculation.

\begin{lemma}\label{actionsofref}
         Action of the reflections on roots:
        {\footnotesize
         $$
         \begin{array}{lllll}
                &\hskip-0.5cm s_{\frr_k^\alpha}(\frr_l^\beta) =
                &\hskip-0.5cm\frr_l^\beta 
                &s_{\frr_k^\alpha}(\frr_{l,m}^{(\beta)}) = 
                &\hskip-0.3cm\left\{\hskip-0.3cm
        \begin{array}{lll}
                &\frr_{l,m}^{(\beta)}       &\text{if } k\neq l,m \\
                &\frr_{l,m}^{(\beta-e\alpha)}    &\text{if } k = l \\
                &\frr_{l,m}^{(\beta+e\alpha)}    &\text{if } k = m \\
        \end{array}
         \right. \\
                &&&&\\
                &\hskip-0.5cm s_{\frr_{l,m}^{(\beta)}}(\frr_k^\alpha) = 
                &\hskip-0.3cm\left\{\hskip-0.5cm
        \begin{array}{lll}
                &\frr_k^\alpha   &\text{if } k\neq l,m \\
                &\frr_m^\alpha  &\text{if } k = l \\
                &\frr_l^\alpha   &\text{if } k = m \\
        \end{array}
        \right.
                &s_{\frr_{j,k}^\alpha}(\frr_{l,m}^{(\beta)})= 
                & \hskip-0.3cm\left\{\hskip-0.5cm
        \begin{array}{lll}
                &\frr_{l,m}^{(\beta)}       &\text{if } \!\{j,k\}\!\cap\!\{l,m\}\! =\! \emptyset \! \\
                &\frr_{k,m}^{(\beta-\alpha)}   &\text{if } j=l , k<m\\
                &\frr_{m,k}^{(\alpha-\beta)}   &\text{if } j=l , m<k\\
                &\frr_{l,k}^{(\beta+\alpha)}    &\text{if } l<j = m<k \\
                &\frr_{j,m}^{(\beta+\alpha)}    &\text{if } j<k = l<m \\
                &\frr_{j,l}^{(\alpha-\beta)}    &\text{if } j<l, k = m \\
        		   &\frr_{l,j}^{(\beta-\alpha)}    &\text{if } l<j, k = m \\
        		   &\frr_{j,k}^{(2\alpha-\beta)}    &\text{if } j=l, k = m \\
        \end{array}
        \right. \\
        \end{array}
        $$
}
\end{lemma}
 
\begin{lemma}\label{pairings}
        Cartan pairings of roots:
        {\small
        $$
                \begin{array}{lllll}
                &\fn({\frr_k^\alpha},{\frr_l^\beta}) = 
                &\hskip-0.5cm\left\{\hskip-0.5cm
                 \begin{array}{lcl}
                 &0          &\text{if } k\neq l \\
                &\ideal{(1-\zeta_d^\beta)} &\text{if } k = l \\
                \end{array}
                \right. 
                \\&&\\
                &\fn({\frr_k^\alpha},{\frr_{l,m}^{(\beta)}} )= 
                &\hskip-0.4cm\left\{\hskip-0.5cm
                \begin{array}{rcl}
                &0          &\text{if } k\neq l,m \\
                &\ideal{}          &\text{if } k = l \mbox{ or if } k=m\\
                \end{array}
                \right. \\&&\\
                 &\fn({\frr_{l,m}^{(\beta)}}{,\frr_k^\alpha})= 
                &\hskip-0.4cm\left\{\hskip-0.5cm
                \begin{array}{rcl}
                &0         &\text{if } k\neq l,m \\
                &\ideal{\left(1-\zeta_d^\alpha\right)}  &\text{if } k = l \mbox{ or if } k=m\\
                \end{array}
                \right. 
                \\&&\\
                &\fn({\frr_{j,k}^{(\alpha)}},{\frr_{l,m}^{(\beta)}} )= 
                &\hskip-0.2cm \left\{\hskip-0.5cm
                 \begin{array}{rcl}
                &0                 &\text{if } \{j,k\}\cap \{l,m\} = \emptyset \\
                &\ideal{}                 &\text{if } |\{j,k\}\cap \{l,m\}|=1 \\
                &\ideal{\left(1+\zeta_{de}^{\alpha-\beta}\right)} &\text{if } \{j,k\} = \{l,m\} \\
                \end{array}
                \right.  &&\\
                \end{array}
        $$
}
\end{lemma}

        The preceding calculations ensure that \index{Rdeen@$\fR(de,e,r)$}
        $$  
                \fR(de,e,r) := \fR_1(d,r) \cup \fR_2(de,r) 
        $$ 
        is a complete reduced root system for $G(de,e,r)$, where:
        $$
        \begin{aligned}
                & \fR_1(d,r) := \{\, \frr_k^i \,\}_{(1\leq k\leq r)(1\leq i\leq d-1)}\\
                & \fR_2(de,r) := \{\, \frr_{k,l}^{(j)}\,\}_{(1\leq k<l\leq r)(0\leq j\leq de-1)}
        \end{aligned}
        $$
        For a given $i$, we write $\fR_1^i(d,r) := \{\, \frr_k^i\,\}_{(1\leq k\leq r)} $ 
        and for a given $j$,  $\fR_2^{(j)}(de,r):= \{\, \frr_{k,l}^{(j)}\,\}_{(1\leq k<l\leq r)}$
        so that
        $$
                \fR_1 = \bigcup_{1\leq i\leq d-1} \fR_1^i \quad \mbox{ and } 
                \quad 
                \fR_2(de,r) = \bigcup_{0\leq j\leq de-1} \fR_2^{(j)} (de,r)\,.
        $$
        The ``even'' and ``odd'' parts of $\fR_2$ are defined to be:
        $$
                \fR_2^0(de,r) := \bigcup_{j=0,2\dots} \fR_2^{(j)}(de,r)
                \quad\text{and}\quad
                 \fR_2^1(de,r) := \bigcup_{j=1,3\dots} \fR_2^{(j)}(de,r)
                 \,.
        $$

\begin{lemma}\habel{arootGdeen}\hfill
\begin{enumerate}
        \item
                The set $\fR(de,e,r)$ is a complete reduced \zroot\ system for $G(de,e,r)$.
        \item
                 The orbits of $G(de,e,r)$ on $\fR(de,e,r)$ are
                \begin{itemize}
                        \item when $e$ is even and $r=2$:
                        $$
                                \left\{ \fR_1^1(d,2),\fR_1^2(d,2),\dots,\fR_1^{d-1}(d,2),
                                \fR^0_2(de,2), \fR^1_2(de,2)\right\}
                                \,,
                        $$
                        \item 
                                in other cases:
                        $$
                        \{\, \fR_1^1(d,r),\fR_1^2(d,r),\dots,\fR_1^{d-1}(d,r),\fR_2(de,r) \,\}
                        $$
                 \end{itemize}
                where the sets $\fR_1^j(d,2)$ are empty if $d=1$ and $\fR_2(de,r)$ is empty
                if $r=1$.
\end{enumerate}
\end{lemma}
 
 \begin{proof}
        The results follow directly from Lemmas~\ref{actionsofref} and~\ref{pairings}.
 \end{proof}
 
 \subsection{Classifying  complete reduced root systems for $G(de,e,r)$}\hfill
 \smallskip
 
        While $\fR(de,e,r)$ is a representative of one genus of complete reduced root 
        system for $G(de,e,r)$ over $\BQ(\zeta_{de})$, 
        there may be others, as described by the following theorem.

        The group $G(de,e,r)$ is a reflection group over $k$ if and only if
        $\zeta_{de}\in k$, except for $G(e,e,2)$ where the field of definition
        is $\BQ(\zeta_e+\zeta_e\inv)$. Thus the following theorem fulfills
        our aim stated in the introduction: classify \zroot\ systems for
        reflection groups over $k$, except for the case of $G(e,e,2)$ over
        its field of definition which will be considered in  Subsection
        \ref{dihedral}.
 
\begin{theorem}\label{rootGd1n}
        Given a family
        $
                \CF = \left\{\, \fa_1,\fa_2,\dots,\fa_{d-1}, \fb_0, \fb_1 \,\right\}
        $
        of fractional ideals where, unless $r=2$ and $de=2p^k$, $p$ prime and 
        $k>1$, we have $\fb_0=\fb_1=\ideal{}$, we define the set
        $$
                \fR_\CF(de,e,r) := 
                 \left( \bigcup_{1\leq i\leq d-1} \fa_i\cdot\fR_1^i(d,r) \right)
                \, \cup\, \fb_0\cdot\fR_2^0(de,r)
                \,\cup\, \fb_1\cdot\fR_2^1(de,r)
                 \,.
        $$

        Then every genus of complete reduced \zroot\ system for $G(de,e,r)$
        when $\zeta_{de}\in k$ contains
        exactly one root system $\fR_\CF(de,e,r)$ where $\CF$
        satisfies the additional conditions
        \begin{itemize}
                \item 
                        The $\fa_j$ and the $\fb_j$ are integral,
                \item
                        $\fb_0$ and $\fb_1$ are relatively prime divisors of
                        $\ideal{(1+\zeta_{de})}$,
                \item
                        for all $i=1,2,\dots,d-1$, $\fa_i$ divides
                        $\ideal{(1-\zeta_d^i)}$, and
                \item
                        for all $i=1,2,\dots,d-1$, $\fb_0$ and $\fb_1$ divide $\fa_i$.
        \end{itemize}
\end{theorem} 

\begin{proof}
        By the description of the orbits of $G(de,e,r)$ over $\fR(de,e,r)$ (see 
        Lemma \ref{arootGdeen} above) and \ref{nebegen}, 
        we see that any complete reduced \zroot\ system for $G(de,e,r)$ 
        over a field containing $\zeta_{de}$ is
        of the form $\fR_\CF(de,e,r)$ for some fractional ideals $\fa_j$ and $\fb_j$.
 
        Without changing the genus, we may assume that the elements of
        $\CF$ are integral and relatively prime.
 
        Now computations of the Cartan pairings (see Lemma \ref{pairings} above) show that
        the following ideals must be integral (for $1\leq\alpha,\beta\leq d-1$ and
        $j,k = 0,1$):
        $$
                \left(1-\zeta_d^\beta\right)\fa_\alpha\fa_\beta\inv \,,\,
                \fa_\alpha\fb_j\inv  \,,\,
                \left(1-\zeta_d^\alpha\right)\fb_j\fa_\alpha\inv   \,,\,   
                 \left(1+\zeta_{de}^{-(j+k)}\right)\fb_j\fb_k\inv
        $$

        We  see that $\fb_0$  and $\fb_1$ divide  all the $\fa_i$ and since
        the  whole family is relatively prime,  it follows that $\fb_0$ and
        $\fb_1$  are  relatively  prime.  Since  $\fa_\alpha$  divides both
        $(1-\zeta_d^\alpha)\fb_0$ and $(1-\zeta_d^\alpha)\fb_1$, we see then that
        $\fa_\alpha$ divides $\ideal{(1-\zeta_d^\alpha)}$.

         Finally we also see that $\fb_0$ and $\fb_1$ divide 
         $\ideal{(1+\zeta_{de})}$.

         When $r>2$, $\fb_0=\fb_1$ since $\fR_2^0(de,r)$ and
         $\fR_2^1(de,r)$ are in the same orbit. Since they are relatively
         prime, they must be trivial.

         When $r=2$, we use that $1+\zeta_{de}$ is a unit unless $de$
         is of the form $2p^k$, see \ref{1+zeta}.
\end{proof}

        We remark that when $k=\BQ(\zeta_{de})$, the ideal
        $\ideal{(1+\zeta_{de})}$ is $\ideal{}$ or prime, see \ref{lem:m=n}.
 
\begin{corollary}[Case of $G(d,1,r)$]\label{Gd1n}
        The map
        $$
                \CF \mapsto \fR_\CF(d,1,r) = \bigcup_{1 \leq d-1} \fa_i\cdot
                \fR_1^i(d,r) \, \cup \,  \fR_2(d,r)
        $$
        induces a bijection between the set of families
        $
                 \CF := \left\{\, \fa_1,\fa_2,\dots,\fa_{d-1}
                 \,\right\}
        $
        of integral ideals such that 
        \begin{itemize}
                \item
                        for all $i$, $\fa_i$ divides $\ideal{(1-\zeta_d^i)}$ and
                \item
                          the ideals $\fa_i$ are relatively prime,
        \end{itemize}
        and the set of genera of complete reduced \zroot\ systems for $G(d,1,r)$.
\end{corollary}

\begin{corollary}[Case of $G(e,e,r)$]\label{Geen}
        Assume that $r>2$.
        There is only one genus of
        \zroot\ systems for $G(e,e,r)$,
        represented by the principal root system $\fR_2(e,r)$.
 \end{corollary}
 
\begin{corollary}[The dihedral group]\habel{Gee2}
\begin{enumerate}
        \item
                        If $e$ is odd there is only one genus of $\BZ[\zeta_e]$-root system
                        for $G(e,e,2)$ represented by the principal root system $\fR_2(e,2)$.
        \item
                If $e$ is even, the map 
                        $$
                                (\fb_0,\fb_1) \mapsto \fb_0\cdot\fR_2^0(e,2) \cup \fb_1\cdot\fR_2^1(e,2)
                        $$
                        induces a bijection between 
                the set of pairs of  relatively prime integral 
                ideals of $\BZ[\zeta_e]$ dividing $\BZ[\zeta_e]{(1+\zeta_e)}$
                and the set of genera of $\BZ[\zeta_e]$-root systems
                        for $G(e,e,2)$.
\end{enumerate}
\end{corollary}

\begin{remark}
        As a particular case of Corollary \ref{Gee2} above, we recover
        the results already cited in Subsection~\ref{subsecB2} 
        about the $\BZ[i]$-root systems for the Weyl group of type $B_2$.
\end{remark}

\subsection{Root systems for $G(e,e,2)$ on its field of definition}\habel{dihedral}
\smallskip

	The   field  of  definition  of  the  dihedral  group  $G=G(e,e,2)$  is  
	$k=\BQ(\zeta_e+\zeta_e\inv)$,  the largest  real subfield  of $\BQ(\zeta_e)$;
	its   ring  of   integers  is   $\BZ[\zeta_e+\zeta_e\inv]$  by  Proposition
	\ref{integerscyclotomic}  (2). In this section  we classify root systems of
	$G$  over  its  field  of  definition.  From  now  on  we write $\zeta$ for
	$\zeta_e$.

\begin{lemma}
	The involutory reflections:
	$$
                s = \begin{pmatrix}
                -1& 2+\zeta + \zeta\inv  \\
                0&1 \\
                \end{pmatrix}
                \qquad \mbox{ and } \qquad
                t = \begin{pmatrix}
                1& 0  \\
                1&-1\\
                \end{pmatrix}
        $$
	generate $G$, and satisfy the dihedral relation $(st)^e=1$.

	Moreover, when $e$ is even, $(st)^{e/2}=-1$.
\end{lemma}

\begin{proof}
	The elements of $G$ of the form $(st)^n$ are rotations.
	These elements have matrices:
        $$
                (st)^n = \frac{1}{\Im(\zeta)} 
                \begin{pmatrix}
                        \Im(\zeta^n+\zeta^{n+1}) & -\Im(\zeta^{n-1}+2\zeta^n+\zeta^{n+1})  \\
                        \Im(\zeta^n) & -\Im(\zeta^n+\zeta^{n-1})\\
                \end{pmatrix} ,
        $$
	where $\Im(z)$ denotes the imaginary part of a complex number $z$.
	From this computation the lemma is obvious.
\end{proof}

	The reflection lines of $s$ and  $t$ respectively are in the directions of 
	the standard vectors 
	$e_1 = \binom{1}{0}$ and $e_2= \binom{0}{1}$ respectively. 

	Define principal roots
	$\frr_s=\BZ_k\cdot(e_1, 2e_1'-(2+\zeta+\zeta\inv)e_2',-1)$ 
	and $\frr_t=\BZ_k\cdot(e_2, -e_1'+2 e_2',-1)$.

\begin{proposition}\label{roots dihedral}
        If $e$ is odd let $\fR := \{(st)^n \frr_s \mid 0 \leq n <e \}$.        
        If $e$ is even, let 
        $
                \fR_s :=\{(st)^n\frr_s \mid 0 \leq n < e/2 \}\,,\,
                \fR_t :=\{(st)^n\frr_t \mid 0 \leq n < e/2 \}\,,
        $
        $
                \text{and }\quad
                \fR :=\fR_s\cup\fR_t
                \,.
        $
\begin{enumerate}
        \item
                In both cases $\fR$ is a complete, reduced, distinguished and 
                principal \zroot\ system for $G$ with root basis $\{\frr_s,\frr_t\}$.
        \item
                There is a single genus of reduced \zroot\ systems for $G$ unless
                $e=2p^k$, $p$ prime, $k\ge 1$. 
                
                In this last case, there is another
                genus, represented by the principal root system given by 
                $\fR_s\cup (2+\zeta+\zeta\inv)\cdot \fR_t$.
\end{enumerate}
\end{proposition}
\begin{proof}
	We have:
        $$
                (st)^n e_1 =  \frac{1}{\Im(\zeta)}
                \begin{pmatrix} 
                        \Im(\zeta^n+\zeta^{n+1}) \\
                         \Im(\zeta^n)
                \end{pmatrix}  \,.
        $$

\bul
	If  $e$ is odd, then the dihedral relation ensures that all the reflections
	of  $G$  are  conjugate: setting  $q=(e-1)/2$,  we have $\zeta^{q+1}=-\zeta_{2e}$  and
	$\Im(\zeta^{q+1})=-\Im(\zeta^q) $. Thus
	$(st)^q e_1 = \frac{\Im(\zeta^q)}{\Im(\zeta)} e_2$ and
	$(st)^q\frr_s=\alpha\inv\frr_t$ with
	$$
		\alpha = \frac{\Im(\zeta)}{\Im(\zeta^q)} = -\zeta^q (1+\zeta)=-2\Re(\zeta^q) =\Re(\zeta_{2e})
		\,
	$$ 
	and $\alpha^2=2+\zeta+\zeta\inv=(1+\zeta)(1+\zeta\inv)$.

	By \ref{1+zeta} in Appendix~\ref{arithmetic}, $1+\zeta_e$ is a unit unless
	$e$ is of the form $2p^k$, $p$ prime, $k\ge 1$.
	In particular, if $e$ is odd,  $\alpha$ is a unit in $\BZ[\zeta]$,  which
	proves that actually $(st)^q\frr_s=\frr_t$.

        Since $\frn{\frr_s}{\frr_t}=-(2+\zeta+\zeta\inv)$ and
        $\frn{\frr_t}{\frr_s}=-1$, it results from Proposition~\ref{prop:jeanlemma} that
        $\fR$ is a \zroot\ system for $G$. Thus
        there is a single genus of reduced root systems by \ref{nebegen}.
\smallskip

\bul
	If  $e$ is even, there  are two conjugacy classes  of reflections, of which
	$s$  and $t$ are representatives. Thus  the orbits of $\frr_s$ and $\frr_t$
	are  distinct, and  their size  is $e/2$  since $(st)^{e/2}=-1$. Again the
	pairings  are integral by Proposition~\ref{prop:jeanlemma}. But, this time we can scale one of
	the orbits by some fractional ideal $\fa$ (see \ref{nebegen}), and for the pairings
	to remain integral we need that $\fa$ be an integral divisor
	of $2+\zeta+\zeta\inv=(1+\zeta)(1+\zeta\inv)$. Thus by \ref{1+zeta},
	$\fa$ can be non trivial only if $e=2p^k$.
	In this last case, by Lemma \ref{zeta+zetainv  prime},  
	$2+\zeta+\zeta\inv$  is prime, so the only
	possible values for $\fa$ are the principal ideals $\ideal{}$ and
	$\ideal{(2+\zeta+\zeta\inv)}$.

	Proposition \ref{prop:RootBasis} ensures that $\{\frr_s,\frr_t\}$ is always a root basis.
\end{proof}

\subsubsection*{The symmetric dihedral root system}\hfill
\smallskip

        A symmetric (``self-dual'') root system for $G(e,e,2)$ can be obtained by 
        adjoining $2 \cos(\frac{\pi}{e})$. This results in the ``symmetric Cartan matrix'':
        $$
                C_{sym}=
                		\begin{pmatrix}
                                2&-2 \cos(\frac{\pi}{e})\\
                                -2 \cos(\frac{\pi}{e})&2\\
			\end{pmatrix}
        $$
        This corresponds to a root system over $\BZ_k$ only if $e$ is odd
        (see item (2) of \ref{dihedralfields}). 
        
        Now  $\left(2+\zeta+\zeta\inv\right) = \left(2 \cos(\frac{\pi}{e})\right)^2$; 
        that is,  adjoining $2\cos(\frac{\pi}{e})$ to $\BZ_k$ makes $\left(2+\zeta+\zeta\inv\right)$ a square. 
        Unless $e$ is a power of 2, $\left(2+\zeta+\zeta\inv\right)$ is a unit, 
        so we do not get any additional root systems.
        However, if $e$ is a power of 2, we get a third, self-dual, root system. The three are:       
	$$
		\fR_s \cup \fR_t ,  
		\qquad 
		\left(2\cos\left(\frac{\pi}{e}\right)\right)^2 \cdot \fR_s \cup \fR_t 
		\qquad \mbox{ and } \qquad 
		2\cos\left(\frac{\pi}{e}\right) \cdot \fR_s \cup \fR_t 
		\,.
	$$
        The last root system in this list (with symmetric Cartan matrix) is stable under the outer 
        automorphism of $G(e,e,2)$.

\subsection{Classifying distinguished root systems for $G(de,e,r)$}\hfill
\smallskip

        The complete reduced systems for $G(e,e,r)$ described above are 
        also distinguished because all reflections, being involutive, 
        are distinguished.

        For $G(de,e,r)$ with $d>1$, the distinguished roots 
        (those corresponding to distinguished reflections) are those in  
        $\fR_2(de,r)$ (corresponding to involutive reflections) 
        and $\fR_1^1(d,r)$.

        Applying Proposition~\ref{RestrictingR} (every complete root system 
        contains a 
        distinguished root system) to Theorem~\ref{rootGd1n} allows us to deduce the following result directly.

\begin{theorem}\label{distinguishedGdeenRS}
        Given a family
        $
                 \CF_r = \left\{\, \fa, \fb_0, \fb_1 \,\right\}
        $
        of fractional ideals, where $\fb_0=\fb_1=\ideal{}$ unless $r=2$
        and $de=2p^l$, $p$ prime, $l\ge 1$;  define 
        $$
                \fR_{\CF_r}(de,e,r) := 
                \fa\cdot\fR_1^1(d,r) 
                 \, \cup\, \fb_0\cdot\fR_2^0(de,r)
                 \,\cup\, \fb_1\cdot\fR_2^1(de,r)
                 \,.
                $$
        Then every genus of distinguished \zroot\ system for $G(de,e,r)$
        over a field containing $\zeta_{de}$ contains
         exactly one root system $\fR_{\CF_r}(de,e,r)$ where $\CF_r$
        satisfies the additional conditions
        \begin{itemize}
                \item 
                        $\fa$, $\fb_0$  and $\fb_1$ are integral,
                 \item
                        $\fb_0$ and $\fb_1$ are relatively prime divisors of
                        $\ideal{(1+\zeta_{de})}$,
                \item
                        $\fa$ divides $\ideal{(1-\zeta_d)}$,
                \item
                        $\fb_0$ and $\fb_1$ divide $\fa$.
        \end{itemize}
\end{theorem} 

        Again, we remark that when $k=\BQ(\zeta_{de})$, the ideal
        $\ideal{(1+\zeta_{de})}$ is $\ideal{}$ or prime, see Lemma~\ref{lem:m=n}.

        The complete root systems for $G(e,e,r)$, including all the dihedrals $G(e,e,2)$, 
        as well as $G(2,1,r)$ as described in Corollaries~\ref{Gd1n}, \ref{Geen} and \ref{Gee2}
        are all distinguished by definition.

\begin{corollary}\label{distinguishedGd1n}
        Each genus of distinguished root system for $G(d,1,r)$ contains a root system of the form:
        $$
                \fR_{\CF_r}(d,1,r) = \fa\cdot \fR_1^1(d,r) \cup \,  \fR_2(d,r)
                \,,
        $$
        where $\fa\in\{\ideal{},\ideal{(1-\zeta_d)}\}$.
        Thus there are two genera of
        distinguished root systems of $G(d,1,r)$, each containing a principal
        root system.
\end{corollary}
 
 \begin{remark}
        It was assumed at the beginning of this section that $de>1$, excluding the case $G(1,1,r)$. 
 
        However it is well known that the reflections contained in the symmetric group $\fS_r$, 
        that is, the Weyl group of type $A_{r-1}$, are precisely the reflections:
        $$
                \Rf_2(1,r)=\{s_{k,l}^{(0)} \mid 1 \leq k<l\leq r\}
        $$ 
        defined as above, with a single orbit of roots 
        $$
                \fR(1,1,r):=\fR_2(1,r) =\{ \frr_{k,l}^{(0)} \mid 1 \leq k<l\leq r\}
        $$ 
        where
        $$
                \frr_{k,l}^{(0)}=\left(\ideal{} (e_k-e_l),\ideal{} (e'_k-e'_l),-1 \right),
        $$
        as defined above. The set $\fR(1,1,r)$ forms a complete and reduced root system 
        for $\fS_r$, and as the reflections are all involutive, it is also distinguished.

        Since the action of $\fS_r$ on the roots consists of a single orbit, there is a unique genus of 
        root system, with representative given by $\fR(1,1,r)$ above.
\end{remark}

	One may notice that (see Remark \ref{rem:defnebe}), as far as we deal with
	well-generated reflection groups,
	the next result allows us to reduce our classification problem to a problem which has
	already been solved by Nebe \cite{nebe}.

\begin{theorem}\label{allprincipal}
	All genera of distinguished root systems over the field of definition of a
	well-generated irreducible complex reflection group contain a principal 
	root system.
\end{theorem}

\begin{proof}
	Let $k$ be the field of definition of the well-generated irreducible
	complex reflection group $G$.
	If $G$ is primitive, then $\BZ_k$ is a P.I.D. and there is nothing to prove.
	If $G=G(d,1,r)$ the statement has been given in Corollary
	\ref{distinguishedGd1n}. If $G=G(e,e,r)$ with $r>2$ the statement has been
	given in Corollary \ref{Geen}.
	Finally, if $G=G(e,e,2)$ the result is proved in Proposition \ref{roots dihedral}.
\end{proof}

\section{\red{Principal root systems and Cartan matrices}}

\subsection{Cartan matrices}


\begin{notation}\label{not:defRS}
        Given a distinguished \zroot\ system $\fR$,
\begin{itemize}
        \item
                for every distinguished reflection $s$ of $G(\fR)$, by 
                Proposition~\ref{prop:distinguishedbijection} there is a 
                corresponding  root
                in $\fR$ which we denote by $\frr_s$        
        \item
                given a set of distinguished reflections $S \subset \Rf(G(\fR))$,
                we set $\fR_S := \{ \frr_s \mid s \in S\}$.
\end{itemize}
\end{notation}

\begin{definition}\label{cartanmatrix}\index{Cartan matrix}
        Let $\fR$ be a principal distinguished \zroot\ system.
        Let $S$ be an ordered subset of distinguished reflections of $\Rf(G(\fR))$.
        For each $s \in S$, choose generators $\al_s$ and $\be_s$
        of respectively $I_{r_s}$ and $J_{r_s}$ as in Remark \ref{rem:principal}.
        The matrix with entries $C_{s,t} = \scal{\al_s}{\be_t}$ for 
        $s,t \in S$ is called a \emph{Cartan matrix} for $\fR_S$.
 \end{definition}
 
        Note that all entries of such a Cartan matrix belong to $\BZ_k$.

        The proof of the following lemma is left to the reader.
        
\begin{lemma}\label{lem:conjugationbydiagonal}\hfill
\begin{enumerate}
        \item
                Two Cartan matrices for the same subset of a
                given principal \zroot\ system 
                are conjugate by a diagonal matrix over $\BZ_k^\times$.
        \item
                Let $R$ and $R'$ be  ordered subsets of two
                principal \zroot\ systems such that the ordered sets
                of reflections attached to them are equal. Then any
                Cartan matrix for $R$ is conjugate 
                to any Cartan matrix for $R'$ by a diagonal matrix.

\end{enumerate} 
\end{lemma}
\smallskip

\subsubsection*{Cartan matrices and genera of principal root systems}\hfill
\smallskip

\begin{theorem}\label{cartandeterminesgenus}
        Let $(V,G)$ be an irreducible reflection group.
        Let  $\fR$ and $\fR'$  be  two  principal distinguished \zroot\ systems
        such that $G(\fR)=G(\fR') = G$.
        Let $S$ be an ordered list of distinguished reflections which generates $G$. 
        If a Cartan matrix for $\fR_S$ is conjugate to a Cartan matrix for $\fR'_S$
        by a diagonal matrix of units of $\BZ_k$, then $\fR$ and $\fR'$ belong 
        to the same genus (that is, there is $\la \in k$ such that 
        $\fR' = \la\cdot\fR$).
\end{theorem}

\begin{proof}
        For each $s\in S$, let $\frr_s \in \fR_S$ and
        $\frr'_s \in \fR'_S$ be the corresponding roots.
        As in in Remark \ref{rem:principal}, choose $\al_s$ a generator of $I_{\frr_s}$ and 
        $\be_s$ a generator of $J_{\frr_s}$, 
        such that $\scal{\al_s}{\be_s}=1-\zeta_s$; similarly
        choose $\al'_s$ a generator of $I_{\frr'_s}$ and 
        $\be'_s$ a generator of $J_{\frr'_s}$, 
        such that $\scal{\al'_s}{\be'_s}=1-\zeta_s$.
        Let $\la_s \in k^\times$ such that $\al'_s = \la_s\al_s$. Thus
        $\be'_s = \la_s\tinv \be_s$.
        
        Let $C$ (resp. $C'$) be the Cartan matrix determined by the above choices. Up to
        changing 
        these choices by units, we may (and we do) assume $C = C'$.
        For any $s,t\in S$ we have
        $$
                C_{s,t} = \scal{\al_s}{\be_t} = \scal{\al'_s}{\be'_t} = \la_s \la_t\inv \scal{\al_s}{\be_t}
                \,.
        $$
        Hence, if $C_{s,t} \neq 0$, then
        $
                \la_s = \la_t
                \,.
        $
        Choose $s_0 \in S$.
        Since $G$ is irreducible, and $S$ generates $G$, 
        for any $t \in S$,  there is a sequence 
        $s_0,\ldots,s_l = t$ such that for all $j$, $s_j \in S$ and
        $ C_{s_{j+1},s_{j}} \ne 0$ for all $j =0,\dots,l-1$, which shows that $\la_{s_0} = \la_t$.
        Let $\la := \la_{s_0}$.
        It follows that $\fR'_S = \la \cdot \fR_S$.
        
        Finally, for any $\frr \in \fR$, by the assumption on orbits
        there is an expression 
        $s_\frr = \lexp{s_1 \cdots s_n} s_0$ for $s_\frr$ 
        in terms of $s_0$. In particular, since $\fR$ is reduced,
        this ensures that $\frr =(s_1 \cdots s_n) \cdot \frr_{s_0}$. Thus
        the fact that $\fR_S$ and $\fR'_S$ belong to the same genus propagates to 
        $\fR$ and $\fR'$.
\end{proof}
        
        The proof of the following proposition results from the bijection
        given in its second part.

\begin{proposition}
        \label{ClassificationByCartan}
        Let $G$ be a finite subgroup of $\GL(V)$ generated by reflections.
        Assume given a family $S$ of
        reflections which generates $G$.
                
        Let $\fR = \{\frr := (I_\frr,J_\frr,\zeta_\frr)\}$ be a distinguished principal
        \zroot\ system such that $G = G(\fR)$.
        We denote by $\Car(\fR_S)$ the set of all matrices $M$ satisfying the following
        conditions:
        \begin{itemize}
                \item
                        $M$ is conjugate to $C$, a Cartan matrix of $\fR_S$,
                        by a diagonal matrix (with diagonal entries in $k^\times$),
                \item
                        the entries of $M$ belong to $\BZ_k$.
        \end{itemize}
        There is a bijection between
        \begin{itemize}
                \item[--]
                        the set of conjugacy classes of $\Car(\fR_S)$
                        under the action of the group of diagonal matrices with
                        entries in $\BZ_k^\times$,
                \item[--]
                        the genera of distinguished principal
                        \zroot\ systems for $G$,
        \end{itemize}
        defined as follows.
        \begin{itemize}
                \item[$\ra$]
                        For $M =DCD\inv \in \Car(\fR_S)$,
                       where $D = (\la_s)_{s\in S}$ is an invertible diagonal matrix, 
                       and for $s \in S$ corresponding to the
                        root $\frr_s \in \fR$, we define 
                        $\frr_s^{(M)} :=  \la_s \cdot \frr_s$
                        and we denote by $\fR^{(M)}$ the orbit under $G$ of the family
                        $(\frr_s^{(M)})_{s\in S}$. Then $\fR^{(M)}$ is a 
                        distinguished principal
                        \zroot\ system for $G$, and $M$ is a Cartan
                        matrix associated with $\fR^{(M)}_S$.
                \item[$\leftarrow$]
                        Let $\fR'$ be a 
                        distinguished
                        principal root system for $G$.
                        By item (1) of Theorem \ref{cartandeterminesgenus}, we know that
                        $C(\fR'_S)$ is conjugate to $C(\fR_S)$ by a diagonal matrix,
                        and this shows that a Cartan matrix for $\fR'_S$ does
                        belong to $\Car(\fR_S)$.
        \end{itemize}
\end{proposition}

        The above result is at the heart of the classification of distinguished root 
        systems given in Section~\ref{sec:class}.

\subsection{Free root lattices}\habel{connection&cartan}
\smallskip

        The third assertion of the following theorem has been proved in
        \cite[Corollary\ 13]{nebe}.
        
\begin{theorem}\label{thm:freelattices}
        Let $(V,G)$ be a well-generated reflection group.
\begin{enumerate}
        \item
                There exists a distinguished principal \zroot\ system
                $\fR$ with $G(\fR) = G$.
        \item
                If $S$ is a set of reflections of cardinality $\dim V$
                which generates $G$, and if the corresponding roots are 
                $\frr_s = (\BZ_k\al_s,\BZ_k\al_s^\vee,\zeta_s)_{s\in S}$,
                then $(\al_s)_{s\in S}$ and $(\al_s^\vee)_{s\in S}$ are $\BZ_k$-bases
                of  $Q_\fR$ and $Q_\fR^\vee$ respectively.
        \item   By (1) and (2), $Q_\fR$ is a free $\BZ_k$-module.
\end{enumerate}
\end{theorem}

\begin{proof}
        Theorem~\ref{allprincipal} ensures that all genera of root systems for 
        the well-generated imprimitive groups $G(d,1,r)$ and $G(e,e,r)$ 
        contain a principal \zroot\ system.
        This,
        together with item (4) of Proposition \ref{prop:RootBasis},
        imply (1) and (2) directly.
\end{proof}

        The next result shows in particular that, for a well generated group, the connection
        index can be easily computed from any Cartan matrix.
        
\begin{proposition}\label{prop:connectionwellgen}
        Assume that $G$ is generated by $r=\dim V$ reflections
        $(s_i)_{1\leq i\leq r}$. For each $i$ ($1\leq i\leq r)$ choose $v_i \in L_{s_i}$,
        $w_i \in M_{s_i}$ such that $\scal{v_i}{w_i} =1-\zeta_{s_i}$. 
        
        Then the connection index of $G$ is equal to    
        $\det\big( \scal{v_i}{w_j}\big)_{1\leq i,j\leq r}$.
\end{proposition}

\begin{proof}
        The value of $\det\big( \scal{v_i}{f_j}\big)_{1\leq i,j\leq r}$ 
        is independent of the choices of the systems $v_i \in L_{s_i}$,
        $w_i \in M_{s_i}$ such that $\scal{v_i}{w_i} =1-\zeta_{s_i}$. 
        By item~(2) of  Theorem \ref{thm:freelattices}, 
        we may choose $v_i = \al_i$ and $w_i = \al_i^\vee$ so that 
        $(\al_i)_{1\leq i\leq r}$ and $(\al_i^\vee)_{1\leq i\leq r}$ are 
        $\BZ_k$-bases of respectively $Q_\fR$ and $Q_\fR^\vee$ for some
        distinguished principal \zroot\ system $\fR$ such that 
        $G(\fR) = G$.
        
        Then the statement is nothing but a translation of the proof of Theorem 
        \ref{th:connectionindex}.
\end{proof}

\begin{remark}\label{rem:connectionbadgenerated}
        Let $(V,G)$ be an irreducible and \textbf{not} well-generated reflection group. 
        Let $r = \dim V$ and
        assume that $G$ is generated by the set of reflections $S = \{s_1,\ldots,s_{r+1}\}$.
        Let $\fR$ be a distinguished principal root system for $G$. 
        For the roots $\frr_{s_1},\ldots \frr_{s_{r+1}}$ in $\fR$
        corresponding to $S$,
        we choose $\al_1,\dots,\al_{r+1}$ and
        $\be_1,\dots,\be_{r+1}$ as in Remark \ref{rem:principal}, and we denote by $C$ the 
        corresponding Cartan matrix for $\fR_S$.
        
        By item (2) of Proposition \ref{prop:RootBasis}, 
        $$
                Q_\fR = \sum_{i=1}^{r+1} \BZ_k \al_i
                \,\,\text{ and }\,\,
                Q_\fR^\vee = \sum_{i=1}^{r+1} \BZ_k \be_i
                \,.
        $$
        We will show (see
        Proposition~\ref{2connectionindices}  and the  tables of  Appendix \ref{table})
        that for every irreducible, not well-generated reflection group $(V,G)$ 
        there always exists at least one distinguished principal root system 
        for which (keeping the notation of the previous paragraph)
        there exists $i_0$ and $j_0$ such that:
        $$
                Q_\fR = \bigoplus_{i\neq i_0} \BZ_k \al_i
                \,\,\text{ and }\,\,
                Q_\fR^\vee = \bigoplus_{i\neq j_0} \BZ_k \be_i
                \,.
        $$
        Then the connection index of $\fR$ is equal to the determinant of the
        sub-matrix
        of $C$ where the $i_0$-th row and the $j_0$-th column have been dropped.
\end{remark}

\subsection{%
        Principal root bases and Cartan matrices
        for imprimitive well generated reflection groups%
}\label{prb and cm}\hfill
\smallskip

        Throughout this subsection, $\zeta$ is a root of unity, and
        $k$ is a number field containing $\zeta$ and closed under conjugation.
        
        Recall (see  Section~\ref{ImprimRS}) that 
        $\frr_{kl}^{(j)} = (\BZ_k(e_k-\zeta^j e_l), \BZ_k(e'_k-\zeta^{j} e'_l), -1)$.
        %
        Write 
        $\alpha_{kl}^{(j)}$ for the vector $e_k-\zeta^j e_l$ where $1\leq k<l\leq n$. 

\goodbreak
\subsubsection*{The case of $G(e,e,r)$}
\hfill
\smallskip

\noindent        \emph{A. $G(e,e,2)$ with $e$ even.}
The Cartan matrix of the root basis $\{\frr_s,\frr_t\}$ of the principal
\zroot\ system $\fR$ of Proposition~\ref{roots dihedral} is:
        $$
                C
                = \begin{pmatrix}
                2&-(2+\zeta_e + \zeta_e\inv) \\
                -1&2\\
                \end{pmatrix}
                =\begin{pmatrix}
                2&-4 \cos^2(\frac{\pi}{e})\\
                -1&2\\
                \end{pmatrix}.
        $$
        Hence:
\begin{lemma}\label{Gee2RB}
     The connection index of $G(e,e,2)$ with $e$ even is
         $$
         	c_{(e,e,2)} = (1-\zeta_e)(1-\zeta_e\inv)
		\,.
	$$
\end{lemma}

\medskip
 \noindent       \emph{B.  The general case $G(e,e,r)$ with $r>2$ or $e$
        odd, $r=2$.}
\smallskip

        By Corollary~\ref{Geen}, there is a unique genus of \zroot\ system
        (necessarily distinguished, being in $G(e,e,r)$), 
        and which contains the representative \zroot\ system  
        $$
                \fR=\fR(e,r) = \left\{ \frr_{kl}^{(j)} \mid 1\leq k<l\leq r,
                0\leq j < r\right\}
                \,.
        $$ 
        Using the notation of Section~\ref{ImprimRS}, the set:
        $$
                S=\left\{s_{12}^{(0)}, s_{23}^{(0)}, \ldots,
                s_{(r-1),r}^{(0)}, s_{(r-1),r}^{(1)}\right\}
                \,,
        $$
        consisting of $r$ involutive reflections, generates $G(e,e,r)$, with corresponding set of roots:
        $$
                \Pi =\left\{\frr_{12}^{(0)}, \frr_{23}^{(0)}, \ldots,
                \frr_{(r-1),r}^{(0)}, 
                \frr_{(r-1),r}^{(1)}\right\}
                \,,
        $$
        where $\frr_{k,l}^{(j)} :=  \bigl(
        \BZ_k(e_k-\zeta_{e}^je_l),\BZ_k(e'_k-\zeta_{e}^je'_l),-1 \bigr)$.
\smallskip
        
        Again, by \ref{prop:RootBasis} (4),  $\Pi$ is a root basis.
        The Cartan matrix for $\Pi$ is:
        $$
        \begin{pmatrix}
        2 & -1& 0 & \cdots \\
        -1 & 2 & -1 & 0 & \cdots \\
        0 & -1 & 2 & -1 & 0 & \cdots \\
        \vdots &  & \ddots & \ddots & \ddots \\
        & \cdots & 0 & -1 & 2 & -1 & 0 & 0  \\
        & & \cdots &  0 & -1 & 2 & -1 & -1   \\
        & & & \cdots & 0 & -1 & 2 & (1+\zeta_e\inv)   \\
        & & & \cdots & 0 & -1 &(1+\zeta_e) & 2   \\
        \end{pmatrix}
        $$
        In order to compute the determinant of the above matrix, we prove the following lemma, 
        which will also be useful later on.

\begin{lemma}\label{lem:jeancartan}
        Let $r\geq 3$. Consider an $r\times r$ matrix of type
        $$
        C_r := 
        \begin{pmatrix}
                2&-1&0&\ldots&0 \\
                -1&2&*&\ldots&* \\
                0&*&*&\ldots&* \\
                \vdots&\vdots&\vdots&\ddots&\vdots\\
                0&*&*&\ldots&*\\
        \end{pmatrix}
        \,.
        $$
        Let $C_{r-1}$ (resp. $C_{r-2}$) be the $(r-1)\times(r-1)$-matrix
        (respectively, the $(r-2)\times(r-2)$-matrix) obtained by suppressing the first
        row and the first column (resp. the first two rows
        and the first two columns). Then
        $$
                \det C_r = 2 \det C_{r-1} - \det C_{r-2}
                \,.
        $$
\end{lemma}

\begin{proof}
        It is immediate by expanding with respect to the first row.
\end{proof}

\begin{proposition}\label{ceen}
        
        For all choices of $e \geq 1$ and $r\geq 2$, 
        the connection index of $G(e,e,r)$ is 
        $$
                c_{(e,e,r)} = (1-\zeta_e)(1-\zeta_e\inv)
                \,.
        $$
\end{proposition}
\begin{proof}
 This results from Lemma \ref{lem:jeancartan} and from 
 Lemma \ref{Gee2RB}.
\end{proof}
\smallskip
        
\subsubsection*{The case of $G(d,1,r)$}\hfill
\smallskip

        In this case, $k=\BQ(\zeta_d)$ and $\BZ_k = \BZ[\zeta_d]$.

        Let $\fa$ be an integral ideal which divides $\ideal{(1-\zeta_d)}$. 
        By Corollary~\ref{distinguishedGd1n} every genus of distinguished root system for $G(d,1,r)$ 
        contains a \zroot\ system of the form $\fR_\fa :=\fa\cdot\fR_1^1(d,r) \cup \fR_2(d,r)$ 
        where $\fa=\BZ_k$ or $\ideal{(1-\zeta_d)}$ (so is principal).
        In the notation of Section~\ref{ImprimRS}, the set:
        $$
                S=\left\{s_{12}^{(0)}, s_{23}^{(0)}, \ldots,
                s_{(r-1),r}^{(0)}, s_r\right\}
         $$
        consisting of $r$ reflections, generates $G(d,1,r)$, and has corresponding set of roots:
        $$
                \Pi_\fa=\left\{\frr_{12}^{(0)}, \frr_{23}^{(0)}, \ldots, \frr_{(r-1),r}^{(0)}, 
                \fa\cdot\frr_{r}^{(1)}\right\}
                \,,
        $$
        where 
        \begin{align*}
                &\fa\cdot\frr_{r}^{(1)} = \bigl( \fa e_r,(1-\zeta\inv)
                \fa\tinv e_r',\zeta_d \bigr) \,,
                \text{ and}\\
                &\frr_{k,l}^{(0)}=\bigl(\ideal{(e_k-e_l)},\ideal{(e'_k-e'_l)},
                -1 \bigr)
                \,. 
        \end{align*}

        In particular, $\Pi := \Pi_{\BZ_k}$ provides a principal
        $\BZ_k$-basis for $Q_{\fR_{\BZ_k}}$.
        
        The Cartan matrix for $\Pi$ is:
        $$
        \begin{pmatrix}
        2 & -1& 0 & \cdots \\
        -1 & 2 & -1 & 0 & \cdots \\
        0 & -1 & 2 & -1 & 0 & \cdots \\
        \vdots &  & \ddots & \ddots & \ddots \\
        & \cdots & 0 & -1 & 2 & -1 & 0 & 0  \\
        & & \cdots &  0 & -1 & 2 & -1 & 0   \\
        & & & \cdots & 0 & -1 & 2 & -(1-\zeta_d)  \\
        & & & \cdots & 0 & 0 &-1 & 1-\zeta_d   \\
        \end{pmatrix} \,,
        $$
        Since $\fa$ is principal, then $\Pi_\fa$ 
        is principal and its Cartan
        basis is obtained by conjugating the above matrix by
        $\diag(1,\ldots,1,(1-\zeta_d))$.
        Applying Lemma \ref{lem:jeancartan} then gives:

\begin{proposition}\label{cd1n}
        The connection index of $G(d,1,r)$  ($d\geq 2$) is
        $$
                c_{(d,1,r)} = 1-\zeta_d
                \,.
        $$
\end{proposition}
\smallskip

\subsection{The case of not well-generated groups $G(de,e,r)$}\hfill
\smallskip

%
	We assume here that $G=G(de,e,r)$ is not well-generated, thus $d>1$
        and $e>1$.
        
        It follows from
        Theorem~\ref{distinguishedGdeenRS}, that
        $\fR:=\fR_1^1(d,r) \cup \fR_2(de,r)$ 
        is a distinguished \zroot\ system for $G$, which is principal
        since with the notation from Section~\ref{ImprimRS},
        $\fR_1^i(d,r)$ consists of:
        $$ 
        \frr_k^i= \BZ_k \cdot \bigl( e_k, {(1-\zeta_d^{-i})}e'_k,\zeta_d \bigr) 
        $$
        with $1 \leq k \leq r$ and  $0<i < d$, and $\fR_2(de,r)$ consists of:
        $$
                \frr_{k,l}^{(j)} = \BZ_k \cdot \bigl( (e_k-\zeta_{de}^je_l),
                (e'_k-\zeta_{de}^je'_l),-1 \bigr) 
        $$
        with $1 \leq k  <l \leq r$ and $0\leq j < de$.
        
		Recall (Definition~\ref{Def_rootBasis_etc}) that a subset 
		$\Pi =\left((I_\frr, J_\frr,\zeta_\frr)\right)_{\frr \in \Pi}\subset \fR$ is a  set of 
		\emph{root generators} if $  Q_\fR = \sum_{i\in \Pi} I_{\frr_i}$.
		 and a \emph{root lattice basis} if
		$Q_\fR= \bigoplus_{\frr \in R} I_\frr$.
		       
\begin{proposition}\label{prop:rootsGdeer}
	Assume that $d>1$ and $e>1$.
\begin{enumerate}
	\item
		The set:
       		$$
                		\Pi=\left\{\frr_{12}^{(0)}, \frr_{23}^{(0)},
                		\ldots, \frr_{(r-1),r}^{(0)}, \frr_{(r-1),r}^{(1)}, 
                		\frr_r^{1} \right\}
        		$$
        		forms a set of root generators for $G=G(de,e,r)$. 
	
        		The corresponding set of $r+1$ 
        		reflections:
        		$$
                		S_\Pi := 	\left\{
						s_{12}^{(0)}, s_{23}^{(0)}, 
                					\ldots, s_{(r-1),r}^{(0)}, s_{(r-1),r}^{(1)}, s_r^1
					\right\}
        		$$
        		generates $G$ and no set of $r$ reflections will do so.
		
	\item
		If $(r,de) \neq (2, 2p^l)$ ($l\geq 1$ any integer and $p$ any prime),
		the genera of distinguished root systems for $G$ are in 
                bijection with the integral ideals dividing $\ideal{(1-\zeta_d)}$. 
		
		More precisely, if $\fa$ is such a divisor, the corresponding 
                genus contains the root system with  set of root generators
		$$ 
			\Pi_\fa=
				\left\{
					\frr_{12}^{(0)}, \frr_{23}^{(0)},
               		 		\ldots, \frr_{(r-1),r}^{(0)}, \frr_{(r-1),r}^{(1)}, 
               				 \fa\cdot\frr_r^{1} 
				\right\}
			\,.
        		$$
\end{enumerate}
\end{proposition}
        
\begin{proof}\hfill

	(1)
	The fact that $S_\Pi$  generates $G$ has already been
	observed in Subsection~\ref{RefsOfDEER}.
	It follows by Proposition~\ref{prop:RootBasis}(2) that the
	set $\Pi$ is a set of root generators.
 \smallskip
        
        (2)
        The $(r+1) \times (r+1)$ Cartan matrix $C$ of $\Pi$ is
   	$$
        \begin{pmatrix}
        		2 & -1& 0 & \cdots \\
        		-1 & 2 & -1 & 0 & \cdots \\
       		 0 & -1 & 2 & -1 & 0 & \cdots \\
       		 \vdots &  & \ddots & \ddots & \ddots \\
        		& \cdots & 0 & -1 & 2 & -1 & 0 & 0 & 0  \\
        		& & \cdots &  0 & -1 & 2 & -1 &-1  & 0   \\
        		& & & \cdots & 0 & -1 & 2 &  1+\zeta_{de}\inv & \zeta_d-1 \\
        		& & & \cdots & 0 & -1 &  1+\zeta_{de} & 2 & -\zeta_{de} (1-\zeta_d)\\
        		& & & \cdots & 0 & 0 &-1 & -\zeta_{de}\inv & 1-\zeta_d   \\
        \end{pmatrix} \,,
        $$
        Theorem~\ref{distinguishedGdeenRS} ensures that (except when
        $r=2$ and $de=2p^l$, $p$ prime, $l\ge 1$), every other genus of root system
        for $G$ is obtained as 
        $\fa\cdot\fR_1^1(d,r) \cup \fR_2(de,r)$ 
        where $\fa$ is an integral ideal  dividing
        $\ideal{(1-\zeta_d)}$.
        This gives rise to root systems with root generators
        $$
        		\Pi_\fa=\left\{\frr_{12}^{(0)}, \frr_{23}^{(0)},
		\ldots, \frr_{(r-1),r}^{(0)}, \frr_{(r-1),r}^{(1)}, 
                	\fa\cdot\frr_r^{1} \right\}
                	\,.
        $$
        When $\fa=\ideal a$ is principal, these root generators are principal
        with Cartan matrix  $C_a$ conjugate to $C$ by $\diag(1,\ldots,1,a)$.
%
\end{proof}

        We will describe  root lattice bases for most genera of distinguished \zroot\
        system for $G$ (and all genera of principal distinguished \zroot\ systems). 
        By Theorem~\ref{distinguishedGdeenRS}, 
        each genus corresponds to the choice of ideals $\fb_0$, $\fb_1$ and $\fa$.
        
\begin{remark}\label{generalremarkgdeer}
        Assume $\fb_0, \fb_1$ and $\fa$ are principal, thus
	$\fb_0=\ideal{b_0}, \fb_1=\ideal{b_1}$ and $\fa=\ideal a$,
        and let $C'$ denote the matrix obtained by conjugating by 
        $\diag(b_0,\ldots, b_0, b_1,  a)$  the  Cartan  matrix  $C$  for  $\Pi$. 
        Let us explain how to determine if the root system
        with Cartan matrix $C'$ has a root lattice basis or coroot lattice basis and how to
        compute the corresponding connection index.

        Write  $l_1$, $l_2$ and $l_3$ for the last three rows of $C'$. 
        These rows satisfy the linear dependency relationship:
        $$
                \frac{1}{b_0}l_1-\frac{1}{b_1}l_2+\frac{(1-\zeta_{de})}a l_3 = 0 
        $$ 
        Similarly, if $c$, $c_2$ and $c_3$ are the last three columns of $C'$, then:
        $$
                b_0 c_1-b_1 c_2+\frac{(1-\zeta_{de}\inv)}{a'}c_3 =0
        $$
        where  $a'$ is the algebraic integer such  that $aa' =(1-\zeta_d)$. 

	Denote  $P_{ij}$ the property that $l_i$  is an integral linear combination
	of  the other  lines in  $\{l_1,l_2,l_3\}$ and  $c_j$ is an integral linear
	combination  of the other  columns in $\{c_1,c_2,c_3\}$.  If $P_{ij}$ holds
	then  the  root  system  has  a  root lattice basis  and  a  coroot lattice basis, and its
	connection  index is $\det C'_{ij}$ where $C'_{ij}$ is the matrix
        obtained from $C'$ by removing the row $l_i$ and the column $c_j$.

	The bottom right corner of the matrix of cofactors of $C'$  (the
	matrix of the $\det C'_{ij}$) has the following form:
	$$
	\begin{pmatrix}
		1-\zeta_d&\frac{b_1}{b_0}(1-\zeta_d)&\frac a{b_0}(1-\zeta_{de}\inv)\\
		\frac{b_0}{b_1}(1-\zeta_d)&1-\zeta_d&\frac a{b_1}(1-\zeta_{de}\inv)\\
		b_0a'(1-\zeta_{de})&b_1a'(1-\zeta_{de})&(1-\zeta_{de})(1-\zeta_{de}\inv)
	\end{pmatrix}
	\,.
	$$
\end{remark}
\smallskip

	We describe case by case, according to values of $d$ and $e$,
	the genera of \zroot\ systems for $G$ when $k=\BQ[\zeta_{de}]$ (the
        field of definition of $G$), and compute connection indices.

        By Theorem~\ref{distinguishedGdeenRS}, genera of distinguished \zroot\ systems
        are represented by
        $
                \fa\cdot\fR_1^1(d,r)\, \cup\, \fb_0\cdot\fR_2^0(de,r)
                 \,\cup\, \fb_1\cdot\fR_2^1(de,r)
        \,
        $
        where
\begin{itemize}
        \item 
                $\fa$ is an integral ideal dividing $\ideal{(1-\zeta_d)}$, 
        \item 
                $\fb_0$ and $\fb_1$ are relatively prime ideals dividing $\fa$ 
                and $\ideal{(1+\zeta_{de})}$.
\end{itemize}

        We will first consider the cases where it is possible that $\fb_0$ and
        $\fb_1$ are non-trivial, which (see Proposition~\ref{prop:rootsGdeer})
        may happen when $r=2$, $e$ is even, and $de=2p^l$. 
        For $\fa$ to be non-trivial, we need that $d$
	    be a prime power, which implies that $d=p^h$ or  $p$ odd, $d=2$. 
        In the second case $1-\zeta_d=2$ is prime to
        $1+\zeta_{de}=1-\zeta_{p^l}$, thus $\fb_0$ and $\fb_1$ are trivial.
        Thus the only cases where $\fb_0$ or $\fb_1$ could be non-trivial is: 

\subsubsection*{The case $r=2$, $e = 2p^l$, $d = p^h$, with $l\ge 0, h\ge 1$.}
\hfill\smallskip

        Let $\fp := \ideal{(1+\zeta_{de})}$. 
        Then $1+\zeta_{de} = 1- (-\zeta_{de})$, and since $de = 2p^{h+l}$,
        $-\zeta_{de}$ has order $p^{h+l}$ if $p$ is odd, and has order $2p^{h+l}$
        if $p = 2$. 
        Thus by Proposition~\ref{zmn}, we have that $\fp$ is prime.
        
        Furthermore, let $e'=e$ if $p=2$, or $e'=e/2$ otherwise.
        Then it follows from Lemma~\ref{lem:m=n}(2) that
        $\ideal{(1-\zeta_d)}=\fp^{e'}$.        
        Hence every ideal dividing $\ideal{(1-\zeta_d)}$ is principal,
        equal to $\fp^n$ for $0\le n\le e'$,
        thus the above root systems are principal. 


\begin{proposition}\label{G(2pq,2p,2) RS}
	Let $e = 2p^l$, $d = p^h$, with $l\ge 0$ and $h\ge 1$,
        and $\fp := \ideal{(1+\zeta_{de})}$.
        Write $e'=e$ if $p=2$, or $e'=e/2$ otherwise. 
        Then every genus of distinguished root system for $G(2p^{h+l},2p^l,2)$
        contains one of the $3e'+1$ root systems:
	$$
                \fa\cdot\fR_1^1(d,r)\, \cup\, \fb_0\cdot\fR_2^0(de,r)
                 \,\cup\, \fb_1\cdot\fR_2^1(de,r)
		\,,
	$$
        where
	$$ 
        		(\fb_0, \fb_1,\fa)  \in
		\{(\ideal{},\ideal{},\ideal{})\}
		\bigcup_{1\le n\le e'}
                \{(\ideal{},\fp,\fp^n),(\fp,\ideal{},\fp^n),(\ideal{},\ideal{},\fp^n)\}
                \,.
	$$
        These root systems are all principal, have a root lattice basis and a coroot lattice basis,
        and their connection index is one of $\ideal{}$, $\fp$ or $\fp^2$.
\end{proposition}

\begin{proof} We consider the cases $p=2$ and $p>2$ separately.

	Suppose that $p=2$. 
	Thus $d=p^h$,  $e=2p^l$ and $e'=e$.
	In this case, $\ideal{(1-\zeta_{de})}= \ideal{(1+\zeta_{de})}= \fp$.
	The following table exhausts the possibilities for the root systems;	
	we use the notation introduced in Remark~\ref{generalremarkgdeer}, as well as
	 ``centre-dot''  notation in the $\fb_0$  and $\fb_1$ columns, which
	indicates that $\fb_0$ (respectively $\fb_1$) may be either $\BZ_k$ or $\fp$.
%
	$$
	\begin{array}{lcc}
		\fb_0,\fb_1,\fa&\text{properties}&\text{connection index}\\
		\hline
		\ideal{},\ideal{},\ideal{}&P_{13},P_{23}& \ideal{(1-\zeta_{de}\inv)}=\fp\\
		\cdot,\cdot,\fp &P_{13}\text{ or }P_{23}& \ideal{(1-\zeta_{de})}=\fp\\
		\cdot,\cdot,\fp^j \mbox{ for } 2 \leq j \leq e-1 &P_{33} & \fp^2\\
		\cdot,\cdot,\fp^{e} &P_{32}\text{ or } P_{31} & \fp\\
		\end{array}
	$$
	Both  $\fp$ and $\fp^2$  occur as connection index, depending on
	the root system, which illustrates that for not well-generated groups, the
	connection index  depends on the root system.
\medskip

	Suppose now that $p>2$. 
	Thus $d=p^h$,  $e=2p^l$ and $e'=p^l$.
	In this case, $\ideal{(1-\zeta_{de})}= \ideal{}$. 
	Again, the table below exhausts the possibilities for the root systems; the notation is as above.
%
	$$
	\begin{array}{lcc}
		\fb_0,\fb_1,\fa&\text{properties}&\text{connection index}\\
		\hline
		\ideal{},\ideal{},\ideal{}&P_{33}&\ideal{}\\
		\ideal{},\fp,\ideal{}&P_{23}&\ideal{}\\
		\fp,\ideal{},\ideal{}&P_{13}&\ideal{}\\
		\cdot,\cdot,\fp^j \mbox{ for } 1\leq j \leq e'&P_{33}&\ideal{}\\
		\end{array}
	$$
\end{proof}

In the remaining cases we know that the genera of root systems are represented
by the system $\fa \fR_1^1(d,r) \cup \fR_2^0(de,r)  \cup \fR_2^1(de,r)$
where $\fa$ runs over integral ideals dividing $\ideal{(1-\zeta_d)}$.

\subsubsection*{The case $d$ composite.}
\hfill
\smallskip

\begin{proposition}\label{d composite}
        Assume $d$ composite. Then $G(de,e,r)$ has a single genus of root
        systems, represented by the principal root system:
	$$\fR_1^1(d,r) \cup \fR_2(de,r) \, ,$$
        which has a root lattice basis and a coroot lattice basis,
        and connection index $\BZ_k$.
\end{proposition}
\begin{proof}
	When $d$ is composite, $1-\zeta_d$ is a unit thus  $\fa$ is trivial.
	So there is a unique genus of root system, which contains the root system 
	given in the statement of the proposition. 
	This root system has root generators $\Pi$, for which
        the Cartan matrix is $C$. The property $P_{13}$ holds
	and the connection index is $\ideal{(1-\zeta_{de}\inv)}=\ideal{}$.
\end{proof}

\subsubsection*{The case $d$ a prime power.}\hfill

\smallskip
	Assume $d=p^a$ with $p$ prime, $a\ge 1$ and write $e=p^hn'$ 
        with $n'$ prime to $p$. Then by
	Proposition~\ref{zmn}, we have the decomposition in prime ideals 
	$\ideal{(1-\zeta_{p^a})} = (\fp_1\ldots\fp_\delta)^{p^h}$, 
	where $\delta=\vp(n')/s$ and $s$ is the
	multiplicative order of $p\mod {n'}$.

        We can say more when $\delta=1$ than in the other cases.

\begin{proposition}
	Let $d = p^a$, $e=p^hn'$ with $p$ prime, $a\ge 1, h\ge 0$ and $n'$
        prime to $p$ and assume that $p$ generates the
        multiplicative group $(\BZ/n'\BZ)^\times$ (this includes the case
        $n'=1$). Let $G := G(de,e,r) = G(p^{a+h}n',p^h n',r)$ and
	$\fp :=\ideal{(1-\zeta_{p^{a+h}})}$.	

        Then (assuming we are not in the case \ref{G(2pq,2p,2) RS})
	every genus of distinguished \zroot\ system
  	for $G$ is represented by one of the $p^h+1$ principal root systems
	$$\fp^n\fR_1^1(d,r) \cup \fR_2^0(de,r)  \cup \fR_2^1(de,r)$$
 	where $0\leq n \leq p^h$. These systems have a root lattice basis
 	and a coroot lattice basis and connection index one of $\ideal{}$,
        $\fp$ or $\fp^2$.
\end{proposition}

\begin{proof}
        Since $\delta=1$, the ideal $\fp$ is prime and
	$\ideal{(1-\zeta_d)}=\ideal{(1-\zeta_{p^a})} = \fp^{p^h}$, thus
        the divisors $\fa$ of $\ideal{(1-\zeta_d)}$ are the principal ideals
        $\fp^n$ as described in the proposition.

	The linear dependence relations noted in Remark~\ref{generalremarkgdeer}
	are still valid, that is if $n'\ne 1$ then $P_{33}$ holds and the
        connection index is $\ideal{}$. If $n'=1$ then
        $\ideal{(1-\zeta_{de})}=\fp$ and if
        $\fa=\ideal{}$ then $P_{13}$ holds and the
        connection index is $\fp$, 
	if $a=\fp^{p^h}$ then $P_{31}$ holds and the connection index is $\fp$, 
	and if $a=\fp^n$ with $0<n<p^h$ then $P_{33}$ holds and the connection
        index is $\fp^2$.
\end{proof}

        When $\delta>1$ we can say less. There are
        $(p^h+1)^\delta$ distinct ideals $\fp_i$ dividing $\ideal{(1-\zeta_d)}$,
	thus $(p^h+1)^\delta$ genera of distinguished \zroot\ systems.

	However, note that by Example~\ref{example nonprincipal}, for $G(39,3,3)$ 
        for example the $\fp_i$ may be non-principal, which gives rise to 
        genera containing no principal root system, which illustrates the 
        failure of Theorem~\ref{allprincipal}
	for non well-generated reflection groups.

	There are still principal root systems for every group $G(de,e,r)$,
	such as:
	$$\fR_1^1(d,r) \cup \fR_2(de,r),$$
	which is distinguished, and:
	$$\fR_1(d,r) \cup \fR_2(de,r),$$
	which is complete and reduced.
 	In particular, for all $d,e>1$ we can fulfill 
        the promise of Remark \ref{rem:connectionbadgenerated}:
	
\begin{proposition}\label{2connectionindices}
	The principal  distinguished \zroot\ system 
	$$\fR_1^1(d,r) \cup \fR_2(de,r)$$
        for $G(de,e,r)$ $(d>1,\,e>1)$ has a root lattice basis and a coroot lattice basis.
	Its  connection  index is $\ideal{(1-\zeta_{ed}\inv)}$.
\end{proposition}
\begin{proof} See proof of Proposition \ref{d composite}.
\end{proof}

\section{\red{Reflection groups and root systems over $\BR$}}

\subsection{Preliminary: positive Hermitian forms}\habel{positiveforms}

        Let $\operatorname{Herm}(V;k)$ denote the $k$-vector space of Hermitian
        forms on $V$. 
         Say that $\varphi: V \ra W$ is \emph{Hermitian} if for any
         $g\in\GL(V)$ we have $\varphi\circ g=g^\vee\circ\varphi$ (equivalently
         $\varphi$ can be represented by an Hermitian matrix with respect to a
         basis $(e_i)_{i \in E}$ of $V$ and 
        its dual basis $(f_i)_{i \in E}$ of $W$).
        Let $\operatorname{Herm}(V,W)$ denote the $k$-vector space of Hermitian maps from 
        $V$ to $W$.
        Since the Hermitian pairing $V\times W \ra k$ is non-degenerate, the linear map
        $$
        \left\{
        \begin{aligned}
                &\operatorname{Herm}(V,W) \ra \operatorname{Herm}(V;k) \,, \\
                &\vp \mapsto \left(\, (v_1,v_2) \mapsto (v_1\mid v_2)_\vp := \scal{v_1}{\vp(v_2)} \, \right)
        \end{aligned}
        \right.
        $$
        is an isomorphism. Moreover, the Hermitian form $(\cdot\mid\cdot)_\vp$ is:
\begin{itemize}
        \item
                non-degenerate if and only if $\vp$ is an isomorphism,
        \item 
                positive (resp. positive definite) if and only if for all $v \in V -\{ 0\}$,
                $\scal{v}{\vp(v)} \geq 0$ (resp. $\scal{v}{\vp(v)} > 0$), in which case we 
                say that $\vp$ is \emph{positive} (resp. \emph{positive definite}).
\end{itemize}
        
        Let $G$ be a finite group of $\GL(V)$ (which may be viewed as a finite subgroup
        of the subgroup of $\GL(V)\times\GL(W)$ which preserves the pairing).
        Then there exists a positive definite Hermitian $kG$-isomorphism $\vp : V \iso W$:
        given any basis ($e_1,\dots,e_r$) of $V$ and its dual basis ($f_1,\dots,f_r$)
        of $W$ (that is, $\scal{e_i}{f_j} = \delta_{i,j}$), the isomorphism $\vp : e_i \mapsto f_i$
        is positive definite, and its average $\frac{1}{|G|}\sum_{g \in G}
        g^\vee\vp g\inv$ is both $G$-stable and positive definite.

\begin{remark}
        If $V$ is an absolutely irreducible $kG$-module, the vector space of $kG$-morphisms
        $V\ra W$ is one dimensional, hence the space of $G$-invariant Hermitian forms on $V$
        is also one dimensional.
        
        In particular then, the trivial positive definite quadratic form on $V$ need not be 
        invariant under $G$:
        consider the case where $k= { \BQ(\sqrt{5})}$  and $G$ is the dihedral group of order 10.
        Then there is a $G$-invariant symmetric bilinear form on $V = k^2$ whose discriminant
        is $(5-\sqrt{5})/8$ -- since it is the determinant (up to a square in $k^\times$) of the matrix:
        $$\begin{pmatrix} 1 & \cos(2\pi/5) \\ \cos(2\pi/5) & 1\end{pmatrix}.$$
\end{remark}

        The proof of the following lemma is immediate.

\begin{lemma}\label{identifying}
        Assume that $s$ is a reflection in $G$, whose reflecting line, reflecting hyperplane
        dual reflecting line, dual reflecting hyperplane are respectively $L$, $H$, $M$, $K$
        (see Definition~\ref{reflectingspaces}).
        Let $\vp : V \iso W$ be any positive definite Hermitian $G$-stable isomorphism. Then
\begin{enumerate}
        \item
                $H$ is the orthogonal of $L$ for the Hermitian form $(\cdot\mid\cdot)_\vp$ on $V$,
        \item
                $M$ is the orthogonal of $H$ for the Hermitian pairing $\scal{\cdot}{\cdot}$,
        \item
                $\vp(L) = M$ and $\vp(H) = K$, and
        \item
                for $x \in L$ and $v\in V$,
                $$
                        s(v) = v - \dfrac{(v\mid x)_\vp}{(x\mid
                        x)_\vp}(1-\zeta) x
                        \,.
                $$                      
\end{enumerate}
\end{lemma}
\bigskip

\subsection{Families of simple reflections for real reflection groups}\habel{realfields}
\smallskip

        In this subsection, we assume that $k$ is a subfield of the field $\BR$ of real numbers.        
        We denote by $\BR^+$ (resp. $\BR^-$) the set of nonnegative
        (resp. nonpositive) real numbers, 
        and we set 
        \index{kplus@$k^+$}\index{kminus@$k^-$}
        $k^+ := \BR^+ \cap k$ and $k^- := \BR^- \cap k$.
        Thus $k = k^+ \cup k^-$ and $k^+ \cap k^- = \{0\}$.

        Let $V$ be a finite dimensional $k$-vector space, 
        and let $G$ be a finite subgroup of $GL(V)$.
        Let $S$ be the set of reflections of $G$.
        Notice that the determinant of a reflection is always $-1$, and so the
        reflections $s\in S$ are in bijection with pairs $(L_s,M_s)$ where $L_s$
        is the reflecting line of $s$ (a line in $V$) while $M_s$ is the dual reflecting
        line of $s$ (a line in $W$).
        
        We fix a positive definite Hermitian  $kG$-isomorphism $\vp : V \iso W$ throughout.
\smallskip

\begin{definition}\label{preorder}
        We call {\em admissible preorder} on $V$ and $W$ a preorder
        obtained as follows: \index{admissible preorder}

        We choose a nonzero element $v_0 \in V$ which belongs to no
        reflecting hyperplanes $H_s\subset V$ for $s\in S$, and such that 
        $\varphi(v_0)$ belongs to no reflecting hyperplane $K_s\subset W$.

	Such a choice induces a preorder on $W$ and on $V$ by:
        
\begin{itemize}
        \item   
                defining $w \in W$ to be \emph{positive} (which we denote
                $w >0$) if $\scal {v_0}w >0$  
                and to be \emph{negative} if $-w$ is positive, and 
        \item
                for $v\in V$ defining $v>0$, if $(v\mid v_0)_\vp >0$, 
\end{itemize}
\end{definition}

		 The relation ``$v_1>v_2$ whenever $v_1-v_2>0$" is preserved under 
		 vector addition and positive scalar multiplication, and so 
		  makes $V$ (respectively $W$) a preordered vector space.
       
                For $s\in S$ we set \index{Lplus@$L^+_s$}\index{Lminus@$L^-_s$} 
                $L_s^+ := \{v \in L_s \mid v >0\} \sqcup \{0\}$  and $L_s^- :=
                -L_s^+$; we have $L_s=L^+_s\sqcup \{0\}\sqcup L^-_s$ .
                We define $M^+_s$ and $M^-_s$ similarly.

        Thus an admissible preorder determines a family of \emph{positive half-lines}
        \index{Family of positive half-lines}
        $(L_s^+)_{s\in S}$
        in $V$ and a family of \emph{positive half-lines}\index{Positive half-lines}
        $(M_s^+)_{s\in S}$ in $W$.

\begin{definition}\label{defsimpleroot}
        Given a positive definite Hermitian $kG$-isomorphism 
        $
                        \vp : V \iso W
                \,,
        $
        and an admissible preorder on $V$ and $W$, 
        a family $\Si$ of reflections of $G$ satisfying:
        \begin{enumerate}[label=(\Alph*)]
        \item
                $V = \bigoplus_{\si \in \Si} L_\si$ and $W = \bigoplus_{\si \in \Si} M_\si$, and
\smallskip

        \item
                for all $s \in S$,
                $L_s^+ \subset \sum_{\si \in \Si} L_\si^+$ and
                $M_s^+ \subset \sum_{\si \in \Si} M_\si^+$,
\end{enumerate}
         is called a
        \emph{family of simple reflections\/} for $G$.
        \index{Family of simple reflections}
        \index{simple reflections}
\end{definition}

\begin{proposition}\label{existencesimpleroots}
        Let $G$ be a finite subgroup of $\GL(V)$, let $S$ be the set of reflections of $G$,
        let $\vp : V \iso W$ be a positive definite Hermitian $kG$-isomorphism, 
        and suppose given an admissible  preorder on $V$ and $W$.
        Then there exists a unique family of simple reflections for $G$.
\end{proposition}
        Thanks to Proposition \ref{existencesimpleroots} it makes sense to say
        that a set of reflections is a ``set of simple reflections'' if it is
        the set of simple reflections determined by an admissible preorder.

        Notice that all assertions concerning $W$ are analogous to those concerning
        $V$, so from now on we only state (and prove) assertions concerning $V$.

        We remark that there are subsets of $S$ which satisfy the criterion (2) 
        of the definition of a family of simple reflections -- as indeed, $S$ itself has that property. 
        It turns out that a minimal such subset is precisely a family  of simple reflections.
        
\begin{lemma}\label{scalnegative}
        Suppose that $\Si \subseteq S$ is minimal subject to satisfying:
        \begin{enumerate}[label=(\Alph*)]
        \setcounter{enumi}{1}
        \item 
                for all $s \in S$,
                $L_s^+ \subset \sum_{\si \in \Si} L_\si^+$. 
        \end{enumerate} 
        Then 
        \begin{enumerate}
        \item For all distinct reflections $\si_1, \si_2\in \Sigma$ we
        have $(L^+_{\si_1}\mid L^+_{\si_2})_\vp \subset k^-$.
        \item
        For distinct reflections $\sigma_1,\sigma_2\in \Sigma$, and 
        $v_2 \in L_{\si_2}^+$, then 
        $\sigma_1(v_2)=v_2+v_1$ for some $v_1 \in L_{\si_1}^+$  
        --- that is,  $\sigma_1(L_{\si_2}^+) =L_{\si_1 \si_2 \si_1\inv}^+$.
        \end{enumerate}
 \end{lemma}

\begin{proof} 
        We prove first (1), arguing by contradiction:
        assume there are $v_1 \in L^+_{\si_1}$ and
        $v_2 \in L^+_{\si_2}$
        such that $(v_1\mid v_2)_\vp > 0$. Thus 
        $$
                \si_1(v_2) = v_2 - 2\dfrac{(v_2\mid v_1)_\vp}{(v_1\mid v_1)_\vp}v_1
                = v_2 -\la v_1 
                \in L_s
        $$
        where  $\la >0$ and $s = \si_1\si_2\si_1\inv$.
        Either $\sigma_1(v_2)$ is positive  or negative. 

        Suppose first that $\si_1(v_2)$ is positive. Then 
        $\si_1(v_2) = \sum_{\si\in\Si} v_\si$ where each $v_\si \in L_\si^+$. 
        There is some $\la_{\si_2} >0$ such that  $v_{\sigma_2}=\la_{\si_2} v_2$, 
        hence 
        $$
                \sum_{\si\in\Si-\{\si_2\}} v_\si + \la v_1 = (1-\la_{\si_2})v_2
                \,.
        $$
        The expression on the left indicates this is a (strictly) positive vector; whence
        $1-\la_{\si_2} >0$. So $v_2 \in \sum_{\si\in\Si-\{\si_2\}} L_\si^+$,
        a contradiction with the minimality of $\Si$.

        If $\si_1(v_2)$ is negative, then 
        $-\si_1(v_2) = \sum_{\si\in\Si} v_\si$ for $v_\si \in L_\si^+$,
        so 
        $$
                \sum_{\si\in\Si-\{\si_1\}} v_\si + v_2 = (\la-\la_{\si_1})v_1
                \,.
        $$
        where $v_{\sigma_1}=\la_{\si_1} v_1$.
        Since this is an expression for a positive vector,  $\la-\la_{\si_1}>0$, 
        so $v_1 \in \sum_{\si\in\Si-\{\si_1\}} L_\si^+$,
        again contradicting the minimality of $\Si$.

        (2) is an immediate corollary of (1).
\end{proof}

\begin{proof}[Proof of  Proposition \ref{existencesimpleroots}]
        As remarked earlier, there are subsets of $S$ which satisfy the criterion (B) for 
        a family of simple reflections -- including $S$ itself.
        
        Suppose that $\Sigma$ is a minimal subset of $S$ satisfying criterion (B).
        Take any partition $\Si = \Si_1 \sqcup \Si_2$ and a vector $v$ such that:
        $$
                v \in \left( \sum_{\si_1\in\Si_1} L_{\si_1}^+ \right) 
                \bigcap 
                \left( \sum_{\si_2\in\Si_2} L_{\si_2}^+ \right)
                \,.
        $$
         By Lemma \ref{scalnegative}(1), we get $(v\mid v)_\vp \leq 0$; and 
         since the form is definite positive, $v=0$. Thus $V = \bigoplus_{\si \in \Si} L_\si$ 
         (criterion (A)) holds.
        
        Finally, we prove that such a $\Si$ is unique.
        Again we argue by contradiction, by assuming the existence of $\Si' \neq \Si$
        which also satisfies criteria (A) and (B) of Definition~\ref{defsimpleroot}. 
        Choose $\si_0 \in \Si- \Si'$ and $v_{\si_0} \in L_{\si_0} ^+$. 
        There exists a family $(v_{\si'})_{\si'\in\Si'}$ with
        $v_{\si'} \in L_{\sigma'}^+$ such that
        $
                v_{\si_0}  = \sum_{\si'\in\Si'} v_{\si'}
                \,.
        $
        Now for each $\si'$ for which $v_{\si'} \neq 0$,
        there exists a family $(v_{\si,\si'})_{\si\in\Si}$ where
        $v_{\si,\si'} \in L_\si^+ $ such that
        $
                v_{\si'} = \sum_{\si\in\Si} v_{\si,\si'} 
                \,.
        $
        Since $\si' \neq \si_0$, there exists $\si\in\Si$, $\si \neq \si_0$, such that
        $v_{\si,\si'} \neq 0$. Such a vector $v_{\si,\si'}$ appears then in the
        decomposition of $v_{\si_0}$ onto $\bigoplus_{\si\in\Si} L_{\si}$,
        which is a contradiction.
\end{proof}
        
\begin{lemma}\label{positivegivespositive}
        Assume given an admissible preorder
        on $V$ and $W$, that $\sigma$ is a simple reflection 
        and that $s$ is any other reflection in $S$, distinct from $\si$. Then:
\begin{enumerate} 
        \item
                $\si(L_s^+) = L_{\si s\si\inv}^+$,
        \item
                $s(L_\si^+) = L_{s \si s\inv}^+$ if and only if $(L_s^+\mid L^+_\sigma) \le 0$, and
        \item 
                $L_\si^+$ is the only positive half-line made negative by $\sigma$. 
\end{enumerate}
\end{lemma}

\begin{proof}
        For any reflections $s_1$ and $s_2$, $s_1(L_{s_2}) = L_{s_1 s_2 s_1 \inv}$. 
        We have to show in each case that the positive half-line remains positive. 
        
        (1)
        Choose $v_s \in L_s^+$. 
        Take $\sigma_1 \neq s$; then $\sigma_1(v_s) \in L_{\si_1 s \si_1\inv}$. 
        We want to show that it is positive.
        So there exist $v_\sigma \in L_\si^+$ ($\si \in \Sigma$) for which
        $
                v_s = \sum_{\si\in\Si}  v_\si
                \,.
        $
        By Lemma~\ref{scalnegative}(2), 
        $$
                \sigma_1(v_s)=\sum_{\si\in\Si-\{\si_1\}} \left( v_\si + v_{\si,1}\right) + \sigma_1(v_{\si_1})
                \,,
        $$
        where, for all $\si$ simple reflections, $v_{\si,1} \in L_{\si_1}^+$.
        Thus 
        $$
                \sigma_1(v_s)=\sum_{\si\in\Si-\{\si_1\}} v_\si + u_1
                \,,
        $$
        where $u_1 \in L_{\si_1}$. Since $\sigma_1 \neq s$, there is a $\sigma \in \Sigma$ 
        for which $v_\sigma$ is non-zero, that is, strictly positive. 
        So $\sigma_1(v_s)$ is in  $L_{\si_1 s \si_1\inv}^+$, 
        the positive part of $L_{\si_1 s \si_1\inv}$. 
        
        (2) 
        For $v_\sigma \in L_\sigma^+$, we have 
        $s(v_\sigma)=v_\sigma - 2 \frac{( v_\si \mid v_s)}{(v_s\mid v_s)} v_s$ 
        where without loss of generality we may assume that $v_s \in L_s^+$.
        Since $\sigma \neq s$, there is some strictly positive $v_{\si'}$ in the simple positive half-line 
        decomposition of $v_s$, where $\si'\neq \si$. 
        Thus $s(v_\si) \in L_{s \si s\inv}^+$ if and only if $(L_s^+\mid L^+_\sigma) \le 0$.
        
        Item (3) follows from (1).
\end{proof}

\begin{lemma}\label{conjugatehalflines}
        Let $G$ be a finite subgroup of $GL(V)$. 
        Any two families of positive half-lines for $G$ are conjugate under $G$.
\end{lemma}

\begin{proof}
        Suppose we have two families of positive half-lines:
        $\CL_S^+ := (L_s^+)_{s\in S}$ and $\CL_S^{++} := (L_s^{++})_{s\in S}$
        which determine families of
        simple roots  denoted by $\Si$ and $\Si'$ respectively.
        We shall prove by induction on
        $|\CL_S^+ \cap (-\CL_S^{++})|$ that $\Si$ and $\Si'$ are conjugate under $G$.
        
        If $|\CL_S^+ \cap (-\CL_S^{++})| = 0$, 
        then $\CL_S^+ = \CL_S^{++}$ and so $\CL_S^+$ and $\CL_S^{++}$
        are conjugate by the identity element of $G$.
        
        If $|\CL_S^+ \cap (-\CL_S^{++})| > 0$, then there exists $\si_0 \in \Si$ such that
        $L_{\si_0}^+ \in -(\CL_S^{++})$. Indeed, if not, then $\CL_S^+ \subset \CL_S^{++}$, hence
        $\CL_S^+ = \CL_S^{++}$.
        By assertion (1),
        $$
                \si_0(\CL_S^+) = \left( \CL_S^+ -\{L_{\si_0}^+\} \right) \sqcup \{-L_{\si_0}^+\} 
                \,,
        $$
        which shows that $\si_0(\CL_S^+)$ is a family of positive half-lines (for the order conjugate
        under $\si_0)$ such that
        $$
                |\si_0(\CL_S^+) \cap (-\CL_S^{++})| = |\CL_S^+ \cap (-\CL_S^{++})| -1
                \,.
        $$
        By the induction hypothesis, we get that $\si_0(\CL_S^+)$ and $\CL_S^{++}$ are
        conjugate under $G$, which shows that $\CL_S^+$ and $\CL_S^{++}$ are
        conjugate under $G$.
\end{proof}

\begin{proposition}\label{propertiessimple}
        Let $V$ be a finite dimensional vector space on a real field $k$.
        Let $G$ be a finite subgroup of $\GL(V)$ and let $S$ the set of all reflections of $G$.
\begin{enumerate}
        \item
                $G$ acts transitively on the set of families of simple reflections.                             
        \item
                Let $\Si$ be a family of simple reflections.
                \begin{enumerate}
                        \item
                                Every reflection of $G$ is conjugate to an element  of $\Si$.
                        \item
                                $\Si$ generates the (normal) subgroup of $G$ generated by $S$.
                \end{enumerate}         
\end{enumerate}
\end{proposition}

\begin{proof}
        Since the choice of a family of positive half-lines
        determines a single family of simple reflections, item (2) of the preceding lemma shows
        the transitivity of $G$ on the families of positive half-lines.
                
        We now turn to assertion (2)(a). 
        Choose an admissible  preorder on $V$ and $W$, 
        and denote by $\CL_S^+ := (L_s^+)_{s\in S}$ 
        the corresponding
        family of positive half-lines and by $\Si$ the corresponding family of simple reflections.
        For each $\si \in \Si$, we choose $e_\si \in L_\si^+$ so that $(e_\si)_{\si \in \Si}$ is 
        a basis of $V$, and for all $s\in S$ and $v \in L_s^+$, we have
        $
                v = \sum_{\si \in \Si} \la_\si(v) e_\si
        $
        with $\la_\si(v) \geq 0$ for all $\si \in \Si$. We set
        $$
                h(v) := \sum_{\si \in \Si} \la_\si(v)
                \,.
        $$      
        For each $s\in S$, we choose an element $v \in L_s$, $v \neq 0$, and
        we denote by $\Omega$ the
        union of the orbits of these vectors under $G$. Since $S$ and $G$
        are finite, so is $\Omega$. 

        Denote by  $\Omega^+$ the positive vectors of $\Omega$,
        and  define:
        $$
                m_\Omega := \min \{ h(v) \mid v\in \Omega^+ \}
                \,.
        $$
        
        Let $s \in S - \Si$. Let $v =  \sum_{\si \in \Si} \la_\si(v) e_\si\in
        \Omega\cap L_s^+$. 
        Since 
        $$
                (v\mid v)_\vp  = \sum_{\si \in \Si} \la_\si(v) (v\mid e_\si )_\vp > 0
                \,,
        $$
        there exists $\si_0 \in \Si$ such that $(v\mid e_{\si_0})_\vp > 0$. 
        By item (1) of Lemma~\ref{positivegivespositive}, we know that
        $
                v' := \si_0(v) \in \Omega^+
                \,.
        $
        Since 
        $
                v' = v - 2\frac{(v \,\mid\, e_{\si_0})_\vp}{(e_{\si_0} \ \!\!\mid\, e_{\si_0})_\vp} e_{\si_0}
                \,,
        $
        we see that $h(v') < h(v)$.
        
        This proves in particular that if, for $v \in L_s$, $h(v) = m_\Omega$, then $s \in \Si$,
        and finally that there exists $g \in G$ such that $gsg\inv \in \Si$, which is (2)(a).
        
        This also proves that every element of $S$ is conjugate to an element of $\Si$
        by an element of the group generated by $\Si$, which proves
        that $\Si$ generates the subgroup of $G$ generated by $S$, \ie\ assertion (2)(b).
\end{proof}

        The notion of \emph{Coxeter system} is defined in \cite[Chap. IV, \S 1,
        3, D\' ef. 1.3]{bou} .

\begin{theorem}\label{coxeter}
        Let $V$ be a finite dimensional vector space on a real field $k$.
        Let $G$ be a finite subgroup of $\GL(V)$ generated by reflections, 
        and let $S$ be the set of all reflections of $G$.
        Then the pair $(G,\Si)$ is a Coxeter system for every  
        is a family of simple reflections $\Si$.
\end{theorem}

\begin{proof}
        For all $\si \in \Si$, we set
        $
                P_\si := \{ g\in G\mid g\inv(L_\si^+) \mbox{ is positive }\}
                \,.
        $
        We shall check that $P_\si$ satisfies the hypotheses of \cite[Ch. IV,
        \S 1, 7, Proposition 6]{bou} and this will prove that
        \begin{itemize}
                \item
                        $(G,\Si)$ is a Coxeter system,
                \item
                        $P_\si$ comprises of all elements $g\in G$ such that 
                        $l_\Sigma(\si g ) > l_\Sigma(g)$,
        \end{itemize}
		where $l_\Sigma(g)$ denotes the length of the shortest decomposition of $g$ 
		in terms of simple reflections.
		
        It is clear that $P_\si \cap \si(P_\si) = \emptyset$. 
        Now take $g \in P_\si$ and  $\si' \in \Si$
        such that $g\si' \notin P_\si$, that is 
        $g\inv(L_\si^+) $ is positive and $\si'g\inv(L_\si^+)$ is negative. As $\si'$ only changes the
        sign on $L_{\si'}$, we must have $g\inv(L_\si) = L_{\si'}$, which implies 
        $g\inv \si g = \si'$.
\end{proof}

        When the family of simple reflections $\Sigma$ is clear from the context, 
        we will just write $l(g)$ for the length of an element $g$ of $G$ 
        with respect to the generating set $\Sigma$.

\begin{corollary}\label{length}         
        If  $g\in  G$ and  $\sigma$ a simple reflection,  then  $g(L_\sigma^+)>0$  is
        equivalent to $l(g\sigma)=l(g)+1$.
\end{corollary}

\subsection{Highest half-lines}\hfill
\smallskip

        The definition of the highest root given in \cite[Ch. VI, \S1.8]{bou} seems
        \apriori\ not to make sense in our setting. For instance, for the group
        $G(5,5,2)$ we have $\BZ_k=\BZ[\phi]$ 
        where $\phi=\frac{1+\sqrt 5}2 >1$
        is a unit of $\BZ_k$.
        As a root is only defined up to multiplication by a unit, it will have
        no well-defined ``length''.
        
        However, the following, a consequence of the definition in \cite{bou}, does make sense:
        
\begin{definition}\label{def:highesthalfline}
        Let $\CL_S^+$ be a family of positive half-lines.
        We  call $L^+$ in $\CL_S^+$ a {\em highest half-line}   \index{Highest half-line} 
        if $(L^+\mid L^+_s)\ge 0$ for all $s\in S$.
\end{definition}

\begin{proposition}\label{ll}
        If  $l_\Sigma(s)$ is maximal over the
        conjugacy class of $s$ in the Coxeter system $(G,\Sigma)$, 
        then $L^+_s$ is a highest half-line.
\end{proposition}

\begin{proof}  
        Let $\Sigma$ be a family of simple reflections and $\sigma \in \Sigma$.
        We want to show that for any positive half-line $L_s^+\in\CL_S^+$, 
        if $(L_s^+\mid L_\sigma^+)<0$  then $l(\sigma s  \sigma)>l(s)$. 
        
        Indeed, suppose otherwise:
        since  $l(s \sigma)=l(s)+1$ we must have $l(\sigma s \sigma)=l(s \sigma)-1$.
        By Corollary~\ref{length} this implies $(s \sigma)\inv(L_\sigma^+)=\sigma s(L_\sigma^+)<0$,
        thus  $\sigma$ changes the sign of $s(L_\sigma^+)$. 
        By (3) of Lemma~\ref{positivegivespositive}, $L_\sigma^+$ is the
        only half-line changed sign by $\sigma$ so we have
        $s(L_\sigma^+)=L_\sigma^+$. Finally, 
        this implies   $(L_s^+\mid     L_\sigma^+)=0$,
        a contradiction.
\end{proof}

The next lemma is a particular case of Theorem \ref{uniquehighest} below.

\begin{lemma}\label{onedihedral} 
        In a dihedral  group, there is exactly one highest half-line 
        in each conjugacy class.
\end{lemma}

\begin{proof}   
        This follows from Proposition \ref{ll} and
        the fact that there is exactly one
        reflection of longest length in each conjugacy class of reflections.
\end{proof}

\begin{proposition}\label{conjiffD}  
        In  a finite Coxeter system two  non-commuting reflections are conjugate if
         and only if they are conjugate in the dihedral group they generate.
\end{proposition}

\begin{proof}
        Let  $s$  and  $t$  be  two  non-commuting  reflections. Let $P$ be the
        fixator  of the intersection of the two reflecting hyperplanes. Then $P$
        is  a parabolic subgroup of $G$ of rank  at most two. It is of rank two
        since  $s$  and $t$ are in $P$ and do
        not  commute. 
        Up to  conjugacy, we may  suppose that $P$  is a standard
        parabolic  subgroup,  so is  defined by two  vertices $s_1$ and $s_2$ of
         the  Coxeter diagram  (note  that $s_1$  and $s_2$ may be different
        from $s$ and $t$). Then $s_1$ and $s_2$ must be adjacent in the Coxeter
        diagram,  otherwise $P$ is of type $A_1\times A_1$ and does not contain
        non-commuting  reflections.  It  follows  from  the fact that $s_1$ and
         $s_2$  are adjacent and the description (see Theorem~\ref{linear}) 
         of linear characters of Coxeter
        groups that  any  linear  character  of  $P$  extends to $G$. Furthermore,
         Theorem~\ref{linear} ensures that linear   characters   separate   conjugacy   classes  of
         reflections, so it follows that $s$ and $t$  are conjugate if and only if
        they are conjugate in $P$, which is a dihedral group.

        Finally, if $H$ is a dihedral subgroup of a dihedral group $G$, two
        reflections  non-conjugate  in  $H$  are  not conjugate in $G$: for
        $H$ to have two classes of reflections, it has to have an
        even bond, and then $H$ has as many linear characters as $G$.
\end{proof}

\begin{lemma}\label{all positive}
        Assume $G$ irreducible, and that $L_s^+$ is a highest half-line.
        Then for any  positive half-line $L_{s'}^+$, we have $(L_s^+\mid L_{s'}^+)>0$.
\end{lemma}

\begin{proof}
		%
        Let $v\in L_s^+$ and write 
        $v=\sum_{\sigma\in\Sigma} \lambda_\sigma(v) e_\sigma\in L_s^+$
        where the $e_\sigma$ and $\lambda_\sigma(v)\ge 0$ are as in the proof of 
        Proposition~\ref{propertiessimple}. Let $I\subset\Sigma$ be the set of $\sigma$
        such that $\lambda_\sigma(v)>0$, and let $J=\Sigma-I$, containing reflections
        $\si'$ for which $\lambda_{\si'}(v)=0$. 

        Now for $\sigma'\in J$, we have $(v\mid e_{\sigma'})_\varphi\ge 0$ by definition
        of a highest half-line, but also $(e_{\sigma}\mid e_{\sigma'})\le 0$ for any
        $\sigma\in I$ by Lemma~\ref{scalnegative}(1). It follows that we must have
        $(e_{\sigma}\mid e_{\sigma'})=0$ for any $\sigma\in I,\sigma'\in J$, which
        contradicts the irreducibility of $G$ unless $J=\emptyset$.
        
        Thus $(L_s \mid L_\sigma) >0$ for all simple $\sigma$; now any $v_s' \in L_{s'}^+$
        is a non-negative linear combination of $e_\sigma$, with at least one non-zero coefficient,
        and so the result follows.
\end{proof}

\begin{theorem} \label{uniquehighest}
If $G$ is a finite real reflection group,
        there is exactly one highest half-line in each conjugacy class of
        reflections.
\end{theorem}

\begin{proof}
        Suppose $L_s^+$ and $L_{s'}^+$ are distinct highest half-lines in the same
        conjugacy class. We will derive a contradiction.

         Two  conjugate reflections belong  to the same  irreducible component of
        $G$,  so  we  may  assume  $G$  irreducible.  Then,  by  Lemma~\ref{all
        positive}, $L_s^+$ and $L_{s'}^+$ are not orthogonal, so $s$ and $s'$ do
         not  commute. Thus, by  Proposition \ref{conjiffD}, they  are conjugate in
        the  dihedral  subgroup  they  generate;  and $L_s^+$ and $L_{s'}^+$ are
        still  highest  half-lines  in  this  subgroup.  This  contradicts
        Lemma~\ref{onedihedral}.
\end{proof}

\subsection{Real root systems}\habel{subsec:realrootsystems}

        Our approach allows us to extend the theory of root systems for Weyl
        groups to root systems for finite Coxeter groups. We spell out in this
        subsection how our definitions translate in this case.
        As in Subsection \ref{realfields}, we assume $k\subset \BR$, and
        we use the notation $k^+$ and $k^-$.
        
        Let $\fR$ be a \emph{reduced} \zroot\ system, and let $G := G(\fR)$.
        For $s$ a reflection in $G$ we denote by $\frr_s$ the element of $\fR$ associated
        with $s$, and for $\frr\in \fR$ we denote by $s_\frr$ the reflection of $G$ determined 
        by $\frr$.
        
        We assume chosen a positive definite Hermitian $kG$-isomorphism 
        $\vp : V \iso W$, and $v_0$ a nonzero element of $V$ defining an order on $V$
        and $W$ (as in Subsection \ref{realfields}). 
        For 
        $\frr = (I_\frr,J_\frr,\zeta_\frr) \in \fR$, we set
        $$
                I_\frr^+ :=  I_\frr \cap V^+ \quad\text{and}\quad J_\frr^+ = J_\frr \cap W^+
                \,.
        $$
        
        If $\Si$ is the family of simple reflections determined by the choice of the order on $V$
        (see Proposition \ref{existencesimpleroots}), we set
        $$
                \fR_\Si := \left\{ \frr \in \fR \mid s_\frr \in \Si \right\}
                \quad\text{and}\quad
                \fR_\Si^\vee:= \left\{ \frr^\vee \in \fR^\vee \mid s_\frr \in \Si \right\}
                \,.
        $$

\begin{theorem}\label{simplebasis}
        Under the above hypotheses and notation,
\begin{enumerate}
        \item
                $\fR_\Si$ is a root basis and $\fR_\Si^\vee$ is a coroot basis,       
        \item
                Whenever $\frr \in \fR$,
                $$
                        I_\frr^+ \subset \bigoplus_{\si\in\Si} I^+_{\frr_\si}
                        \,.
                $$
\end{enumerate}
\end{theorem}

\begin{proof}
        Assertion (1) results from the fact that $\Si$ generates $G$ 
        (Proposition \ref{propertiessimple}, (2)(b))
        and from Proposition \ref{prop:RootBasis}, (2). 
        Assertion (2) is an immediate consequence of Proposition \ref{existencesimpleroots}, (2).
\end{proof}
\smallskip

\section{\red{Bad numbers}}
\smallskip

        Let $W$ be an irreducible Weyl group. Then
\begin{enumerate}
        \item
                By \cite[Ch.vi, \S 2, Proposition 7]{bou},
                $$
                        |W| = r!c_W(n_1\cdots n_r) 
                         \,.
                $$
                where $n_1,\ldots,n_r$ are the coefficients of the longest root 
                on the basis of simple roots, 
                and $c_W$ is he connection index of $W$.
        \item
                The set of \emph{bad primes\/} for $W$ is defined as the set 
                of prime divisors of the product $n_1\cdots n_r$. 
\end{enumerate}
        The above definition for the set of bad primes is equivalent to several others, see
         \cite[4.3(c)]{spst}.

       Theorem~\ref{th:connectionindex} ensures that the connection index for any 
       irreducible, well-generated complex reflection group $(V,G)$ is well-defined, 
       independent of the choice of root system. 
       Moreover, by Propositions~\ref{ceen} and~\ref{cd1n}, and the fact that the rings of integers
       for the fields of definition of the primitive reflection groups are principal, we know
       that the connection index of an irreducible, well-generated complex reflection group
       is always principal. We write $c_G$ for a generator of the connection index, 
       which is then well-defined up to a unit.
       This allows us to extend the definition of bad primes as follows:  

\begin{theodef}\label{def:badprime}
        For every irreducible well-generated complex reflection group $G$ of rank
        $r$:
\begin{enumerate}
        \item 
                $c_G r!$ divides $|G|$.
         \item
                We define the {\em bad prime ideals} 
                \index{bad prime ideals}
                for $G$ as the set of prime ideals of 
                $\BZ_k$ which divide $|G|/(c_G r!)$.
\end{enumerate}
\end{theodef}

\begin{proof}
        The proof of (1) is by inspection, using the values of $c_G$ given in
         the tables of Appendix \ref{table} for primitive well-generated groups, and
        in \ref{cd1n} and \ref{ceen} for imprimitive groups.
\end{proof}

        Now let $W$ denote an irreducible \emph{spetsial group\/} -- that is, 
        either a member of one of the imprimitive families $G(e,e,r)$, $G(d,1,r)$ or 
        a primitive complex reflection group denoted $G_n$ where  
        $$n \in \{4,6,8,14,23,24,25,26,27,28,29,30,32,33,34,35,36,37\}$$
        according to the Shephard-Todd notation. 
        (Note that spetsial groups are all well-generated.)
        For the definitions relative to the program 
        \emph{``Spetses''\/}, we refer the reader to \cite{spe2}.

\begin{definition}\label{def:badnumber}
         We denote by $\Bad_W$ 
         \index{BadW@$\Bad_W$}
         the largest integral ideal $\fa$ such that, 
         whenever $S(x) \in \BZ_k[x,x\inv]$ is a Schur element of an
         irreducible character of the spetsial Hecke algebra of $W$, then
         $ S(x)\fa\inv \subset\BZ_k[x,x\inv]$.
\end{definition}
        
        An  equivalent definition,  which can  be more  convenient to  use, is that
        whenever  $D(x)\in  k[x]$  is  a  unipotent  degree of the principal series
        attached  to  $W$,  then  $\fa  D(x)\subset  \BZ_k[x]$.         
\begin{conjecture}\label{badunip}
        $\Bad_W$ can also be characterized as the largest ideal $\fa$ such 
        that whenever  $D(x)\in  k[x]$  is  any  unipotent  degree
        attached  to  $W$,  then  $\fa  D(x)\subset  \BZ_k[x]$.
\end{conjecture}

\begin{remark}
	We were unable to answer the following question:
	is $\Bad_W$ always a principal ideal?
\end{remark}

        It is apparent from the values of unipotent degrees computed by
        Lusztig that the following theorem holds
        in the case of Weyl groups.
        
\begin{theorem} For any spetsial complex reflection group $W$,  
\begin{enumerate}
        \item
                $\Bad_W$ divides $|W|/(r! c_W)$.
        \item
                All bad prime ideals divide $\Bad_W$.
\end{enumerate}
        or, equivalently,
        
        $\Bad_W$ divides $|W|/(c_W r!)$ and $|W|/(c_W r!)$ divides a power of $\Bad_W$.
\end{theorem}
   It follows that the ``bad primes'' defined in \cite[above Proposition 1.31]{Broue-Kim}
   are the same as ours.
\begin{proof}
        For the primitive  spetsial reflection groups, the proof is by inspection of:
\begin{itemize}
        \item 
                the tables of unipotent degrees of the principal series
                given in \cite{spe2}, and
        \item 
                the values of $c_W$ given in the tables of Appendix~\ref{table}.
\end{itemize}
        The values of $\Bad_W$ are reported in the tables of Appendix~\ref{table}.
        In  every case,
        $\Bad_W$ is  also  equal to the largest
        integral   ideal  $\fa$  such  that  for  {\em all}
        unipotent degrees $\fa D(x)\subset\BZ[x]$; that is,
        Conjecture \ref{badunip} holds.

\medskip

        We now prove the theorem for the imprimitive groups $G(d,1,r)$.
        Set $\zeta := \exp{(2 \pi i /d)}$. By Proposition~\ref{cd1n}, $c_W=1-\zeta$,
        thus $|W|/(c_W r!)=d^r/(1-\zeta)$. To compute $\Bad_W$, we use the
        formula for the unipotent degrees given in \cite[3.8]{ma}. Let us
        recall the setup.

        \begin{itemize}
        \item 
        The  unipotent degrees  of the  principal series  of $G(d,1,r)$ are
        parameterized  by the $d$-symbols  $S=(S_0,\ldots,S_{d-1})$ of rank
        $r$  and  of  shape $(|S_0|,\ldots,|S_{d-1}|)=(m+1,m,\ldots,m)$ for
        $m\in\BN$ large enough.

       	\item 
		The unipotent degree attached
        to $S$ is of the form $P_S(x)/f_S$ where 
        \begin{itemize}[label=$\rhd$]
        \item $P_S(x)\in \BZ_k[x]$ is a monic
        polynomial and 
        \item $ f_S=\tau(d)^m/ (\prod_{0\le i<j\le
        d-1}(\zeta^i-\zeta^j)^{|S_i\cap S_j|})$ up to a unit, where
        $\tau(d)=\prod_{0\le i<j\le d-1}(\zeta^i-\zeta^j)$, so that
        \begin{equation}\label{fS} f_S=\prod_{0\le i<j\le d-1}
        (\zeta^i-\zeta^j)^{m-|S_i\cap S_j|} \,. \end{equation} 
        \end{itemize} 
        Due to
        the shape of the symbols, $m\ge |S_i\cap S_j|$, and so  $f_S\in \BZ_k$. 
        Notice
        also  that  $f_S$  depends  only  on  the  equivalence class of $S$, since
        shifting  all the $S_i$ increases  by 1 all the  $|S_i\cap S_j|$, and also
        increases $m$ by $1$, thus leaving invariant $f_S$.
		\end{itemize}
		
        To  prove  item  (1)  of  the  theorem,  we have to show that $f_S$
        divides $|W|/(r!c_W)=d^r/(1-\zeta)$. 
        We will show this by induction
        on  $r$. 
        By \cite[\S 3.C]{ma}, a symbol of
        the  principal series is not 1-cuspidal, thus admits a $(1,1)$-hook;
        that is, any of the considered symbols  of rank $r$ can be obtained
        from a symbol of rank $r-1$ by increasing by $1$ one of the entries
        $\lambda\in  S_i, \lambda+1\notin S_i$ for some  $i$. 

        The effect of this is to reduce by at most $1$ the
        $|S_i\cap S_j|$ for $j\ne i$, that is  to multiply $f_S$
        by at most $\prod_{j\in[0\ldots d-1], j\ne i}(\zeta^i-\zeta^j)$, 
        which is equal to $d$
        up to a unit. But increasing $r$ by $1$ multiplies
        $d^r/(1-\zeta)$ by $d$, so the divisibility of $d^r/(1-\zeta)$ by $f_S$ is
        preserved.
        It remains to show the starting point of the induction, which is
        that when $r=1$, $f_S$ divides $(d/(\zeta-1))$.
        This results from \cite[\S 5.2]{spe2}, which finishes the proof of
        item (1) for $G(d,1,r)$. Note that by
        \cite[\S 5.3]{spe2} it can be seen that Conjecture \ref{badunip}
        holds in this case.

\newcommand\loccit{{\it loc.cit.\kern -2pt }}
        Let  us  prove  item  (2)  for  $G(d,1,r)$.  First, note that for a
        $d$-symbol  $S$ of rank  $r$, if $S'$  is the symbol  of rank $r+1$
        obtained  by increasing by $1$ the highest entry in $S$, then $f_S$
        divides  $f_{S'}$. Indeed,  the numbers  $m-|S_i\cap S_j|$ can only
        increase when going from $S$ to $S'$. Thus it is sufficient to show
        that  for $r=1$ any prime ideal  dividing $d$ (which is the same as
        dividing $d/(1-\zeta)$) divides $\Bad_W$. For
        instance  (see  \cite[\S  5.2]{spe2})  for  the  character  denoted
        $\rho_1$  in \loccit\ we  have $f_S=d/(\zeta-1)$, which proves the
        result   (since  $\ideal{(1-\zeta)}=\ideal{(1-\zeta\inv)}$  divides
        $d/(\zeta-1)$).

\medskip

        To complete the proof,  we consider the case of the groups $G(d, d, r)$ ($r\ge 2$)
        following \cite[\S 6]{ma}. The setup is as follows:
        
        \begin{itemize}
        \item
        The unipotent
        characters of the principal series are parameterized by $d$-symbols
        of rank $r$ and shape $(m,\ldots,m)$. 

        \item
        By \cite[6.4]{ma}, the unipotent degree attached
        to $S$ is of the form $P_S(x)/f'_S$ where $P_S(x)\in \BZ_k[x]$ 
        is a monic polynomial and, up to unit, we have
        $f'_S=f_S \gamma(S)/d$ where $f_S$ is as in equation~\ref{fS} and
        $\gamma(S)$ is the cardinality of the subgroup of $\BZ/d$ leaving $S$
        invariant, where $i\in\BZ/d$ acts on $S$ by mapping it to the symbol
        $S'$ such that $S'_j=S_{(j+i)\pmod d}$.
        \end{itemize}

        This time we have $|W|/(r! c_W)=d^{r-1}/((1-\zeta)(1-\zeta\inv))$
        (see Proposition~\ref{ceen}). To show that $f_S \gamma(S)/d$ divides that number
        is equivalent to  showing that
        \begin{equation}\label{cond fS}
                f_S \gamma(S) \text{ divides } \frac{d^r}{(1-\zeta)(1-\zeta\inv)}
        \end{equation}
        We will show this by a double induction. When $\gamma(S)=1$ we proceed
        by induction on $r$. Just as in the case of $G(d,1,r)$, the symbol
        admits a $(1,1)$-hook and we reduce the problem to the case of rank
        $r-1$. The starting case is $r=2$.

        In  this case  we may  look at  \cite[4.1]{lu} where  the unipotent
        degrees  of the  dihedral groups  $G(d,d,2)$ are  given in the form
        $P_S(x)/f_S$  where  $P$  is  an  integral  polynomial and $f\in \BZ_k$
        divides  $d/((1-\zeta)(1-\zeta\inv))$. The divisibility is obvious
        except when  $d$ is even and $f_S=d/2$, where one needs 
        Corollary \ref{zeta divides} below.
        It  can be  seen also  that for $G(d,d,2)$ Conjecture \ref{badunip}
        holds.

        The other case of the induction is when $\gamma:=\gamma(S)>1$. In this case
        if we set $d'=d/\gamma$, then $S$ is the concatenation of $\gamma$ copies of a 
        $d'$-symbol $S'$ of rank $r':=r/\gamma$.  
        
\begin{lemma} 
                Up to a unit, we have $f_S=f_{S'}^\gamma$.
\end{lemma}

\begin{proof}
        Given $0\le i,j\le d-1$, there are unique expressions
        $i=i'+i''d'$ and $j=j'+j''d'$ for $i',j'\in[0,\ldots, d'-1]$.
        Using that $\zeta^{d'}=\zeta_\gamma$ where $\zeta_\gamma:=\exp{2i\pi/\gamma}$,
        and that $S_i=S_{i'}, S_j=S_{j'}$, we can write
        $$
                f_S=\prod_{0\le i''\le j''\le \gamma-1}
                \prod_{{0\le i',j'\le d'-1}\atop
                {\text{$i'<j'$ if $i''=j''$}}} \left(\zeta^{i'}\zeta_\gamma^{i''}-
                \zeta^{j'}\zeta_\gamma^{j''}\right)^{m-|S_{i'}\cap S_{j'}|}
        $$
        We make the following observations on the above formula:
\begin{itemize}
        \item 
                We can assume $i'\ne j'$ since the terms where $i'=j'$
                have zero exponent.
        \item 
                The term indexed by $i',j'$ is the negative of the term
                indexed by $j',i'$. Thus we can decide to retain only the
                terms where $i'<j'$, up to doubling the exponent when $i''\ne j''$.
\end{itemize}
        Doubling the exponent is compensated by making the product over 
        all $i'',j''$, giving:
        $$
                f_S=\prod_{0\le i'', j''\le \gamma-1}
                \prod_{0\le i'<j'\le d'-1}
                \left(\zeta^{i'}\zeta_\gamma^{i''}-
                \zeta^{j'}\zeta_\gamma^{j''}\right)^{m-|S_{i'}\cap S_{j'}|} .
        $$
        Apply the formula $\prod_{0\le j''\le \gamma-1}(a-b\zeta_\gamma^{j''})=
        a^\gamma-b^\gamma$ to get:
        $$
                f_S=\prod_{0\le i'<j'\le d'-1}
                \left(\zeta^{\gamma i'}-\zeta^{\gamma j'}\right)^{\gamma(m-|S_{i'}\cap S_{j'}|)} ,
        $$
        which is what we want since $\zeta^\gamma=\exp{2\pi i/d'}$.
\end{proof}

        The lemma can be used to deal with the symbols for which $\gamma>1$.
        We distinguish two cases. 
        
        The first case is $r'=1$. 
        Let us recall that $f_S$ is invariant (up to sign) by the action of $\BZ/d'$ ---
        this can be seen directly from its formula or from the fact
        that according to \cite{ma} the whole unipotent degree is invariant (up
        to sign) by that action. Then, up to $(\BZ/d')$-action there is only one
        $d'$-symbol of rank $1$, given by $S'=(\{1\},\{0\},\ldots,\{0\})$.
        A direct computation shows that 
        $f_{S'}=d'$ up to unit. We thus have to show that
        $d^{\prime \gamma}\gamma$ divides $d^r/((1-\zeta)(1-\zeta\inv))$.
        Using that $d^r=(\gamma d')^\gamma$ it remains to see that
        $(1-\zeta)(1-\zeta\inv)$ divides $\gamma^{\gamma-1}$. 
        This follows from item (1) of Corollary \ref{zeta divides} (see
        Appendix \ref{arithmetic}).
        
        The other case is $r'>1$. 
        The case $\gamma=1$, already treated above, ensures that 
        $S'$ satisfies the condition \ref{cond fS}, that is:
        $$
        f_{S'} \gamma(S') \text{ divides } \frac{d'^{r'}}{(1-\zeta')(1-{\zeta'}\inv)}
        $$
        where  $\zeta'=\exp(2 \pi i / d')$. Using the fact that $\gamma(S')=1$, raising both sides to the power of $\gamma$ gives:
        $$
                (f_{S'}^\gamma \gamma) \mbox{ divides } 
                \frac{d^{\prime r' \gamma}\gamma}{(1-\zeta')^\gamma (1-\zeta^{\prime-1})^\gamma}
                =
                \frac{d^{\prime r}\gamma}{(1-\zeta')^\gamma(1-\zeta^{\prime-1})^\gamma}
        $$ 
         It suffices now to show that
        $d^{\prime r}\gamma/((1-\zeta')^\gamma(1-\zeta^{\prime-1})^\gamma)$ divides
        $d^r/((1-\zeta)(1-\zeta\inv))$. Using that $d^r=d^{\prime r}\gamma^r$
         and simplifying, it suffices to show that $(1-\zeta)(1-\zeta\inv)$
        divides $\gamma^{r-1}(1-\zeta')^\gamma (1-\zeta^{\prime-1})^\gamma$, which is
        an immediate consequence of Corollary \ref{zeta divides} since $r>1$.

        It remains to prove item (2) for $G(d,d,r)$. By the same
        argument as for $G(d,1,r)$ given a symbol $S$ of rank $r$ we may find
        a symbol $S'$ of rank $r+1$ such that $f_S$ divides $f_{S'}$.
        We proceed by induction on the rank, starting from the base case $r=2$.
        
        For $d\notin\{2,3\}$, according to
        \cite[4.1]{lu},  for $G(d,d,2)$ there is a symbol $S$ such that
        $f_S = d/((1-\zeta)(1-\zeta\inv))$. 
        Then item (2) of Corollary \ref{zeta divides}   
        completes the proof in this case.
                
        For $d=2$, the group $G(2,2,2)$
        is not irreducible and we do not have to consider it.  
        The group $G(2,2,3)$ is the Weyl group of type $A_3$, and
        $|W|/(r! c_W)=1$ and there is nothing to prove. We start the induction
        at $G(2,2,4)$, the Weyl group of type $D_4$ and for the symbol
        $S=(\{1,2\},\{0,3\})$, for instance, we find $f_S=2$.

        For $d=3$, the group $G(3,3,2)$ is the Weyl group of type $A_2$, and
        $|W|/(r! c_W)=1$ and there is nothing to prove. We start the induction
        at $G(3,3,3)$, and for the symbol
        $S=(\{0,1\},\{1,2\},\{0,2\})$, for instance, we find $f_S=(1-\zeta_3)$
        up to a unit.
\end{proof}

\section{\red{Classification of distinguished root systems for irreducible
primitive reflection groups}}\label{sec:class}
\smallskip

        As noticed previously (see for example \cite{nebe}), it can be checked
        that whenever
        $G$ is a primitive irreducible reflection group, its field of definition
        $k = \BQ_W$ (see \ref{bessis}) has class number 1, \ie\ the ring $\BZ_k$
        is a principal ideal domain.

        In  this  case,  every  root  system  is  principal in the sense of
        Definition~\ref{def:princrootsys}, and hence gives rise to a Cartan
        matrix as in Definition~\ref{cartanmatrix}.

        We   present  here  a  classification  of  the  \emph{distinguished
        \zroot\  systems for primitive groups} (up to genus), based on
        the  data in the \CHEVIE\ package of \GAP. 
        The classification may be summarised by looking at Appendix~\ref{table}, 
        which exhibits, for each primitive irreducible reflection group $G$:
        \begin{itemize}
        \item a diagram describing its presentation;
        \item a Cartan matrix $C$ which corresponds to the data in \CHEVIE\ 
        as well as diagonal matrices giving Cartan matrices for all other genera of root system
        by conjugation of $C$;
        \item the ring of integers of the field of definition, 
        a generator of the connection index and 
        a generator of the ideal $\Bad_G$.
        \end{itemize}
        A complete legend for the table is on page~\pageref{table}.
        
        In \GAP,
        all  vectors are row vectors, matrices  operate from the right, and
        in \CHEVIE\ the pairing used is not Hermitian. Consequently, to
        change from \CHEVIE\ conventions to our conventions, one has to
        transpose and conjugate  the list of coroots.  Thus the Cartan matrices
        given  in  the  tables of Appendix~\ref{table}  are  the transpose of 
        one obtained by applying the conventions of the preceding sections.

        For each primitive irreducible reflection group $G$, \CHEVIE\ contains a
        Cartan matrix $C$ which satisfies the assumptions of Proposition
        \ref{ClassificationByCartan}. Thus there exists a (principal) root
        system for $G$ since $C$ satisfies the following set of properties:
\begin{enumerate}
        \item \label{CMProps}
                $C$  is the Cartan matrix of a ordered set of distinguished
                roots  $\fR_0$ for $G$, whose  root lines generate $V$ and
                such  that the corresponding  reflections generate $G$. The
                ordered  set of  reflections corresponding  to $\fR_0$ are
                called the ``standard'' generators of $G$.
        \item
                 The entries of $C$ are elements of $\BZ_k$.
\footnote{
        Actually,  the data in \CHEVIE\ did not always  verify (2) when there
         were several $G$-orbits of distinguished reflections, but we have found it
        possible  to adjust the data by multiplying by a global scalar all roots in
         one  of the orbits in  order to satisfy (2).  These modifications have now
         been incorporated in the \CHEVIE\ database.}
        \item
                For each triple $(\alpha_i,\beta_i,\zeta_i)$
                defining  an element of $\fR_0$, its orbit under $G$ is finite,  
                and whenever two  such triples define the
                same   reflection  they  differ   by  the  action   of  an  element  of
                $\BZ_k^\times$.
\end{enumerate}
 
        It follows from the above properties that the root system $\fR$ determined by 
        $C$ is actually distinguished. Indeed:

\begin{itemize}
        \item 
                all roots in $\fR$ are distinguished,
        \item 
                there is only one root in $\fR$ attached to each distinguished reflection in $G$.
\end{itemize}

        It  follows from  Lemma~\ref{lem:conjugationbydiagonal} that
        any  other  distinguished  root  system  $\fR'$  for  the
        standard  generators taken in the same order corresponds to conjugating
        $C$  by a diagonal matrix. So the  first step to classify genera of
        root  systems  is  to  determine  whether  $C$ can be modified by a
        diagonal matrix so that the entries remain integral.
\smallskip

\subsection{Cases with only one genus of distinguished root systems}\hfill
\smallskip

        By Proposition \ref{nebegen}, if $G$ has a single orbit of
        distinguished reflections, it has a single genus of distinguished root
        systems.
         This proves uniformly the uniqueness (provided they exist) of 
         distinguished root systems 
        for 19 of the 34 primitive irreducible reflection groups, that is
        the groups $G_n$ where $n\in\{4,8,12,16,20,22,23,24,25,27,
          29,30,31,32,\hfill\break 33,34,35,36,37\}$.
        The existence of a distinguished root system for these groups is shown
        by the \CHEVIE\ root data and gives the Cartan matrices 
        (with integral entries) in Appendix \ref{table}.

\begin{proposition}\label{typical}
        For  each  of  the groups $G_9,G_{10},G_{11},G_{14},G_{17},G_{18},$
        $G_{19},G_{21}$  there  is  a  unique  genus  of distinguished root
        systems;   indeed,  the   matrix  corresponding   to  the  standard
        generators  as given in \CHEVIE\ is the unique one in its class
        modulo conjugation by diagonal matrices with integral entries.
\end{proposition}

\begin{proof}
        We  first show how  the proof goes  for $G_{11}$. The Cartan matrix
        for the standard generators given in \CHEVIE\ for $G_{11}$ is:
        $$
        \begin{pmatrix}
                2&\zeta_3^2(\sqrt{-3}-\sqrt{-2})&1-\zeta_{24}\\
                1&1-\zeta_3&\zeta_{24}^{13}\\
                1+\zeta_{24}^5&\zeta_{24}^7+\zeta_{24}^5&1-i\\
        \end{pmatrix}.
        $$

        When we conjugate by the matrix $\diag(1,a,b)$
        with $a,b\in k$ we get
        $$
        \begin{pmatrix}
                2&a\zeta_3^2(\sqrt{-3}-\sqrt{-2})&b(1-\zeta_{24})\\
                a\inv&1-\zeta_3&a\inv b\zeta_{24}^{13}\\
                b\inv(1+\zeta_{24}^5)&ab\inv(\zeta_{24}^7+\zeta_{24}^5)&1-i\\
        \end{pmatrix}.
        $$

        Knowing that $\sqrt{-3}-\sqrt{-2},1-\zeta_{24}$ and
        $1+\zeta_{24}^5$ are units (which is easily checked in \CHEVIE\
        if  not obvious), we see that for the above matrix to have integral
        entries  $a,a\inv,b,b\inv$ must  be integral,  which forces $a$ and
        $b$ to be units and $\diag(1,a,b)$ to be a matrix of units, thus to
        preserve  the  genus.  
        
        A  similar reasoning, using the fact that particular entries of the
        Cartan  matrix are  units, also  applies to  the other cases of the
        proposition,   whose   Cartan   matrices   are  given  in  Appendix
        \ref{table}.  To help the  reader apply the  above argument, 
        the entries of the Cartan matrices which are units are in bold
        in  Appendix  \ref{table},  and  the  same
        convention is applied for the matrices  given in the  next subsection. 
\end{proof}

\subsection{The cases with more than one genus}\habel{sec:atyp}
\smallskip

        The   approach  here  follows  the  same  lines  as  the  proof  of
        Proposition~\ref{typical}.  The Cartan matrix cited in each case is
        the  one given in  \CHEVIE\ satisfying  the conditions (1), (2)
        and  (3) of  page~\pageref{CMProps}, and is also the one given in
        Appendix \ref{table}.
\smallskip

\subsubsection*{The case of $G_5$}\hfill
\smallskip

        The Cartan matrix $C$ is:
        $$
        \begin{pmatrix}
                1-\zeta_3&\unit{1}\\
                -2\zeta_3&1-\zeta_3\\
        \end{pmatrix}.
        $$
        A matrix $\diag(1,a)$ conjugates $C$ to an integral matrix if and only
        if $a$ is an integral divisor of $2$. Since $2$ is prime in
        $\BZ[\zeta_3]$, this gives rise to two distinct genera of distinguished
        \zroot\ systems.
\smallskip

\subsubsection*{The case of $G_6$}\hfill
\smallskip

        The Cartan matrix $C$ is:
        $$
        \begin{pmatrix}
                2&\frac{(3+\sqrt 3)(\zeta_3-1)}3\\
                \unit{-1}&1-\zeta_3\\
        \end{pmatrix}.
        $$
        A  matrix $\diag(1,a)$ conjugates $C$ to  an integral matrix if and
        only if $a$ is an integral divisor of $\frac{(3+\sqrt
        3)(\zeta_3-1)}3$,  which  up  to  a  unit  is  equal  to  $i+1$. In
        $\BZ_k=\BZ[\zeta_{12}]$,  the  ideal  $\ideal{(i+1)}$  is  prime of
        square  $\ideal  2$.  This  gives  rise  to  2  distinct  genera of
        distinguished \zroot\ systems.
\smallskip

\subsubsection*{The case of $G_7$}\hfill
\smallskip

        The Cartan matrix $C$ is:
        $$
        \begin{pmatrix}
                        2&\zeta_3^2(1-i)&-\zeta_3^2(i+1)\\
                \unit{\zeta_3(1-i\zeta_3)}&1-\zeta_3&-\zeta_3(1-i)\\
                \unit{i\zeta_3(1-i\zeta_3)}&i+1&1-\zeta_3\\
        \end{pmatrix}.
        $$
        A matrix $\diag(1,a,b)$ conjugates $C$ to an integral matrix if and
        only  if both $a$ and $b$ are  integral divisors of $i+1$, which is
        prime in $\BZ[\zeta_{12}]$. This gives rise to 4 distinct genera of
        distinguished \zroot\ systems.

        The group $G_7$ has an outer automorphism induced by an element of $\GL(V)$,
        induced by the embedding  of  reflection groups $G_7\subset G_{15}$,
        where $[G_{15}:G_7]=2$. This automorphism exchanges the conjugacy
        classes of the reflections corresponding to rows 2 and 3 of $C$, thus
        exchanges two of the genera of the root systems and leaves the other{\red s}
        fixed. Thus we get one more root system than \cite{nebe}, which counts
        the systems up to isomorphism.
\smallskip

\subsubsection*{The case of $G_{13}$}\hfill
\smallskip

        The Cartan matrix $C$ is:
        $$
        \begin{pmatrix}
                2&\sqrt 2&i-1\\
                \unit{1+\sqrt 2}&2&-1+\sqrt {-2}\\
                \unit{-\zeta_8(1+\sqrt 2})&
                -1-\sqrt {-2}&2\\
        \end{pmatrix}.
        $$
        Now  $i-1=\zeta_8^3  \sqrt  2$,  and  $\sqrt  2  =\zeta_8^3(1+\sqrt
        2)(1-\zeta_8)^2$  where $1+\sqrt  2$ is  a unit.  Hence in terms of
        ideals, in $\BZ_k=\BZ[\zeta_8]$ we have $\ideal{\sqrt
        2}=\ideal{(1-i)}=(\ideal{(1-\zeta_8)})^2$ (see Corollary \ref{Im'n}
        of Appendix \ref{arithmetic}).

        A matrix $\diag(1,a,b)$ conjugates $C$ to an integral matrix if and
        only  if $a$ and $b$ are equal up  to a unit, and both are integral
        divisors  of $\sqrt  2$. This  gives rise  to 3  distinct genera of
        distinguished  \zroot\ systems,  corresponding respectively to
        the   values   $\BZ_k$,   $\ideal{(1-\zeta_8)}$  and  $\ideal{\sqrt
        2}=\ideal{(1-i)}$ for $\ideal a$.
\smallskip

\subsubsection*{The case of $G_{15}$}\hfill
\smallskip

        The Cartan matrix $C$ is:
        $$
        \begin{pmatrix}
                2&\unit{-\zeta_{24}(1-\zeta_{24}^{19})}&\unit{1}\\
                \unit{1-\zeta_{24}\inv}&1-\zeta_3&\unit{1}\\
                \unit{\zeta_8^3}(1-\zeta_8)^2&\unit{u}(1-\zeta_8)^2&2\\
        \end{pmatrix}
        $$
        where $\unit u=(1+\sqrt 2)(\zeta_3+i)$ is a unit.

        A matrix $\diag(1,a,b)$ conjugates $C$ to an integral matrix if and
        only  if  $a$  is  a  unit  and  $b$  is  an  integral  divisor  of
        $(1-\zeta_8)^2$,  which  is  a  square  in $\BZ_k=\BZ[\zeta_{24}]$.
        Hence  there are  3 distinct  genera of  distinguished \zroot\
        systems.
\smallskip

\subsubsection*{The case of $G_{26}$}\hfill
\smallskip

        The Cartan matrix $C$ is:
        $$
        \begin{pmatrix}
                2&\unit{-1}&0\\
                \zeta_3-1&1-\zeta_3&\unit{\zeta_3^2}\\
                0&\unit{-\zeta_3^2}&1-\zeta_3\\
        \end{pmatrix}.
        $$
        A matrix $\diag(1,a,b)$ conjugates $C$ to an integral matrix if and
        only  if $a$ and $b$  are equal up to  a unit and both are integral
        divisors  of $1-\zeta_3$, which is prime in $\BZ_k=\BZ[\zeta_3]$. So there are 2
        distinct genera of distinguished \zroot\ systems.
\smallskip

\subsubsection*{The case of $G_{28}=F_4$}\hfill
\smallskip

        The Cartan matrix $C$ is:
        $$
        \begin{pmatrix}
                2&\unit{-1}&0&0\\
                \unit{-1}&2&\unit{-1}&0\\
                0&-2&2&\unit{-1}\\
                0&0&\unit{-1}&2\\
        \end{pmatrix}.
        $$

        A  matrix $\diag(1,a,b,c)$ conjugates $C$  to an integral matrix if
        and  only if $a$ is a  unit and $b$ and $c$  are equal up to a unit
        and  are integral divisors of $2$. Since $2$ is prime in $\BZ$ this
        leaves 2 distinct genera of distinguished \zroot\ systems.
\bigskip

\newpage

\appendix\section{On roots of unity}\label{arithmetic}
\smallskip

\subsection{Notation and summary of known properties}\hfill
\smallskip

        For any natural integer $n$, we denote by
        \begin{itemize}
                \item
                        $\vp(n)$ the order of the multiplicative
                        group $(\BZ/n\BZ)^\times$,
                \item
                        $\Phi_n(X)$ the $n$-th cyclotomic polynomial,
                        monic element of $\BZ[X]$ inductively 
                        defined by the equality
                        $$
                                X^n-1 = \prod_{d\mid n} \Phi_d(X)
                                \,,
                        $$
                \item
                        $\BZ_n$ (resp. $\BQ_n$) the ring
                        (resp. the field) generated
                        by the group $\bmu_n$ of all $n$-th roots of unity.             
        \end{itemize}
        The following omnibus proposition states properties which 
        are either well known, or easy to establish.
\smallskip

\begin{proposition}\habel{prop:omnibuscyclo}

\begin{enumerate}
        \item
                $
                        \Phi_n(X) = \prod_{\text{$\zeta$ of order $n$}}
                        (X-\zeta)
                        \,,
                $
                and $\deg \Phi_n(X) = \vp(n)$.
        \item
                If
                $
                        n = \prod_{p\in\CP(n)} p^{v_p(n)}
                $
                (where $\CP(n)$ denotes the set of prime numbers 
                dividing $n$),
                then
                $$
                        \vp(n) = \prod_{p\in\CP(n)} p^{v_p(n)-1}(p-1)
                        \,.
                $$
        \item   
                $
                        n = \prod_{(d\mid n)(d\neq 1)} \Phi_d(1)
                        \,.
                $
        \item
                We say that an integer $n$ is \emph{composite} if it is divisible by at
                least two different prime numbers. Then
                $$
                        \Phi_n(1) =
                         \prod_{\text{$\zeta$ of order $n$}} (1-\zeta) =
                        \left\{
                        \aligned
                                &1\quad\text{ if } n \text{ is composite,} \\
                                &p\quad\text{ if } n= p^m 
                                        \text{ ($p$ a prime number)}.
                        \endaligned
                        \right.
                $$
                \item
                (Changing $X$ to $-X$)
                        $$
                        \begin{aligned}
                                \Phi_d(-X) &=\Phi_{2d}(X) \,
                                \text{ for $d$ odd and } d>1 \,, \\ 
                                \Phi_d(-X) &= \Phi_d(X) \,
                                \text{ for $d$ divisible by 4.}
                        \end{aligned}
                        $$
        \item
                The polynomial $\Phi_n(X)$ is irreducible in $\BZ[X]$, 
                hence
                $$
                        \BQ_n \simeq \BQ[X]/\Phi_n(X)
                        \quad\text{and}\quad
                        \BZ_n \simeq \BZ[X]/\Phi_n(X)
                        \,.
                $$
        \item
                Let $a$ and $b$ be relatively prime.
                Then
                $
                        \vp(ab) = \vp(a)\vp(b)
                        \,,
                $
                 so:
                \begin{enumerate}
                        \item
                                $[\BQ_{ab}:\BQ_a] = \vp(b)\,,$
                        \item
                                $\Phi_b(X)$ is irreducible in $\BQ_a[X]$,
                                hence in $\BZ_a[X]\,,$
                        \item
                                $\BZ_{ab} \simeq \BZ_a[X]/\Phi_b(X)\,.$
                \end{enumerate}
        \item
                Whenever $p$ is a prime number,
                $$
                        \left\{
                        \begin{aligned}
                                \Phi_n(X^p)& =
                                  \Phi_n(X)\Phi_{pn}(X) \text{ if } p\nmid n \,,\\
                                \Phi_n(X^p)&=\Phi_{pn}(X) \text{ if } p\mid n \,.
                        \end{aligned}
                        \right.
                $$
                In particular
                $$
                        \Phi_{p^a}(X) = \Phi_p(X^{p^{a-1}})
                        \,.
                $$
\end{enumerate}
\end{proposition}

\begin{corollary}\label{1+zeta}
If $m$ divides $n$, then $(1+\zeta_m)\BZ_n\ne\BZ_n$ precisely when 
$m=2p^k$, $p$ prime, $k\ge 0$.
\end{corollary}
\begin{proof}
By \ref{prop:omnibuscyclo}(4), $1+\zeta_m$ is not a unit
if $-\zeta_m$ is a prime power.
\end{proof}
\subsection{Decomposition of the ideal $\fI_{m,n}$ in $\BZ_n$}\hfill
\smallskip
        
        Let $m \in \BN$ divide $n$. If $\zeta$ and $\xi$ are two roots of 
        unity, both of order $m$, then the elements $\zeta$ and $\xi$ 
        generate the same multiplicative group, and hence each of them is
        a power of the other. 
        Thus the elements $(1-\zeta)$ and $(1-\xi)$ are multiples of
        one  another in the ring $\BZ_n$, and so 
        generate 
        the same principal ideal of $\BZ_n$:
        $$
                \fI_{m,n} := (1-\zeta)\BZ_n = (1-\xi)\BZ_n.
        $$
       
\index{Imn@$\fI_{m,n}$}
\begin{lemma}\label{lem:mcompositeornot}
        Let $m$ and $n$ be natural integers such that $m$ divides $n$.
        Then
        $$
        \left\{
        \aligned
                &\fI_{m,n} = \BZ_n \quad\text{ if } m \text{ is composite,} \\
                &\fI_{m,n}^{\vp(m)} = p\BZ_n \quad\text{ if } m= p^a 
                        \text{ ($p$ a prime number)}.
        \endaligned
        \right.
        $$
\end{lemma}

\begin{proof}
        This is an immediate consequence of the definition of $\fI_{m,n}$
        and of item (4) of \ref{prop:omnibuscyclo}. 
\end{proof}

        We now investigate the decomposition of $\fI_{m,n}$ 
        into a product of prime ideals of $\BZ_n$. 
        By the preceding lemma, $\fI_{m,n}$ is invertible if $m$ is composite. 
        So from now on we assume that $m = p^a$ where $p$ is a 
        prime number and $a\geq 1$. 
        The general result is provided by Proposition~\ref{zmn} below;
        the next lemma 
        gives the result in the particular case where $n/m$ is prime to $p$.
        
\begin{lemma}\label{lem:m=n}
        Let $m = p^a$ and $n = p^an'$ be natural integers,
        where $n'$ is prime to $p$.
\begin{enumerate}
        \item
                If $n' = 1$ (\ie\ $m =n$), the ideal $\fI_{m,m}$ is maximal
                in $\BZ_m$. More precisely,
                $$
                                \BZ_{p^a}/\fI_{p^a,p^a}  \simeq \BF_p\,.
                $$
        \item
                Let $r$ be the order of $p$ in the multiplicative group 
                $(\BZ/n'\BZ)^\times$ and let $d := \vp(n')/r$. Then
                $$
                        \fI_{m,n} = \fp_1\cdots\fp_d\,,
                $$
                where $\fp_1,\cdots,\fp_d$ are the maximal ideals in $\BZ_n$ 
                such that $\BZ_n/\fp_i$ is a finite field 
                with $p^r$ elements (for $i = 1,2,\dots,d$).
\end{enumerate}
\end{lemma}

\begin{proof}

        (1)
        The ideal $\fI_{m,m}$ is maximal, since 
        (by item (4) of Proposition  \ref{prop:omnibuscyclo}) we have
        $$
                \BZ_m/\fI_{m,m} = \BZ[X]/\genby{\Phi_m(X), 1-X}
                = \BZ/\Phi_m(1) = \BZ/p\BZ
                \,.
        $$

        (2)
        In $\BF_p[X]$, $\Phi_{n'}(X)$ splits into irreducible polynomials
        of degree $r$, thus there are $d$ of them.
        Thus the ring $\BF_p[X]/\Phi_{n'}[X]$ is isomorphic to 
        $\BF_{p^r}\times\cdots\times\BF_{p^r}$ ($d$ factors).
        
        By \ref{prop:omnibuscyclo} 7(c) and the proof of (1),  we have
        $$
                \BZ_{p^{a}n'}/\fI_{p^{a},p^{a}n'} 
                = \BZ_{p^{a}}[X]/\genby{\Phi_{n'}[X],\fI_{p^a,p^a}}
                \simeq \BF_p[X]/\Phi_{n'}[X]
                \,,
        $$
        hence
        $$
                \fI_{p^{a},p^{a}n'} = \fp_1\dots\fp_d
        $$
        where $\fp_1,\dots,\fp_d$ are the maximal ideals of $\BZ_n$
        such that
        $\BZ_n/\fp_i$ is a field with $p^r$ elements.
\end{proof}

\begin{proposition}\label{zmn}
        Assume $m = p^a$ for some prime number $p$, and
        $n = p^{a+h} n'$ where
        $n'$ is an integer not divisible by $p$.
        We denote by $r$ the multiplicative order 
        of $p$ modulo $n'$ and we set $d := \vp(n')/r$.
        Then
        $$
                \fI_{m,n} = (\fp_1\cdots\fp_d)^{p^h}
                \,,
        $$
        where $\fp_1,\cdots,\fp_d$ are $d$ maximal ideals in $\BZ_n$ 
        such that $\BZ_n/\fp_i$ is a finite field 
        with $p^r$ elements (for $i = 1,2,\dots,d$).
\end{proposition}
        
\begin{proof}
        Proposition~\ref{prop:omnibuscyclo}(2) implies
        $
                \vp(p^{a+h}) = p^h\vp(m)
                .
        $
        By Lemma~\ref{lem:mcompositeornot},
        $$
                p\BZ_{p^{a+h}} = \fI_{p^{a+h},p^{a+h}}^{p^h\vp(m)} 
                = \fI_{m,p^{a+h}}^{\vp(m)}
                \,.
        $$
        Since $\fI_{p^{a+h},p^{a+h}}$ is a maximal ideal (Lemma~\ref{lem:m=n}(1)), 
        it follows from
        the uniqueness of the decomposition of an ideal
        into a product of prime ideals in the Dedekind domain
        $\BZ_n$ that
        $$
                \fI_{m,p^{a+h}} = \fI_{p^{a+h},p^{a+h}}^{p^h}
                \,,
        $$
        which implies
        $$
                \fI_{m,n} = \fI_{p^{a+h},n}^{p^h}
                \,.
        $$
        Applying item (2) of Lemma \ref{lem:m=n}, we get
        $$
                \fI_{m,n} = (\fp_1\cdots\fp_d)^{p^h}
                \,.
        $$
\end{proof}

\begin{example}\label{example nonprincipal}
        Take $p = 13$ , $a = 1$, $h = 0$,  $n'= 3$, hence $n = 39$, $r = 1$, $d = 2$.
        Then (see proof of item (2) of \ref{lem:m=n}):
        $$
                \BZ_{39}/\fI_{13,39} = \BF_{13}[X]/\Phi_3(X)
                \,.
        $$      
        The decomposition of $\Phi_3(X)$ in $\BF_{13}[X]$ is 
        $
                \Phi_3(X) = (X-3)(X-9)
                \,.
        $
        Define two ideals of $\BZ_{39}$ as follows:
        $$
        \begin{aligned}
                &\fp_1 := \fI_{13,39} +  (\zeta_3-3)\BZ_{39} \,,\\
                &\fp_2 := \fI_{13,39} +  (\zeta_3-9)\BZ_{39}
                \,.
        \end{aligned}
        $$
        Then $\fp_1$ and $\fp_2$ are  distinct maximal ideals of $\BZ_{39}$
        such that
        $
                \BZ_{39}/\fp_1 \simeq \BZ_{39}/\fp_2 \simeq \BF_{13}
                \,,
        $
        and we have
        $$
                \fI_{13,39} = \fp_1\fp_2
                \,.
        $$
        It can be checked, for example with the PARI-GP command:
        \begin{verbatim}
        f=bnfinit(polcyclo(39)); 
        bnfisprincipal(f,idealprimedec(f,13)[1])
        \end{verbatim}
        that the ideals $\fp_1$ and $\fp_2$ are not principal ideals.
\end{example}

\begin{corollary}\label{Im'n}
        Assume that $m' \mid m \mid n$ and that $m$ is a power of
        a prime $p$. Then
        $$
                \fI_{m',n} = \fI_{m,n}^{m/m'}
                \,.
        $$
\end{corollary}

\begin{proof}
        This is an immediate consequence of the above Proposition~\ref{zmn},
        since the integer $r$ and the maximal ideals $\fp_1,\ldots,\fp_d$
        depend only on the pair $(n,p)$.
\end{proof}

\begin{corollary}\label{zeta divides}
        Let $m$ be a natural integer and let $\zeta$ be a root of unity
        of order $m$. Let $n$ be a multiple of $m$.
\begin{enumerate}
        \item
                For any natural integer $m'>1$ dividing $m$, 
                $(1-\zeta)(1-\zeta\inv)$
                divides $m'$ in $\BZ_n$.
        \item
                If $m\notin\{2,3\}$, every prime factor of $m$
                divides $\dfrac{m}{(1-\zeta)(1-\zeta\inv)}$
                in $\BZ_n$.
\end{enumerate}
\end{corollary}

\begin{proof}\hfill

        (1)
        We first notice that if $m=m'$ the result follows from the equality
        $\prod_{i=1}^{m-1}(1-\zeta^i)=m$. We may thus assume $m>m'$,
        and we do so.
        If $m$ is composite, by Lemma~\ref{lem:mcompositeornot},
        $\fI_{m,n}=\BZ_n$ and there
        is nothing to prove. Otherwise, by
        Corollary~\ref{Im'n}, we have 
        $\fI_{m',n}=\fI_{m,n}^{m/m'}$; so in particular 
        $\fI_{m',n}$, 
        which divides $m'\BZ_n$, 
        is divisible by 
        $\fI_{m,n}^2 = \left[(1-\zeta) \BZ_n \right] \left[(1-\zeta\inv) \BZ_n \right] $.
   
        (2)
        The  statement is equivalent to saying that any prime
        factor  of  $m$  divides  $m\BZ_n/\fI_{m,n}^2$. 
        By Lemma~\ref{lem:mcompositeornot},
        $\fI_{m,n}=\BZ_n$  if $m$ is composite. 
        Thus we need only consider the case
        where $m$ is a prime power. 
        Since 
        $$
                m\BZ_n = {{\prod_{i=1}^{m-1}(1-\zeta^i)}}\BZ_n
                        = \prod_{i=1}^{m-1}\fI_{m/\gcd(i,m)} \,,
        $$
        then $m\BZ_n$  is
        divisible  by  $\fI_m$  to  the power at least $\varphi(m)$. 
        For $m \not \in \{2,3,4,6\}$, $\varphi(m)>2$ and the assertion holds.
        The cases $m=2$ and $3$ are excluded by the hypothesis, 
        and $m=6$ is composite.
        By the formula above for $m=4$ we have
        $4\BZ_n =\fI_{4,n}^2\fI_{2,n}$, and by
        Corollary~\ref{Im'n}, we have $\fI_{2,n}=\fI_{4,n}^2$, completing the proof.
\end{proof}
\smallskip

\subsection{On cyclotomic fields}\hfill
\smallskip

\subsubsection*{On real subfields of cyclotomic fields}\hfill
\smallskip

\begin{proposition}\label{dihedralfields}
        Let $e>1$ be an integer. We set $\zeta_e := \exp(2\pi i/e)$.
        \begin{enumerate}
                \item
                        The field $\BQ[\cos(2\pi/e)]$ is the largest real subfield
                        of the cyclotomic field $\BQ(\zeta_e)$, \ie\
                        $\BQ[\cos(2\pi/e)] = \BQ(\zeta_e) \cap \BR\,.$
                \item
                        If $e$ is odd, $\BQ(\zeta_e) = \BQ(\zeta_{2e})$
                        and $\BQ[\cos(2\pi/e)] = \BQ[\cos(\pi/e)]$. \\                  
                        If $e$ is even, $\BQ(\zeta_e) \subsetneq \BQ(\zeta_{2e})$
                        and $\BQ[\cos(2\pi/e)] \subsetneq \BQ[\cos(\pi/e)]\,.$
        \end{enumerate}
\end{proposition}

\begin{proof}\hfill

        (1) Consider the sequence of polynomials $P_n(T) \in \BZ[T]$
        inductively defined by $P_0(T) = 2$, $P_1(T) = T$, and
        $P_{n+1}(T) = TP_n(T) - P_{n-1}(T)$. Then, for all $n \geq 1$,
        $$
                X^n+X^{-n} = P_n(X+X\inv)
                \,.
        $$
        Since $\zeta_e+\zeta_e\inv = 2\cos(2\pi/e)$, it is clear that
        $\BQ[\cos(2\pi/e)] \subset \BR \cap \BQ(\zeta_e)$. Conversely,
        any real element of $\BQ(\zeta_e)$ is a $\BQ$-linear combination of
        elements of the form $\zeta^j +\zeta^{-j}$, hence is a 
        $\BQ$-linear combination
        of $P_j(\zeta + \zeta\inv)$, proving that it belongs to $\BQ[\cos(2\pi/e)]$.
        
        (2) If $e$ is odd, $\zeta_{2e} = -\zeta_e^{(e+1/2)}$, hence 
        $\BQ(\zeta_e) = \BQ(\zeta_{2e})$, and by part (1) this implies
        $\BQ[\cos(2\pi/e)] = \BQ[\cos(\pi/e)]$.
        
        If $e$ is even, $\vp(2e) = 2 \vp(e)$, and so 
        $\BQ(\zeta_e) \subsetneq \BQ(\zeta_{2e})$. Moreover,
        $i \in \BQ(\zeta_{2e})$, hence an element
        $x + i y$ (with $x,y \in \BR$) belongs to $\BQ(\zeta_{2e})$
        if and only if $x$ and $y$ belong to $\BQ(\zeta_{2e})$.
        This shows that
        $\BQ(\zeta_{2e}) = \BQ[\cos(\pi/e)] \oplus i \BQ[\cos(\pi/e)]$,
        hence that
        $[\BQ[\cos(\pi/e)]:\BQ] = \vp(2e)/2 = \vp(e)$. Since
        $\BQ[\cos(2\pi/e)] \subsetneq \BQ(\zeta_e)$, this shows that
        $\BQ[\cos(2\pi/e)] \subsetneq \BQ[\cos(\pi/e)]\,.$
\end{proof}

        By \cite[Theorem 2.6. and Proposition 2.16]{was}, 

\begin{proposition}\habel{integerscyclotomic}
\begin{enumerate}
        \item
                The ring of integers of $\BQ(\zeta_e)$ is $\BZ[\zeta_e]$.
        \item
                The ring of integers of $\BQ(\zeta_e+\zeta_e\inv)$ is 
                $\BZ[\zeta_e+\zeta_e\inv]$.
\end{enumerate}
\end{proposition}

        The ring $\BZ[\zeta_e+\zeta_e\inv]$ is not necessarily a P.I.D. 
        For example (see \cite[Theorem 1.1]{mi}),
        $\BZ[\zeta_{163}+\zeta_{163}\inv]$ has class number 4.
        Nevertheless (\ibidem),

\begin{theorem}
        Let $p$ be a prime number. Then the ring $\BZ[\zeta_p+\zeta_p\inv]$
        is a P.I.D. if $p\leq 151$.
\end{theorem}

        The next lemma is used in section 7 above.
        
\begin{lemma}\label{zeta+zetainv prime}
        Let $\zeta$ be a root of unity of order $p^a$, 
        where $p$ is a prime and $a \geq 1$ is a integer.
        Then the principal ideal of $\BZ[\zeta+\zeta\inv]$ 
        generated by $(1-\zeta)(1-\zeta\inv)=(2-\zeta-\zeta\inv)$ is prime.
\end{lemma}

\begin{proof}
        It suffices to prove that the norm
        $
                N_{\BQ[\zeta+\zeta\inv]/\BQ}(2-\zeta-\zeta\inv)
        $
        of $(2-\zeta-\zeta\inv)$ is equal to $p$.
        By item (4) of Proposition~\ref{prop:omnibuscyclo}, we see that
        $$
                N_{\BQ[\zeta]/\BQ}(1-\zeta) = N_{\BQ[\zeta]/\BQ}(1-\zeta\inv) = p
                \,.
        $$
        Since 
        $$
                N_{\BQ[\zeta]/\BQ[\zeta+\zeta\inv]}(1-\zeta) =
                N_{\BQ[\zeta]/\BQ[\zeta+\zeta\inv]}(1-\zeta\inv) = 
                (1-\zeta)(1-\zeta\inv)
                \,,
        $$
        the assertion follows from the fact that
        $$
                N_{\BQ[\zeta]/\BQ} = N_{\BQ[\zeta+\zeta\inv]/\BQ} \cdot
                N_{\BQ[\zeta]/\BQ[\zeta+\zeta\inv]}
                \,.
        $$      
\end{proof}
\smallskip

\subsubsection*{Further properties of cyclotomic fields}\hfill
\smallskip

        The following property may be found, for example, 
        in \cite[\S 6.5]{sam}.

\begin{proposition}\label{quadincycl}
        Let $p$ be an odd prime. Then the cyclotomic field 
        $\BQ(\zeta_p)$ contains a single
        quadratic extension, namely $\BQ(\sqrt{(-1)^{(p-1)/2}p})$.
\end{proposition}

\begin{corollary}\label{quadincycl2}
        Let $p$ be an odd prime. Then
 \begin{itemize}
        \item 
                $\BQ[\sqrt p]\subset\BQ(\zeta_p)$ if $p\equiv 1\pmod 4$.
        \item 
                $\BQ[\sqrt{-p}]\subset\BQ(\zeta_p)$ if $p\equiv 3\pmod 4$.
 \end{itemize}
        In the second case we have $\BQ[\sqrt{-p}]\subset\BQ(\zeta_{4p})$.
\end{corollary}

        Let us recall that the fields $\BQ[\sqrt{-p}]$ for $p\equiv 3\pmod 4$ are
        principal ideal domains if and only if $p\in \{3,7,11,19,43,67,163\}$.

        The next result shows that a cyclotomic extension cannot contain 
        the roots of some rational numbers.

\begin{proposition}\label{rootsincyc}
        A cyclotomic field $\BQ(\zeta_n)$ cannot contain an element $\alpha$ 
        whose minimal   polynomial is $X^m-a$ for $a\in \BQ$ and $m\geq 3$.
\end{proposition}

\begin{proof}[Sketch of proof]
        If $\BQ(\zeta_n)$ contains such an element $\alpha$, it also contains 
        the Galois closure of the extension $\BQ(\alpha)/\BQ$, namely 
        $\BQ(\alpha,\zeta_m)$. This is impossible since the Galois group of 
        $\BQ(\alpha,\zeta_m)/\BQ$ is not abelian.
\end{proof}

        The next theorem maybe found in \cite[Chap. 11, Theorem 11.1]{was} .
        
\begin{theorem}\label{washington}
        A cyclotomic field $\BQ(\zeta_m)$ is a principal ideal domain if 
        and only if
        $m\le 22$ or $m\in\{24,25,26,27,28,30,32,33,34,35,36,38,
        \allowbreak 40,42,44,45,48,50,54,60,66,70,84,90\}$.
        \end{theorem}

\vfill\eject
\section{A table of Cartan matrices}
The following table gives, for each exceptional irreducible finite complex
reflection group $G_4,\ldots,G_{34}$:
\begin{itemize}
\item  The diagram describing its  presentation, using the same conventions
as  in  \cite[Appendix  A]{berkeley};  a  subset of nodes representing each
conjugacy   class  of  hyperplanes  have  been  given  labels  in  the  set
$\{s,t,u\}$.
\item  The Cartan matrix  $C$ of a principal  distinguished \zroot\ system
for  the standard  generators corresponding  to the  diagram. If  there are
several  genera  of  \zroot\  systems,  
below  $C$ we list 
the diagonal  matrices which conjugate $C$ to the Cartan matrix for other genera
are listed below $C$.

For badly generated groups, the rows (resp. columns) of the Cartan
matrix which are integral linear combinations of the others are indicated by
$*$. That there is always at least one such row and column fulfills
the promise of Remark \ref{rem:connectionbadgenerated}.
\item The value of $\BZ_k$.
The values of $\BZ_k$ for $G_{14}, G_{20}, G_{21}, G_{22}$ and $G_{27}$
can be readily seen using Exercise 4.5.13 in \cite{mues}.
\item A generator of the connection index for each root system.
\item A generator of the (principal) ideal $\Bad_G$.
\end{itemize}

For   the  omitted   exceptional  groups   $G_{35}=E_6$,  $G_{36}=E_7$  and
$G_{37}=E_8$,  the diagram and the Cartan matrix are well known.
We have
$\BZ_k=\BZ$ in each case, connection indices are
$$c_{E_6}=3, c_{E_7}=2\text{ and }c_{E_8}=1,$$
and the numbers $\Bad_G$ are:
$$\Bad_{E_6}=6, \Bad_{E_7}=6\text{ and }\Bad_{E_8}=120.$$

Finally, in the tables, we use the notation $\phi=\frac{1+\sqrt 5}2$.

\newpage
\label{table}
\font\hugecirc = lcircle10 scaled 1950
\font\largecirc = lcircle10 scaled  \magstep 3
\font\mediumcirc = lcircle10 scaled \magstep 2
\font\smallcirc = lcircle10 
\newcommand\BBcirc{$\hbox{\hugecirc\char"6E}$}
\newcommand\Bcirc{$\hbox{\largecirc\char"6E}$}
\newcommand\sBcirc{$\hbox{\mediumcirc\char"6E}$}
\newcommand\sbcirc{$\hbox{\smallcirc\char"6E}$}
\newcommand\ucirc{$\hbox{\mediumcirc\char"13}$}
\newcommand\lcirc{$\hbox{\mediumcirc\char"12}$}
\newcommand\cnode[1]{{\kern -0.4pt\bigcirc\kern-7pt{\scriptstyle#1}\kern 2.6pt}}
\newcommand\ncnode[2]{{\kern -0.4pt\mathop\bigcirc\limits_{#2}\kern-8.6pt{\scriptstyle#1}\kern 2.3pt}}
\newcommand\node{{\kern -0.4pt\bigcirc\kern -1pt}}
\newcommand\snode{{\kern -0.4pt{\scriptstyle\bigcirc}\kern -1pt}}
\newcommand\nnode[1]{{\kern -0.6pt\mathop\bigcirc\limits_{#1}\kern -1pt}}
\def\bar#1pt{{\vrule width#1pt height3pt depth-2pt}}
\def\dbar#1pt{{\rlap{\vrule width#1pt height2pt depth-1pt} 
                 \vrule width#1pt height4pt depth-3pt}}
\def\tbar#1pt{{\rlap{\vrule width#1pt height1pt depth0pt}
           \rlap{\vrule width#1pt height3pt depth-2pt}
               \vrule width#1pt height5pt depth-4pt}}
\newcommand\overmark[2]{\kern -1.5pt\mathop{#2}\limits^{#1}\kern -2pt}
\newcommand\trianglerel[3]{
 \nnode#1\bar14pt\kern-12pt\raise7.5pt\hbox{$\displaystyle \underleftarrow 6$}
  \kern 2pt\nnode#2
 \kern-27pt\raise9pt\hbox{$\diagup$}
 \kern -2pt
 \raise16.5pt\hbox{$\node$\rlap{\raise 2pt\hbox{$\kern 1pt\scriptstyle #3$}}}
 \kern -1pt\raise9pt\hbox{$\diagdown$}
 \kern 4pt
}
\newcommand\vertbar[2]{\rlap{$\nnode{#1}$}
                 \rlap{\kern4pt\vrule width1pt height17.3pt depth-7.3pt}
           \raise19.4pt\hbox{$\node$\rlap{$\kern 1pt\scriptstyle#2$}}}

\newcommand\diagg[1]{
\ifnum #1 = 4 \ncnode3s \bar20pt\cnode3\fi
\ifnum #1 = 5 \ncnode3s\dbar20pt\ncnode3t\fi
\ifnum #1 = 6 \nnode s\tbar20pt\ncnode3t\fi
\ifnum #1 = 7
 {\scriptstyle s}\node\kern13.5pt\raise2.5pt\hbox{$\Bcirc$}
  \kern-15.5pt\rlap{\raise10pt\hbox{$\cnode3\kern 1pt\scriptstyle t$}}
  \lower10pt\hbox{$\cnode3\kern 1pt\scriptstyle u$}
\fi
\ifnum #1 = 8  \ncnode4s \bar20pt\cnode 4\fi
\ifnum #1 = 9  \nnode s\tbar20pt\ncnode4t\fi
\ifnum #1 = 10 \ncnode3s\dbar20pt\ncnode 4t\fi
\ifnum #1 = 11
 {\scriptstyle s}\node\kern13.5pt\raise2.5pt\hbox{$\Bcirc$}
  \kern-15.5pt\rlap{\raise10pt\hbox{$\cnode3\kern 1pt\scriptstyle t$}}
  \lower10pt\hbox{$\cnode 4\kern 1pt\scriptstyle u$}
\fi
\ifnum #1 = 12
 {\scriptstyle s}\node\kern13.5pt\raise2.5pt\hbox{$\Bcirc$}
 \kern-25.7pt\raise2.5pt\hbox{$\sBcirc$}
 \kern-11.8pt\rlap{\raise10pt\hbox{$\node$}}\lower10pt\hbox{$\node$}
\fi
\ifnum #1 = 13
 \kern 15pt\raise2.5pt\hbox{$\BBcirc$}\kern-49.5pt
 {\scriptstyle s}\kern 2pt\snode\kern8pt\raise2.5pt\hbox{$\sbcirc$}
 \kern-11.8pt\rlap{\raise9.5pt\hbox{$\snode\kern 3pt\scriptstyle t$}}
                 \lower8pt\hbox{$\snode$}
 \kern-12pt{\lower 3pt\hbox{$\scriptstyle 5$}}
 \kern8pt{\scriptstyle 4}
\fi
\ifnum #1 = 14 \nnode s\overmark 8{\bar20pt}\ncnode3t\fi
\ifnum #1 = 15 {\scriptstyle s}\node \kern 0pt
               \rlap{\kern -4pt \lower 10pt\hbox{$\scriptstyle 5$}}
            \raise 16.6pt\hbox{$\ucirc$}
               \kern -28.8pt\lower 12.2pt\hbox{$\lcirc$}
            \kern -28.8pt \phantom{\dbar 14 pt}
    \kern-2.6pt
    \raise10pt\hbox{$\node$\rlap{\hbox{$\kern 2pt\scriptstyle t$}}}
    \kern-8.8pt
    \lower10.2pt\hbox{$\cnode 3$\rlap{\hbox{$\kern 2pt\scriptstyle u$}}}
\fi
\ifnum #1 = 16 \ncnode 5s \bar20pt\cnode 5\fi
\ifnum #1 = 17 \nnode s\tbar20pt\ncnode 5t\fi
\ifnum #1 = 18 \ncnode 3s\dbar20pt\ncnode 5t\fi
\ifnum #1 = 19
 {\scriptstyle s}\node\kern13.5pt\raise2.5pt\hbox{$\Bcirc$}
 \kern-15.5pt\rlap{\raise10pt\hbox{$\cnode 3\kern 1pt\scriptstyle t$}}
 \lower10pt\hbox{$\cnode 5\kern 1pt\scriptstyle u$}
\fi
\ifnum #1 = 20 \ncnode3s\overmark 5{\bar20pt}\cnode3\fi
\ifnum #1 = 21 \nnode s\overmark{10}{\bar20pt}\ncnode 3t\fi
\ifnum #1 = 22 {\scriptstyle s}\node\kern13.5pt \raise2.5pt\hbox{$\Bcirc$}
  \kern-28pt\rlap{\lower.4pt\hbox{$5$}}\kern 28pt
  \kern-15.5pt\rlap{\raise10pt\hbox{$\node$}}
  \lower10pt\hbox{$\node$}
\fi
\ifnum #1 = 23 \nnode s\overmark 5{\bar16pt}\node\bar16pt\node\fi
\ifnum #1 = 24 
 \nnode s \dbar14pt\node
 \kern-26.2pt\raise8.5pt\hbox{$\diagup$}
 \kern -3.1pt
 \raise16.7pt\hbox{$\node$}
 \kern -2.5pt\raise8pt\hbox{$\diagdown$}
 \kern 4pt
 \fi
\ifnum #1 = 25 \ncnode3s\bar16pt\cnode3\bar16pt\cnode3 \fi
\ifnum #1 = 26 \nnode s\dbar16pt\ncnode3t\bar16pt\cnode3\fi
\ifnum #1 = 27
 \nnode s \dbar14pt\node
 \kern-27pt\raise9pt\hbox{$\diagup$}\kern -2pt
 \raise16.5pt\hbox{$\node$}\kern -1pt\raise9pt\hbox{$\diagdown$}\kern 4pt
 \fi
\ifnum #1 = 28 \nnode s\bar10pt\node\dbar10pt\nnode u\bar10pt\node\fi
\ifnum #1 = 29 \nnode s\bar10pt
 \node\rlap{\raise9pt\hbox{$\underleftarrow{}$}}
        \rlap{\kern 5.5pt\vrule width1pt height15pt depth-10pt
      \kern 1pt \vrule width1pt height15pt depth-10pt}
      \dbar14pt\node
 \kern-26pt\raise9pt\hbox{$\diagup$}
 \kern -3pt\raise17pt\hbox{$\node$}
 \kern -2pt\raise9pt\hbox{$\diagdown$}
 \kern 5pt
 \fi
\ifnum #1 = 30 \nnode s\overmark 5{\bar10pt}\node\bar10pt\node\bar10pt\node\fi
\ifnum #1 = 31
 \node\kern -2pt\raise8pt\hbox{$\diagup$}
      \kern -2.5pt\raise16pt\hbox{$\nnode s$}
      \kern-12.5pt\bar 19pt\node
   \kern-4pt\raise18.3pt\hbox{$\sBcirc$}
   \kern-17.6pt\bar19pt\node
   \kern-20.5pt\raise16pt\hbox{$\node$}
   \kern-2pt\raise8pt\hbox{$\diagdown$}
\fi
\ifnum #1 = 32 \ncnode3s\bar10pt\cnode3\bar10pt\cnode3\bar10pt\cnode3\fi
\ifnum #1 = 33 \nnode s\bar10pt\trianglerel tuw\bar10pt\node\fi
\ifnum #1 = 34 \nnode s\bar10pt\trianglerel tuw\bar10pt\node\bar10pt\node\fi
\ifnum #1 = 35 \nnode{s_1}\bar10pt\nnode{s_3}\bar10pt\vertbar{s_4}{s_2}\bar10pt\nnode{s_5}\bar10pt\nnode{s_6}\fi
\ifnum #1 = 36 \nnode{s_1}\bar10pt\nnode{s_3}\bar10pt\vertbar{s_4}{s_2}\bar10pt\nnode{s_5}\bar10pt\nnode{s_6}\bar10pt\nnode{s_7}\fi
\ifnum #1 = 37 \nnode{s_1}\bar10pt\nnode{s_3}\bar10pt\vertbar{s_4}{s_2}\bar10pt\nnode{s_5}\bar10pt\nnode{s_6}\bar10pt\nnode{s_7}\bar10pt\nnode{s_8}\fi
}

\def\vinsert#1pt{height#1pt&\omit&\omit&\omit&\omit&\omit&\omit&\cr}
\newcommand\tbl[1]{
\offinterlineskip
\halign{
\tabskip 0pt
\vrule##
\tabskip 0.2 cm
&\vrule height11.4pt depth11pt width0pt
$##$\hfil&
\hfil$##$\hfil&
$##$\hfil&
\hfil$##$\hfil&
\hfil$##$\hfil&
\hfil$##$\hfil&
\tabskip 0pt\vrule##\cr
\noalign{\hrule}
&\hbox{Name}&\hbox{Diagram}&\hbox{Cartan matrices}&\BZ_k&c_G&\Bad_G&\cr
\noalign{\hrule}
#1
\noalign{\hrule}
}
\vfill\eject
}
\renewcommand{\kbldelim}{(}
\renewcommand{\kbrdelim}{)}
{\footnotesize
\tbl{
\vinsert2pt
&G_4&\diagg{4}&
\begin{pmatrix}1-\zeta_3&\unit{\zeta_3^2}\\\unit{-\zeta_3^2}&1-\zeta_3
\end{pmatrix}
&\BZ[\zeta_3]&2&2\sqrt{-3}&\cr
\vinsert10pt
&G_5&\diagg{5}&
\begin{pmatrix}1-\zeta_3&\unit{1}\\-2\zeta_3&1-\zeta_3\end{pmatrix}
&\BZ[\zeta_3]&1&&\cr
&&&\diag(1,2)&&&&\cr
&G_6&\diagg{6}&
\begin{pmatrix}2&\unit{(i+\zeta_3)}(1+i)\\ \unit{-1}&1-\zeta_3\end{pmatrix}
&\BZ[\zeta_{12}]&1+i&4\sqrt 3&\cr
&&&\diag(i+1,1)&&&&\cr
&G_7&\diagg{7}&
\kbordermatrix{
&*&&\cr
*&2&\zeta_3^2(1-i)&-\zeta_3^2(1+i)\cr
*&\unit{\zeta_3(1+i\zeta_3)}&1-\zeta_3&-\zeta_3(1-i)\cr
*&\unit{i\zeta_3(1+i\zeta_3)}&1+i&1-\zeta_3\cr}
&\BZ[\zeta_{12}]&1&&\cr
\vinsert4pt
&&&\diag(i+1,1,i+1)&&1&&\cr
&&&\diag(i+1,i+1,1)&&1&&\cr
&&&\diag(i+1,1,1)&&1&&\cr
&G_8&\diagg{8}&
\begin{pmatrix}1-i&-i\\ 1&1-i\\\end{pmatrix}
&\BZ[i]&1&12&\cr
\vinsert10pt
&G_9&\diagg{9}&
\begin{pmatrix}2& \unit{-\frac{2+\sqrt 2}{i+1}}\\ \unit{-1}&1-i\end{pmatrix}
&\BZ[\zeta_8]&1&&\cr
\vinsert10pt
&G_{10}&\diagg{10}&
\begin{pmatrix}1-\zeta_3& \unit{1}\\ \unit{-i-\zeta_3}&1-i\end{pmatrix}
&\BZ[\zeta_{12}]&1&&\cr
\vinsert10pt
&G_{11}&\diagg{11}&
\kbordermatrix{&*&*&*\cr
*&2&\unit{\zeta_3^2(\sqrt{-3}-\sqrt{-2})}&\unit{1-\zeta_{24}}\cr
*&\unit{1}&1-\zeta_3&\unit{\zeta_{24}^{13}}\cr
*&\unit{1+\zeta_{24}^5}&\unit{\zeta_{24}^7+\zeta_{24}^5}&1-i}
&\BZ[\zeta_{24}]&1&&\cr
\vinsert10pt
&G_{12}&\diagg{12}&
\kbordermatrix{&*&*&*\cr
*&2&-1-\sqrt{-2}&-1+\sqrt{-2}\cr
*&-1+\sqrt{-2}&2&-1-\sqrt{-2}\cr
*&-1-\sqrt{-2}&-1+\sqrt{-2}&2}
&\BZ[\sqrt{-2}]&1&&\cr
\vinsert10pt
&G_{13}&\diagg{13}&
\kbordermatrix{&*&*&*\cr
*&2&\sqrt 2& i-1\cr 
&\unit{(1+\sqrt 2)}&2&-1+\sqrt{-2}\cr
&\unit{-\zeta_8 (1+\sqrt 2)}&-1-\sqrt {-2}&2}
&\BZ[\zeta_8]&1&&\cr
\vinsert2pt
&&&
\diag(1-\zeta_8,1,1)&&1&&\cr 
\vinsert2pt
&&&\diag(i+1,1,1)&&1&&\cr
}
\hskip-2cm
\tbl{
\vinsert10pt
&G_{14}&\diagg{14}&
\begin{pmatrix}2&\unit{-1}\\ \unit{-\zeta_3^2(\sqrt{-3}+\sqrt {-2})}&1-\zeta_3\end{pmatrix}
&\BZ[\zeta_3,{\scriptstyle\sqrt{-2}}]&1&12\sqrt{-2}&\cr
\vinsert2pt
&G_{15}&
\diagg{15}
&
\kbordermatrix{&*&*&*\cr
&2&\unit{-\zeta_{24}-\zeta_{24}^8}&\unit{1}\cr
&\unit{1-\zeta_{24}\inv}&1-\zeta_3&\unit{1}\cr
*&\unit{\zeta_8^3}(1-\zeta_8)^2&\unit{-\zeta_8(\zeta_3+i)}\sqrt2&2
}
&\BZ[\zeta_{24}]&1&&\cr
\vinsert2pt
&&&\diag(1,1,1-\zeta_8)&&1&&\cr
\vinsert2pt
&&&
\diag(1,1,1+i)
&&1&&\cr
\vinsert9pt
&G_{16}&\diagg{16}&
\begin{pmatrix}1-\zeta_5& \unit{1}\\ \unit{-\zeta_5}&1-\zeta_5\end{pmatrix}
&\BZ[\zeta_5]&1&&\cr
\vinsert9pt
&G_{17}&\diagg{17}&
\begin{pmatrix}2& \unit{1}\\ \unit{1-\zeta_5-\zeta_{20}^7}&1-\zeta_5\end{pmatrix}
&\BZ[\zeta_{20}]&1&&\cr
\vinsert9pt
&G_{18}&\diagg{18}&
\begin{pmatrix}1-\zeta_3& \unit{-\zeta_{15}^4}\\
\unit{\zeta_{15}\inv+\zeta_{15}}&1-\zeta_5\end{pmatrix}
&\BZ[\zeta_{15}]&1&&\cr
\vinsert9pt
&G_{19}&\diagg{19}&
\kbordermatrix{&*&*&*\cr
*&2&\unit{1-\zeta_3\zeta_{60}\inv}&\unit{-1}\cr
*&\unit{1+\zeta_{60}}&1-\zeta_3&\unit{-\zeta_{60}}\cr
*&\unit{\zeta_5-1+\zeta_{20}^7}&\unit{\zeta_{60}^{19}+\zeta_{60}^{11}}&1-\zeta_5
}
&\BZ[\zeta_{60}]&1&&\cr
\vinsert9pt
&G_{20}&\diagg{20}&
\begin{pmatrix}1-\zeta_3& \unit{-\zeta_3\phi}\\
\unit\phi&1-\zeta_3\end{pmatrix}
&\BZ[\zeta_3,\phi]&1&&\cr
\vinsert9pt
&G_{21}&\diagg{21}&
\begin{pmatrix}2& \unit{1}\\ \unit{1-\zeta_3+\frac\phi{\zeta_{12}}}&1-\zeta_3\end{pmatrix}
&\BZ[\zeta_{12},\phi]&1&&\cr
&G_{22}&\diagg{22}&
\kbordermatrix{&*&*&*\cr
*&2&-i-\phi&-i+\phi\cr
*&i-\phi&2&-1+i\phi\cr
*&i+\phi&-1-i\phi&2}
&\BZ[i,\phi]&1&&\cr
\vinsert10pt
&{\displaystyle G_{23}\atop\displaystyle =H_3}&\diagg{23}&
\begin{pmatrix}2& \unit{-\phi}&0\\
\unit{-\phi}&2&\unit{-1}\\0&\unit{-1}&2\end{pmatrix}$$
&\BZ[\phi]&2&2\sqrt5&\cr
\vinsert10pt
&G_{24}&\diagg{24}&
\begin{pmatrix}2&\unit{-1}&\unit{-1}\\ \unit{-1}&2&(1-\sqrt{-7})/2\\ \unit{-1}&(1+\sqrt{-7})/2&2\end{pmatrix}
 &\BZ[\frac{1+\sqrt{-7}}2]&1&2\sqrt{-7}&\cr
\vinsert10pt
&G_{25}&\diagg{25}&
\begin{pmatrix}1-\zeta_3& \unit{\zeta_3^2}&0\\ \unit{-\zeta_3^2}&1-\zeta_3&\unit{\zeta_3^2}\\
0&\unit{-\zeta_3^2}&1-\zeta_3\end{pmatrix}
&\BZ[\zeta_3]&\sqrt{-3}&6&\cr
\vinsert10pt
}
\tbl{
\vinsert2pt
&G_{26}&\diagg{26}&
\begin{pmatrix}2&\unit{-1}&0\\ \zeta_3-1&1-\zeta_3&\unit{\zeta_3^2}\\
0&\unit{-\zeta_3^2}&1-\zeta_3\end{pmatrix}
&\BZ[\zeta_3]&1&18\sqrt{-3}&\cr
\vinsert4pt
&&&\diag(1,1-\zeta_3,1-\zeta_3)&&&&\cr
\vinsert10pt
&G_{27}&\diagg{27}&
\begin{pmatrix}2&\unit{-1}&\unit{-1}\\ \unit{-1}&2&\frac{\zeta_3^2(\sqrt{-3}+\sqrt {5})}2\\
\unit{-1}&\frac{(\sqrt{5}-\sqrt{-3})\zeta_3}2&2\end{pmatrix}
&\BZ[\zeta_3,\phi]&1&6\sqrt5&\cr
\vinsert10pt
&{\displaystyle G_{28}\atop\displaystyle =F_4}&\diagg{28}&
\begin{pmatrix}
2&\unit{-1}&0&0\\
\unit{-1}&2&\unit{-1}&0\\
0&-2&2&\unit{-1}\\
0&0&\unit{-1}&2\\
\end{pmatrix}
&\BZ&1&24&\cr
\vinsert4pt
&&&\diag(1,1,2,2)&&&&\cr
\vinsert9pt
&G_{29}&\diagg{29}&
\begin{pmatrix}2&\unit{-1}&0&0\\\unit{-1}&2&i+1&\unit{-1}\\0&1-i&2&\unit{-1}\\0&\unit{-1}&\unit{-1}&2\end{pmatrix}
&\BZ[i]&1&20&\cr
\vinsert9pt
&{\displaystyle G_{30}\atop\displaystyle =H_4}&\diagg{30}&
\begin{pmatrix}2&\unit{-\phi}&0&0\\\unit{-\phi}&2&\unit{-1}&0\\
0&\unit{-1}&2&\unit{-1}\\0&0&\unit{-1}&2\end{pmatrix}
&\BZ[\phi]&1&120&\cr
\vinsert10pt
&G_{31}&\diagg{31}&
\kbordermatrix{&*&*&*&&\cr
*&2&i+1&1-i&\unit{-i}&0\\
*&1-i&2&1-i&\unit{-1}&\unit{-1}\\
*&i+1&i+1&2&0&\unit{-1}\\
&\unit{i}&\unit{-1}&0&2&0\\
&0&\unit{-1}&\unit{-1}&0&2}
&\BZ[i]&1&&\cr
\vinsert6pt
&G_{32}&\diagg{32}&
\begin{pmatrix}1-\zeta_3&\unit{\zeta_3^2}&0&0\\ \unit{-\zeta_3^2}&1-\zeta_3& \unit{\zeta_3^2}&0\\
 0&\unit{-\zeta_3^2}&1-\zeta_3&\unit{\zeta_3^2}\\0&0&\unit{-\zeta_3^2}&1-\zeta_3\end{pmatrix}
&\BZ[\zeta_3]&1&120\sqrt{-3}&\cr
\vinsert6pt
&G_{33}&\diagg{33}&
\begin{pmatrix}2&\unit{-1}&0&0&0\\\unit{-1}&2&\unit{-1}&\unit{-\zeta_3^2}&0\\0&\unit{-1}&2&\unit{-1}&0\\
0&\unit{-\zeta_3}&\unit{-1}&2&\unit{-1}\\0&0&0&\unit{-1}&2\end{pmatrix}
&\BZ[\zeta_3]&2&6&\cr
\vinsert6pt
&G_{34}&\diagg{34}&
\begin{pmatrix}2&\unit{-1}&0&0&0&0\\\unit{-1}&2&\unit{-1}&\unit{-\zeta_3^2}&0&0\\0&\unit{-1}&2&\unit{-1}&0&0\\
0&\unit{-\zeta_3}&\unit{-1}&2&\unit{-1}&0\\0&0&0&\unit{-1}&2&\unit{-1}\\0&0&0&0&\unit{-1}&2\end{pmatrix}
&\BZ[\zeta_3]&1&42&\cr
\vinsert2pt
}
}

\vskip1cm

\printindex
\vskip1cm

\end{document}